\documentclass[envcountsect, runningheads]{llncs}
\usepackage[T1]{fontenc}
\usepackage{graphicx}
\usepackage{color}
\usepackage{etoolbox}

\usepackage{amsmath,mathtools,mathrsfs,amssymb,stmaryrd,wasysym,enumitem}
\usepackage[dvipsnames]{xcolor}
\usepackage{tikz}
\usetikzlibrary{calc, shapes} 
\usepackage{standalone}
\usepackage[justification=centering]{caption}
\usepackage{subcaption}
\usepackage[linesnumbered,ruled,vlined]{algorithm2e}
\SetKwProg{Fn}{Function}{:}{}
\usepackage{comment}

\newcommand{\todo}[1]{{\color{red}#1}}

\newcommand{\pow}{\mathcal{P}}
\newcommand{\fillstep}{\rightarrowtriangle}
\newcommand{\cyc}[1]{\langle #1 \rangle}

\newcommand\xqed[1]{%
  \leavevmode\unskip\penalty9999 \hbox{}\nobreak\hfill
  \quad\hbox{#1}}
\newcommand\qee{\xqed{\fullmoon}}

\newcommand\Lshape[2]{
\draw[#1] #2 -- ++(-1, 0) -- ++(0, 1) -- ++(2,0) -- ++(0, -2) -- ++(-1, 0) -- cycle
}
\newcommand{\move}[3]{
\draw[->, >=stealth, color = #1, very thick] #2 -- #3
}
\newcommand{\step}[2]{
\draw[->, >=stealth, thick, #1] #2-- ++(1,0)
}

\newcommand{\trigrid}[1]{
    \foreach \k in {1, ..., #1}
    \draw (\k-1,#1-\k) grid (#1,#1-\k+1);
    }
    
\newcommand{\TriangShape}[2]{
    \draw[#2] (0, #1) -- (#1, #1) -- (#1, 0);
    \foreach \k in {1, ..., #1}
    \draw[#2] (#1 - \k +1, \k -1) -- (#1 - \k, \k -1) -- (#1 - \k, \k);
}

\newcommand\cross[2]{
\draw[#1] #2 -- ++(0, -1) -- ++(1, 0) -- ++(0,1) -- ++(1, 0) -- ++(0,1) -- ++(-1, 0) --++ (0,1) --++ (-1, 0)--++(0,-1)--++ (-1,0) --++ (0,-1)-- cycle
}

\tikzstyle{here}=[ellipse,draw, fill=black, text=white, inner sep = 0.01cm, font = \large]
\tikzstyle{empty}=[ellipse,draw, fill=white, text=black, inner sep = 0.01cm, font = \large]
\tikzstyle{round}=[circle,draw,text=black, inner sep = 0.01cm, minimum width = 0.5cm]

\newcommand{\cylgrid}[2]{
\draw (0,0) grid (#1, #2);
    \foreach \k in {0, ..., #1}{
    \draw[dashed] (\k-0.5,0) -- (\k-0.5, #2);
    \draw[dashed] (\k,#2) arc (0:180:0.25);
    \draw[dashed] (\k,0) arc (360:180:0.25);};
    \foreach \l in {0, ..., #2} \draw[dashed] (#1, \l) --++ (0.5, 0);
}

\newcommand{\Tshape}[3]{\draw[#1] (#2-0.25, #3-1.4) -- (#2-0.25, #3+1.4) -- (#2+1.4, #3) -- cycle;}

\definecolor{paletteA}{HTML}{000000}
\definecolor{paletteB}{HTML}{000000}
\definecolor{points}{HTML}{FFFFFF}

\newcommand{\caleyfree}[7]{
\ifnumless{0}{#1}{%
  \pgfmathtruncatemacro\newlev{#1-1}
  \pgfmathtruncatemacro\len{#3}
  \ifnumless{1}{#1}{%
    \draw[draw=palette#2,very thick] (2 pt,0) -- (\len-2 pt,0) node[midway,auto,text=palette#2]{\footnotesize #4};
  }{%
    \draw[draw=palette#2,very thick] (2 pt,0) -- (\len-2 pt,0);
  }
  \node at (\len pt,0) (O) {};
  \begin{scope}[shift={(O)}]
    \ifstrequal{#2}{E}{%
      \begin{scope}[rotate=90] \caleyfree{\newlev}{A}{\len/2}{#5}{#6}{#7}{#4} \end{scope}
      \begin{scope}[rotate=0]  \caleyfree{\newlev}{B}{\len/2}{#4}{#5}{#6}{#7} \end{scope}
      \begin{scope}[rotate=-90]\caleyfree{\newlev}{A}{\len/2}{#7}{#4}{#5}{#6} \end{scope}
    }{%
      \begin{scope}[rotate=90] \caleyfree{\newlev}{B}{\len/2}{#5}{#6}{#7}{#4} \end{scope}
      \begin{scope}[rotate=0]  \caleyfree{\newlev}{A}{\len/2}{#4}{#5}{#6}{#7} \end{scope}
      \begin{scope}[rotate=-90]\caleyfree{\newlev}{B}{\len/2}{#7}{#4}{#5}{#6} \end{scope}
    }
  \end{scope}}{\fill[color = points] circle(2pt);}}

\definecolor{paletteGA}{HTML}{942C00}
\definecolor{paletteGB}{HTML}{004C99}

\newcommand{\caleygrid}[2]{
  \begin{scope}[rotate=-90] \caleyfree{#1}{B}{#2}{\large$b^{-1}$}{\large$a$}{\large$b$}{\large$a^{-1}$} \end{scope}
  \begin{scope}[rotate=0]   \caleyfree{#1}{A}{#2}{\large$a$}{\large$b$}{\large$a^{-1}$}{\large$b^{-1}$} \end{scope}
  \begin{scope}[rotate=90]  \caleyfree{#1}{B}{#2}{\large$b$}{\large$a^{-1}$}{\large$b^{-1}$}{\large$a$} \end{scope}
  \begin{scope}[rotate=180] \caleyfree{#1}{A}{#2}{\large$a^{-1}$}{\large$b^{-1}$}{\large$a$}{\large$b$} \end{scope}
}

\newcommand{\Z}{\mathbb{Z}}
\newcommand{\C}{\mathcal{C}}
\newcommand{\TI}{\mathrm{TI}}
\newcommand{\I}{\mathrm{I}}
\newcommand{\N}{\mathbb{N}}
\newcommand{\R}{\mathbb{R}}
\newcommand{\orb}{\mathcal{O}}
\newtheorem{alg}[algocf]{Algorithm}
\newcommand{\spacc}{\vspace{-0.1cm}}

\newcommand{\rank}{\mathrm{rank}}

\begin{document}
\title{Solitaire of Independence}

\author{Ville Salo\inst{1}\orcidID{0000-0002-2059-194X} \and
Juliette Schabanel\inst{2}\orcidID{0009-0002-3647-4892}}
\authorrunning{V. Salo \and J. Schabanel}
\institute{University of Turku, Turku, Finland \and
École Normale Supérieure PSL, Paris, France}
\maketitle

\begin{abstract}
In this paper, we study a reversible process (more precisely, a groupoid/group action) resembling the classical 15-puzzle, where the legal moves are to ``move the unique hole inside a translate of a shape $S$''. Such a process can be defined for any finite subset $S$ of a group, and we refer to such a process as simply ``solitaire''. We develop a general theory of solitaire, and then concentrate on the simplest possible example, solitaire for the plane $\Z^2$, and $S$ the triangle shape (equivalently, any three-element set in general position). In this case, we give a polynomial time algorithm that puts any finite subset of the plane in normal form using solitaire moves, and show that the solitaire orbit of a line of consecutive ones -- the line orbit -- is completely characterised by the notion of a so-called fill matrix. We show that the diameter of the line orbit, as a graph with edges the solitaire moves, is cubic. We show that analogous results hold for the square shape, but indicate some shapes (still on the group $\Z^2$) where this is less immediate. We then explain in detail the connection of the solitaire to TEP and more generally permutive subshifts. Namely, the solitaire is a closure property of various sets of subsets of the group that can be associated to such a subshift, such as the independence, spanning and filling sets.
\end{abstract}

\section{Introduction}

In this paper, we introduce the \emph{solitaire of independence}, which is a game played on subsets of a group, alternatively an abstract rewriting system, groupoid or group action. Specifically, if $G$ is a group, and $S \Subset G$ is finite, then on subsets of $G$ we can play a variant of the classical ``15-puzzle'' by, for each $g \in G$, allowing a move from $X \subset G$ to $Y \subset G$ if for some $g \in G$,  $X$ and $Y$ have the same elements outside $gS$, and intersect $gS$ in distinct sets of size $|S| - 1$. Or in words, we take $G$ as a playing grid and consider sets of marbles (subsets of $G$) on it which we can move as follows: if there is a translate $gS$ of the shape $S \Subset G$ such that exactly one cell inside it has no marble, then any marble inside can be moved to the hole.  

For example, if $G = \Z^2$, and we consider the triangle shape $T = \{(0,1), (0, 0), (1, 0)\}$, then the action consists in arbitrarily permuting the three subpatterns depicted in Figure \ref{fig:triangle_action} at some coordinates where one of them appears in the pattern.
\begin{figure}[ht!]
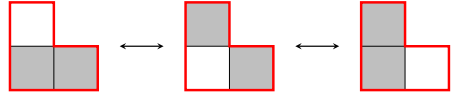

    \centering
    \includestandalone[height = 1.2cm]{Figures/Triangle/fig_triang_action}
    \caption{The action of the triangle shape. Grey denotes $1$ (or a marble), white denotes $0$ (no marble).}
    \label{fig:triangle_action}
\end{figure}

A simple illustration of a valid sequence of moves for this shape on subsets of the set $\{(a, b) \;|\; a, b \geqslant 0, a+b \leqslant 3\} \subset \Z^2$ is shown in Figure~\ref{fig:valid_moves}.

On the website \cite{SaSolitaire}, one can play the solitaire on $\Z^2$ with any small shape.

\begin{figure}[ht!]
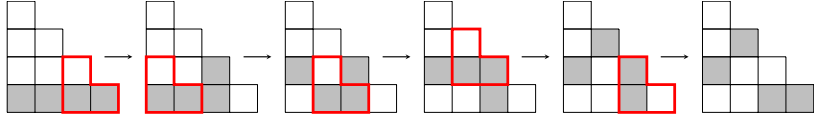

    \centering
    \includestandalone[height = 1.5cm]{Figures/Triangle/example_moves}
    \caption{A simple sequence of valid moves of triangle solitaire starting from the line $L_4 = [0, 3]\times\{0\}$.}
    \label{fig:valid_moves}
\end{figure}

In the present paper, we develop a basic theory of solitaire processes, and illustrate it with a detailed look at the triangle shape, for which we can give a full description of the family of sets reachable from a given set by playing the solitaire. In particular, for this example our main theorem is the following:

\begin{theorem}
For the triangle solitaire process, for any finite set $P \Subset \Z^2$, there is a sequence of solitaire moves with cubic length in $|P|$ which puts $P$ in normal form, and such a sequence can be found in polynomial time.
\end{theorem}

By normal form, we mean that we can find a canonical representative of the orbit, such that we can compare orbits by comparing the representatives. The normal form we use is a disjoint (and ``non-touching'') union of lines with additional elements lined up on top.

We have implemented this algorithm, and it is available in \cite{SolitaireGitHub}.

The general theory of solitaire is tightly connected with another process called the filling process described in Section~\ref{sec:Filling}, which is a simpler (confluent) process where we inductively construct the largest set that could possibly be involved in the solitaire game, by filling in the missing holes in sets $gS$ (i.e.\ if $|gS \cap P| = |S|-1$, then $gS$ may be added to $P$).

Due to confluence, the filling process gives rise to a closure operator, and our primary interest is in understanding the connection between patterns that fill a particular set, and orbits in the solitaire process. Of specific interest is the case where in fact all sets that are minimal in cardinality, and fill a given set, are in the same solitaire orbit. We refer to this case as a \emph{nice solitaire theory}. We do not know, at present, general conditions under which solitaire processes are nice. However, we obtain that the triangle shape, as well as the square shape $\{0,1\}^2$, have nice solitaire theories on the group $\Z^2$.

We also explain the connection with \emph{TEP subshifts} from \cite{Sa22}, which are a generalization of spacetime subshifts of bipermutive cellular automata, and more generally \emph{permutive\footnote{Some authors write ``permutative''.} subshifts} (which also generalize the polygonal subshifts of \cite{FrKr20}).

In particular, to any subshift we associate its ``independent sets'' (sets that can have arbitrary contents) and the sets ``spanning'' (uniquely determining) the contents of a given area. These are meant as analogs of notions from linear algebra and matroid theory with the same names. There are also the ``filling'' sets, which are the sets that are spanning ``by applying the filling process''. We show that the solitaire process preserves these sets in the case of TEP and permutive subshifts. (The fact it acts on the independent sets gives the process its full name ``solitaire of independence''.)

The present paper is an extended version of the conference paper \cite{SaSc23} where we presented the results about the triangle shape in isolation. The result here about the diameter of orbits is slightly sharper.

In Section~\ref{sec:Def} we state our notations and remind the reader of some basic definitions. In Section~\ref{sec:General}, we define the solitaire game and the filling process for general groups and pattern, and give a few general properties on them. Section~\ref{sec:Excess} introduces different notions to measure the difference between a pattern and a minimal pattern having the same filling, and study what this difference means for orbits. Sections~\ref{sec:Triangle}, \ref{sec:Plane} and \ref{sec:Free} study some particular case of solitaires, respectively the triangle shape on $\Z^2$, general convex shapes on $\Z^2$ and more particularly the square; and convex shapes on the free group, with a deeper study of the triangle shapes.
In Section~\ref{sec:Algo}, we study the algorithmic aspects of convex solitaire on the plane, giving algorithms and their complexity to recognise the orbit of a pattern or find a sequence of move from one pattern to another. Finally, in Section~\ref{sec:Permutive}, we link the solitaire back to the TEP-subshift theory it originates from, detailing the correspondence between the two and the consequences of one theory on the other.

\section{Conventions and basic definitions}
\label{sec:Def}

For the most part, we give definitions as we go, but we list some general conventions here, and also collect some definitions that are used throughout.

We have $0 \in \N$. An \emph{alphabet} is just a finite set $A$ whose elements are called \emph{letters} or \emph{symbols}. Groups $G$ are always countable, with the main interest being in finitely-generated groups. 

By $\subset$ we mean $\subseteq$ and by $\Subset$ we mean "is a finite subset of". If $X$ is a set, we denote by $\mathcal{P}(X)$ the set of subsets of $X$ and for $k$ an integer, $\mathcal{P}_k(X)$ the set of subsets of size $k$. If $X$ and $Y$ are sets, we denote by $X\triangle Y$ their symmetric difference $X\triangle Y = (X\cup Y)\setminus (X\cap Y)$. The set of functions from $A$ to $B$ is denoted $B^A$.

For a function $f :A \to B$ and $C \subset A$, write $f|_C$ for the restriction of $f$ to $C$. For functions that are thought of as ``patterns'' (see below), we also write $f_c$ for $f(c)$. If $F$ is a family of functions (or patterns), then $F|_A = \{f|_A \;|\; f\in F\}$. 

We use some conventions from formal language theory and combinatorics on words, in particular if $A$ is a finite alphabet, $A^*$ denotes the words on the alphabet (or elements of the free monoid), and concatenation of $u, v \in A^*$ is written simply as $uv$.

We use standard big-O notation: for $f, g : \N \to \N$, we have $g \in O(f)$ or ``$g$ is $O(f)$'' (resp.\ $\Omega(f)$) if $g(n) \leqslant Mf(n)$ (resp.\ $\geqslant$) for large $n$ and a constant $M$. We have $\Theta(f) = O(f) \cap \Omega(f)$.

If $X$ is a set, a \emph{closure operator} is a map $\tau : \pow(X) \to \pow(X)$ which is \emph{extensive} meaning $\forall A \subset X: A \subset \tau(A)$, \emph{monotone} meaning $\forall A, B \subset X: A \subset B \implies \tau(A) \subset \tau(B)$, and \emph{idempotent} meaning $\forall A \subset X: \tau(\tau(A)) = \tau(A)$.

The fixed points or \emph{closed sets} of a closure operator form a \emph{closure system}, i.e.\ $X$ is closed and an arbitrary intersection of closed sets is closed.

A set of sets $\mathcal{S} \subset \pow(X)$ is a \emph{down set} if $A \in \mathcal{S}$ and $B \subset A$ imply $B \in \mathcal{S}$. \emph{Up sets} are defined symmetrically.

The \emph{support} of a permutation is the set of elements that it moves. Note that this set is closed under the permutation.

If $G$ is a group and $g \in G$, we denote by $\langle g \rangle$ the (cyclic) group generated by $g$, $\langle g \rangle = \{g^k \mid k \in \Z\} \subset G$. The \emph{order} of $g$ is the cardinality of this group. 

If $A$ is an alphabet, $G$ a group a \emph{pattern} is $p : P \to A$ with $P \subset G$ (not necessarily finite). When playing the solitaire, we also talk about patterns as subsets of $G$ (subsets are of course in bijection with patterns over the alphabet $A = \{0,1\}$). If $a \in A$ and $g \in G$, write $a^g$ for the unique pattern $p$ of type $p : \{g\} \to \{a\}$, 

If $B, D \subset G$ for a group $G$, we say a pattern $p \in A^B$ \emph{appears} in another pattern $q \in A^D$ if there exists $g \in G$ such that $gB \subset D$ and $q(gb) = p(b)$ for all $b \in B$.

\section{The general theory of solitaire}
\label{sec:General}


\subsection{Tilings and (TEP)-subshifts}

Let $A$ be a finite alphabet and $G$ a group. Elements of $A$ are called \emph{symbols} and those of $G$ \emph{cells}. A configuration is a mapping $\gamma : G \to A$. When $A = \{0, 1\}$, configurations are equivalent to subsets of $G$ with the conversion $\gamma \leftrightarrow \{g \in G \mid \gamma(g) = 1\}$. In this case, cells assigned the value $1$ may be referred as \emph{points} and those assigned to $0$ as \emph{empty} cells. 

The \emph{full shift} on $G$ with alphabet $A$ is the set of configurations $A^G$ with the product topology and the $G$-action given by $g\gamma_h = \gamma_{g^{-1}h}$. A \emph{subshift} is a closed $G$-invariant subset of $A^G$. A subshift is of \emph{finite type}, shortened as \emph{SFT}, if it is of the form $X = \{\gamma \in A^G \mid \forall g \in G, g\gamma \in U\}$ with $U \subset A^G$ a clopen. Equivalently, $X$ is an SFT is there is a finite \emph{domain} $S \Subset G$ and a family $\mathcal{T} \subset A^S$ such that $\gamma \in X \iff \forall g \in G, X_{\mid gS} \in \mathcal{T}$. $\mathcal{T}$ is called the set of \emph{allowed patterns} and we say that $\mathcal{T}$ \emph{generates} $X$. 

Let $C \Subset S \Subset G$. A SFT with domain $S$ is \emph{$(C, S)$-totally extremally permutive} (TEP) if for each $c \in C$, for each $\pi : S \setminus\{c\} \to A$, there is a unique $a \in A$ such that the extension of $\pi$ with $\pi(c)=a$ is an allowed pattern. In other words, whatever the symbols chosen for the cells $S\setminus\{c\}$, there is a unique choice for $c$ that gives an allowed pattern.

In \cite[Example~5.15]{Sa22}, a very rich source of TEP subshifts is provided. We explain here a slightly simplified version of this example, using groups instead of quasigroups. 

\begin{example}
\label{ex:BasicTEP}
    Let $G$ be a group, $S \Subset G$ be arbitrary and $A$ be any abelian group. Then requiring that the entries of $S$ sums to zero (or indeed another constant) gives an $S$-TEP subshift. The Ledrappier subshift example discussed in more detail in Section~\ref{sec:Ledrappier} is of this nature. More generally, we can take $A$ a non-abelian group, as long as we specify the order in which its elements are multiplied. \qee
\end{example}

The Ledrappier example studied in Section~\ref{sec:Ledrappier} is the case $S = \{(0,0), (0,1), (1,0)\}$ of this construction with $A = \Z/2\Z$.



Examples of TEP-subshifts are shown in Figure~\ref{fig:TEP ex}. Those images come from \cite{Sa22}, and show samples of TEP subshift on $\Z^2$ and the free group $F_2$ with alphabet either $\{0,1\}$ and the rule as in Example~\ref{ex:BasicTEP}, with alphabet either $\Z/2\Z$ or $S_3$ (the symmetric group on three points). 

\begin{figure}
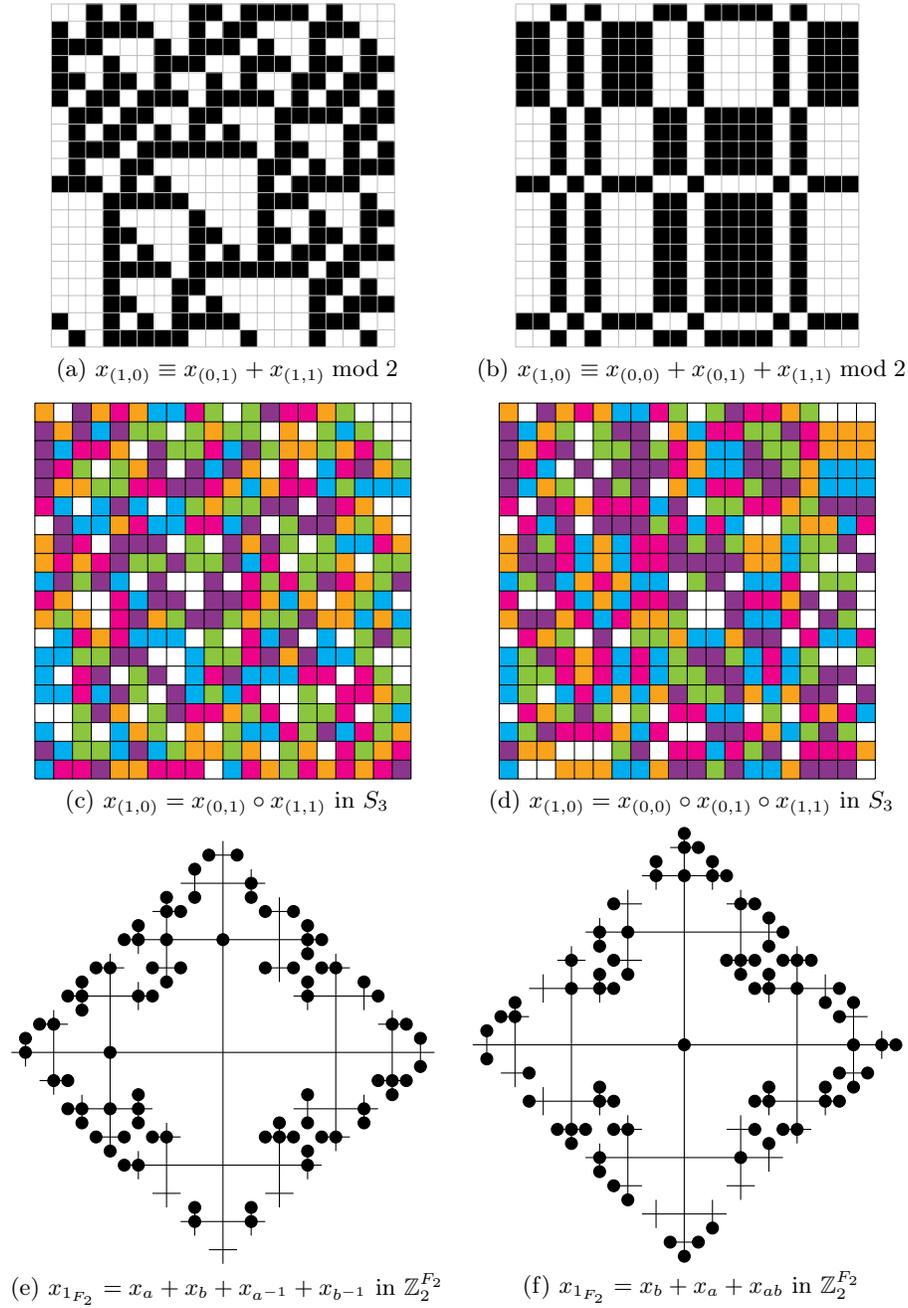

\begin{minipage}[t]{0.5\textwidth}
    \centering
    \includestandalone[width=0.75\textwidth]{Figures/TEP/bintriangle}
    \vspace*{-1mm}
    \subcaption{$x_{(1,0)} \equiv x_{(0,1)} + x_{(1,1)} \bmod 2$}
    \vspace*{2mm}
\end{minipage}
\begin{minipage}[t]{0.5\textwidth}
    \centering
    \includestandalone[width=0.75\textwidth]{Figures/TEP/binsquare}
    \vspace*{-1mm}
    \subcaption{$x_{(1,0)} \equiv x_{(0,0)} + x_{(0,1)} + x_{(1,1)} \bmod 2$}
    \vspace*{2mm}
\end{minipage}
\begin{minipage}[t]{0.5\textwidth}
    \centering
    \includestandalone[]{Figures/TEP/S3triangle}
    \vspace*{-1mm}
    \subcaption{$x_{(1,0)} = x_{(0,1)} \circ x_{(1,1)}$ in $S_3$}
    \label{fig:S3triangle}
    \vspace*{2mm}
\end{minipage}
\begin{minipage}[t]{0.5\textwidth}
    \centering
    \includestandalone[]{Figures/TEP/S3square}
    \vspace*{-1mm}
    \subcaption{$x_{(1,0)} = x_{(0,0)} \circ x_{(0,1)} \circ x_{(1,1)}$ in $S_3$}
    \vspace*{2mm}
\end{minipage}
\begin{minipage}[t]{0.5\textwidth}
    \centering
    \vspace*{0.8mm}
    \includestandalone[]{Figures/TEP/FreeLedra}
    \vspace*{-1mm}
    \subcaption{$x_{1_{F_2}} = x_a + x_b + x_{a^{-1}} + x_{b^{-1}}$ in $\Z_2^{F_2}$ }
    \label{fig:FreeLedra}
\end{minipage}
\begin{minipage}[t]{0.5\textwidth}
    \vspace*{-1mm}
    \centering
    \includestandalone[]{Figures/TEP/FreeU}
    \vspace*{-1mm}
    \subcaption{$x_{1_{F_2}} = x_b + x_a + x_{ab}$ in $\Z_2^{F_2}$ }
    \label{fig:FreeU}
\end{minipage}
\vspace*{-1mm}
\caption{Uniform samples from some TEP subshifts. In (a), (b), (e) and (f), black points are $1$s and white ones $0$. In (c) and (d), white is the identity, blue, green and orange the transpositions and purple and pink the $3$-cycles. Example (a) is the Ledrappier example.}
\label{fig:TEP ex}
\end{figure}

In all the examples above, $C = S$, but we note that there are also interesting cases where this does not happen. The case $C \subsetneq S$ is in particular interesting because it allows us to in a sense inject arbitrary non-permutive functions in the definition:

\begin{example}
\label{ex:TEPsWithf}
Let $G$ be a group, let $C \Subset S \Subset G$ be arbitrary. Let $A$ be an abelian group. Let $f : A^{S \setminus C} \to A$ be an arbitrary function. Let $P$ be the set of patterns $p \in A^S$ such that $f(p|_{S \setminus C}) + \sum p|_C = 0$. Then $P$ is $(C, S)$-TEP: given any values for $p \in A^{S \setminus \{c\}}$ with $c \in C$, there is a unique value we can put in $c$ that makes the sum equal to $0$. Again, $A$ need not necessarily be abelian, as long as we must pick an order of multiplication. \qee
\end{example}

TEP-subshift were introduced by the first author in \cite{Sa22}. They are our main motivation for studying the solitaire. Let us provide some context and historical motivation for why we are interested in TEP subshifts here. Subshifts of finite type (equivalently, tilings systems) have been studied since the 60's as logical \cite{Wang61} and dynamical systems \cite{Mozes89}, and since the very beginning, there has been a lot of emphasis on computability issues. Berger showed in 1969 \cite{Be66} that even the emptiness of a given SFT is an undecidable problem.

Thus it is of major interest in symbolic dynamics to find classes of SFTs where some dynamical problems can actually be solved in practical instances. Algebraic subshifts are the best-known source \cite{Schmidt,BeKa24}. TEP subshifts are an interesting non-algebraic class of SFTs (still including many interesting algebraic examples like the Ledrappier subshift), which also has some good computability properties.

In TEP subshifts, not only is emptiness decidable, but we can find concrete sets (only depending on the geometry of the defining shape and the group) where we can put arbitrary contents (these will be called independent sets). The solitaire is the natural closure property of these sets, and helps us to understand the valid patterns in a TEP subshift.

This TEP motivation is also part of the reason we want to study the solitaire on general groups rather than $\Z^2$. In recent decades, symbolic dynamics has moved increasingly to Cayley graphs of groups, and it is now known that similar phenomena to those on $\Z^2$ can be seen in a wide variety of groups. In particular, the undecidability of the emptiness problem (and the related problem of strongly aperiodic SFTs) has been exhibited on a large class of groups. For example this was shown for the lamplighter group (discussed in Example~\ref{ex:lamplighter}) in \cite{BaSa24}.

\subsection{Solitaire moves}

Natural objects to consider when studying a TEP subshift are its independent sets, namely the sets $I \subset G$ such that for all $\pi : I \to A$, there is a valid configuration $\gamma$ such that $\gamma_{|I} = \pi$. Those sets were introduced by the first author in ~\cite{Sa22}, who noticed that it is easy to exhibit some independent set but much harder to find them all. To produce more independent set from a known one, the \emph{solitaire of independence} was introduced, but only a preliminary discussion of its properties was given. 

We begin by defining the solitaire process on a group. We give a slightly more general definition than what was discussed in the introduction, where the unique hole can only be moved if it is one of the elements in a particular subset $C \Subset S$.

While our main interest is indeed in the case $C = S$, cases $C \subsetneq S$ also arise naturally in applications (see \cite{Sa22} and Section~\ref{sec:Ledrappier}) and do not add any difficulties in proofs.

\begin{definition}
Let $G$ be a countable group. If $C \Subset S \Subset G$ are finite sets, then a \emph{$(C, S)$-solitaire move at $g \in G$} is a pair $(P, Q) \in \pow(G)^2$ such that
\[ |P \cap gS| = |Q \cap gS| = |S| - 1 \wedge P \triangle Q \in \pow_2(gC). \]
In general, a \emph{$(C, S)$-solitaire move} is a $(C, S)$-solitaire move at any $g \in G$, and $(S, S)$-solitaire moves are called simply $S$-solitaire moves. A $(C, S)$-solitaire move $(P, Q)$ is usually written as $P \rightarrow_{C, S} Q$ or even $P \rightarrow Q$, and $P \rightarrow^* Q$ denotes the transitive closure of this relation.
\end{definition}

Intuitively, if there is a way to position an $S$-shape window on a pattern such that only one empty cell is visible, than one can do an $S$-solitaire move at this position, which consists in moving one of the visible points to the empty cell. For a $(C,S)$-solitaire move, elements in $S\setminus C$ are ``pivot only'' elements, i.e.\ they are needed to make a move but cannot be moved.

An example of solitaire on $\Z^2$ is given in Figures~\ref{fig:triangle_action} and \ref{fig:valid_moves}, the first one stating the rules and the second showing a few valid moves. An example on the free group on two generators is shown in Figure~\ref{fig:free move}.

\begin{figure}[ht!]
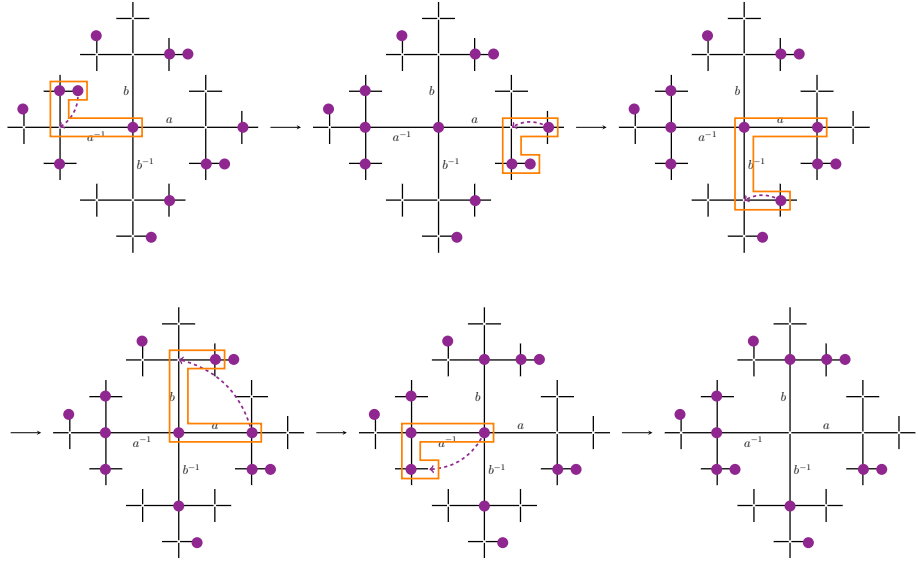

    \centering
    \includestandalone[width =\textwidth]{Figures/General/ex_free}
    \caption{A simple sequence of valid moves of the solitaire on the free group $F_2$ with $S = \{e, a^{-1}, a^{-1}b, a^{-1}ba$\}.}
    \label{fig:free move}
\end{figure}

An example with $C \neq S$ on $\Z^2$ is given in Figure~\ref{fig:C move}. One may check that the allowed moves and the orbits are very different, the moves being more and more restricted the smaller $C$ is.  

\begin{figure}[ht!]
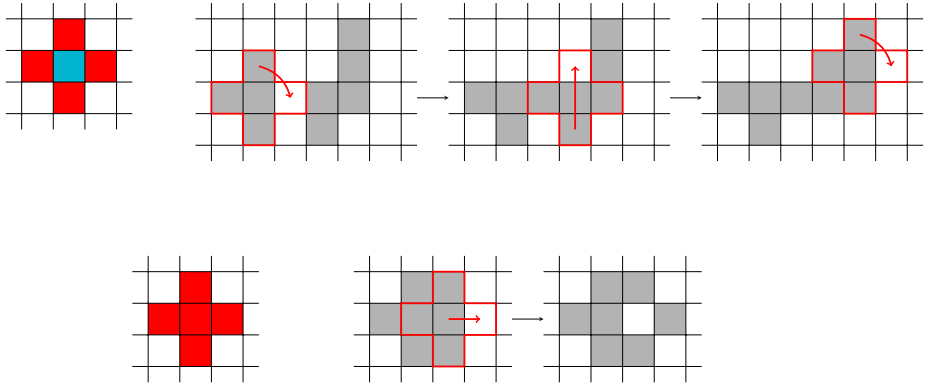

    \centering
    \includestandalone[width =\textwidth]{Figures/General/C_ex_Z2}
    \caption{First line: A simple sequence of $(C,S)$-solitaire moves on $\Z^2$ with $S =\{(0,0), (0,1), (1,0), (0,-1), (-1,0)\}$ and $C = S \setminus \{(0,0)\}$ depicted on the left ($C$ in blue and $S\setminus C$ in red). \\ Second line: a valid $(S,S)$-solitaire move that links two patterns that are not in the same $(C,S)$-orbit, with $S$ depicted on the left.}
    \label{fig:C move}
\end{figure}

Note that despite the arrow being directed, the moves are invertible. One can think of the $(C, S)$-solitaire moves in many ways. Of course, one can think of them as a relation, or as forming the edges of an undirected graph with vertices $\pow(G)$, or as the objects of an abstract rewriting system (with the edges giving the legal rewrite rules).

A slightly less obvious, but useful, point of view, one can think of solitaire moves as a group action. At each $g \in G$, we have an natural action of the group of permutations on $C$, $\mathrm{Sym}(C)$, which permutes the points in $gC$. Solitaire moves are a restriction of this action. Let $\sigma \in \mathrm{Sym}(gC)$, in configurations $P \in \pow(G)$ with $|P \cap gS| = |S| - 1 \wedge |P \cap gC| = |C| - 1$, its action moves each point in $x \in gC\cap P$ to the position $\sigma(x)$ and in other configuration, its action is trivial. Then the application of several solitaire moves translates as an action of the free product of the symmetric groups $\mathrm{Sym}(gC)$ over $g \in G$ on $\pow(G)$, which simply applies all actions sequentially. This action is by homeomorphisms when $\pow(G)$ is given the product topology.




The classical 15-puzzle roughly corresponds to the case $G = \Z^2, C = S = \{0,1\}^2$ of the solitaire (we study this case in Section~\ref{sec:Square}). It is standard to see the 15-puzzle as a groupoid, and the solitaire can also be seen as a groupoid in an obvious way. Though this point of view will not be used, we briefly recall that a \emph{groupoid} is a small category where every morphism invertible. The objects of the solitaire groupoid are then the configurations, and there is a unique morphism from each configuration to any other configuration that can be reached by solitaire moves.

If $P \in \pow(G)$, its \emph{orbit}, denoted as $\orb(P)$ is its equivalence class for the relation $\rightarrow^*$, i.e. the set of patterns that can be obtained from $P$ using solitaire moves. We say a set $X \subset \pow(G)$ is \emph{closed under solitaire} if $\forall P \in X: \forall Q \in \pow(G): P \rightarrow^* Q \implies Q \in X$. The following will be generalized in Theorem~\ref{thm:TISolitaire}.

\begin{proposition}[\cite{Sa22}]
    For all $C \Subset S \Subset G$, the set of $(C,S)$-independent sets is closed under solitaire.
\end{proposition}

This allows to use the solitaire moves to generate new independent sets from a known one. The question is then, starting from a given independent set, which ones can we reach with solitaire moves? In the following section, we answer this question for some cases of solitaire, such as the triangle solitaire on $\Z^2$ (Section~\ref{sec:Triangle}).

\begin{remark}
\label{rem:fg}
We note that the theory of solitaire might as well be built on finitely-generated groups. Namely, let $H = \langle S^{-1}S \rangle$, i.e.\ the subgroup of $G$ generated by all $a^{-1} b$ such that $a, b \in S$. Then it is easy to see that each solitaire move only touches elements in a single coset. Thus, the cosets behave completely independently, i.e.\ applying solitaire moves on points in one coset does not change the content of the other cosets, and no solitaire move relies on points from different cosets.
\end{remark}

\subsection{Filling process}
\label{sec:Filling}

Another very natural question with TEP-subshifts is ``given a partial assignment of symbols to a subset $P \subset G$, in which additional positions in $G$ can we deduce the contents?''. In this section, we present a natural process that aims to describe this domain. 

\begin{definition}
\label{def:Filling}
Let $G$ be a countable group. If $C \Subset S \Subset G$ are finite sets, then a \emph{$(C, S)$-filling move at $g \in G$} is a pair $(P, Q) \in \pow(G)^2$ such that
\[ P \varsubsetneq Q \wedge |P \cap gS| = |S| - 1 \wedge Q = P \cup gC. \]
In general, a \emph{$(C, S)$-filling move} is a $(C, S)$-filling move at any $g \in G$, and $(S, S)$-filling moves are called simply $S$-filling moves. A $(C, S)$-filling move $(P, Q)$ is usually written as $P \fillstep_{C, S} Q$ or even $P \fillstep Q$, and $P \fillstep^* Q$ denotes the transitive closure of this relation.
\end{definition}

\begin{figure}[ht]
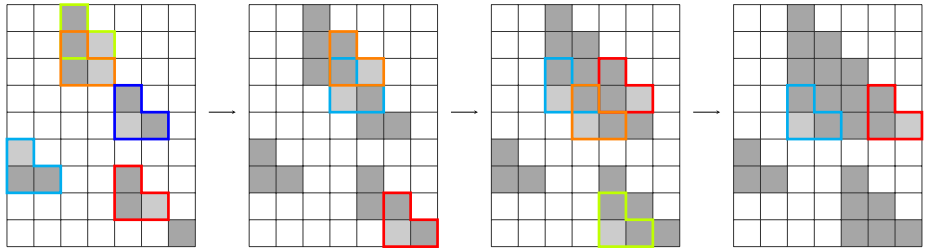

    \centering
    \includestandalone[width =\textwidth]{Figures/Triangle/filling}
    \caption{An example of the filling process for the triangle shape on $\Z^2$. Each colored triangle corresponds to a filling step, adding a point, and the final pattern is the filling closure. One can check that the order in which the filling steps are performed has no impact on the final result.}
    \label{fig:fill process}
\end{figure}

Intuitively, assuming $P$ is the set of cells of $G$ whose content is known, then performing a filling step on $P$ consists in using the TEP property of the subshift to deduce the content of other cells and extend the ``known'' domain. 

Note that the filling process is not reversible. However, we have another desirable property:

\begin{lemma}
The filling process has the \emph{diamond property}, meaning whenever $P \fillstep Q, P \fillstep R$ with $Q \neq R$, there exists $T$ such that $Q \fillstep T, R \fillstep T$. In particular, the system is \emph{confluent}, meaning whenever $P \fillstep^* Q, P \fillstep^* R$, there exists $T$ such that $Q \fillstep^* T, R \fillstep^* T$.
\end{lemma}

\begin{proof}
Suppose $P \fillstep Q, P \fillstep R$ with $Q \neq R$. Then $P \varsubsetneq Q$, $P \varsubsetneq R$ and there exist $g \in G$ such that 
\[ |P \cap gS| = |S| - 1 \wedge Q = P \cup gC \]
and $g' \in G$ such that
\[ |P \cap g'S| = |S| - 1 \wedge R = P \cup g'C. \]

Take $T = Q \cup R$. Observe that if $Q \cap g'S = g'S$, then $Q \supset R$ so in fact $Q = T = R$ (which we assumed does not hold). Thus we must have $|Q \cap g'S| = |P \cap g'S| = |S| - 1$ and thus $Q \fillstep Q \cup g'C = P \cup gC \cup g'C = T$ as desired. Symmetrically, $R \fillstep T$.

The fact that confluence follows from the diamond property is well-known, and follows from completing the paths $P \fillstep^m Q, P \fillstep^n R$ to an $m$-by-$n$ rectangle, by successive applications of the diamond property. \qed
\end{proof}

\begin{definition}
Let $G$ be a countable group. If $C \Subset S \Subset G$ are finite sets, then the \emph{(C, S)-filling closure} of $P \in \pow(G)$ is defined as
\[ \varphi_{C, S}(P) = \bigcup \{Q \;|\; P \fillstep_{C, S}^* Q\}. \]
\end{definition}

\begin{remark}
Recall that a \emph{directed set} is a set together with a partial order such that any two elements have a common upper bound. The previous lemma shows that for $P \subset G$, the set $F_P = \{Q \subset G \mid P \fillstep_{C, S}^* Q\}$ is a directed set under $\fillstep^*$ and thus under inclusion, since the filling process can only increase sets. The set $F_P$ defined above is precisely the limit of this set  (seen as a net indexed by itself) in the product topology of $\pow(G)$. (For the definition of a limit of a net see any standard reference \cite{Ke17}.)
\end{remark}

\begin{lemma}
The $(C, S)$-filling closure operator $\varphi$ is a closure operator.
\end{lemma}

\begin{proof}
By definition, we need to show $P \subset \varphi(P)$ (extensivity), $P \subset Q \implies \varphi(P) \subset \varphi(Q)$ (monotonicity) and $\varphi(\varphi(P)) = \varphi(P)$ (idempotency) for all sets $P, Q$. But in fact, all of these properties are obvious. \qed
\end{proof}

It follows from the previous lemma that the set of sets $P \in \pow(G)$ which are their own filling closure form a closure system, and we call the fixed points \emph{filling closed sets}. In particular, the following is a well-known consequence:

\begin{lemma}
\label{lem:IntersectionOfAll}
The filling closure $\varphi(P)$ is the intersection of all filling closed sets that contain $P$. 
\end{lemma}


The connection between the solitaire and the filling process is now the following:

\begin{lemma}
\label{lem:ContainedInFilling}
Let $G$ be a countable group, and let $C \Subset S \Subset G$ be finite sets. Let $P \in \pow(G)$ be arbitrary. Then for all $Q \in \orb(P)$, we have $\varphi_{C,S}(Q) = \varphi_{C,S}(P)$.
\end{lemma}

\begin{proof}
It suffices to show that if $P \in \pow(G)$ and $P \rightarrow_{C, S} Q$, then $Q \subset \varphi(P)$. Namely then by idempotency $\varphi(Q) \subset \varphi(P)$ as well.

Suppose thus $P \rightarrow_{C, S} Q$, meaning
\[ |P \cap gS| = |Q \cap gS| = |S| - 1 \wedge P \triangle Q \in \pow_2(gC). \]

It suffices to show that $P \fillstep P \cup Q$ by a move at $g$, i.e.\
\[ |P \cap gS| = |S| - 1 \wedge P \cup Q = P \cup gC, \]
because then $Q \subset P \cup Q \subset \varphi(P)$. Of course $|P \cap gS| = |S| - 1$ is immediate. 
We have $P \cup gC = P \cup Q$ because the only difference between $P$ and $Q$ is that they lack a different element of $gC$. \qed
\end{proof}

We note a useful property of the filling operator (which is a general property of closure operators).

\begin{lemma}
\label{lem:ChangeSubset}
Suppose $\varphi(P) = \varphi(Q)$. Then $\varphi(P \cup R) = \varphi(Q \cup R)$ for any $R$.
\end{lemma}

\begin{proof}
\[ \varphi(P \cup R) = \varphi(\varphi(P \cup R)) \supset \varphi(Q \cup R) \]
since $R \subset P \cup R \subset \varphi(P \cup R)$ and $Q \subset  \varphi(P) \subset \varphi(P \cup R)$.
\qed
\end{proof}

\subsection{Linear sets, finite fillings and finite solitaires}

Our interest is specifically on the solitaire orbits of finite sets. More precisely, we would like to be able to answer quickly to the question ``are these two patterns in the same orbit?''. A first step toward this is to answer the question ``given a shape $S$ and a pattern $P$, is the $S$-orbit of $P$ finite?''. The goal of this section is to introduce definitions and state initial results toward answering this second question.

For the following definitions, we concentrate on the case $C = S$. 


\begin{definition}
Let $G$ be a group. A subset $S \Subset G$ is \emph{linear} if there exists $a, b, c \in G$ with $b$ of infinite order, such that $S \subset a \langle b \rangle c$ where $\langle b \rangle = \{b^n \;|\; n\in \Z\}$.
\end{definition}

Note that $\langle b \rangle$ is the (cyclic) group generated by the element $b$, and $b$ being of infinite order means precisely that this group is infinite. Recall that a \emph{left coset} of a subgroup $H \leqslant G$ is a set of the form $aH$.

\begin{lemma}
\label{lem:Linear}
If $S$ is a linear shape, then there exists a shape $S'$ which is contained in a cyclic group, such that the set of $S$-solitaire moves coincides with the set of $S'$-solitaire moves.
\end{lemma}

\begin{proof}
We have
\[ S \subset a \langle b \rangle c = ac \langle  c^{-1}bc \rangle. \]
Then $S' = (ac)^{-1} S \subset \langle c^{-1}bc \rangle$. Clearly $S'$-solitaire moves are equivalent to $S$-solitaire moves, as the set of moves only depends on the set of left translates of the shape. \qed
\end{proof}

\begin{figure}[ht]
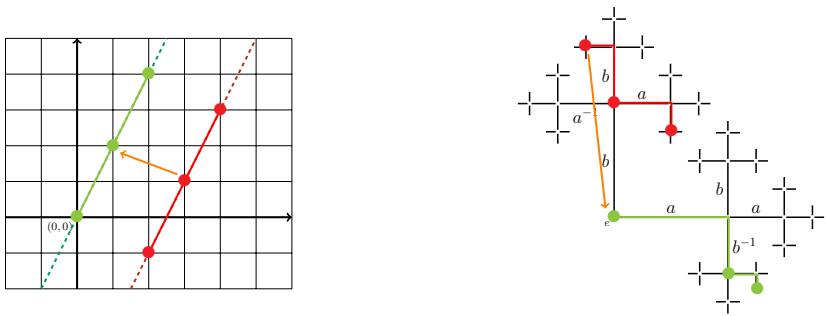

    \centering
    \includestandalone[width=0.9\textwidth]{Figures/General/translate}
    \caption{Two examples of linear shapes, on $\Z^2$ on the left side and on $F_2$ on the right side. $S$ is the red shape and the green one is a possible $S'$. \\ Left: $S \subset (2, -1)\langle (1,2) \rangle$ $S' \subset \langle (1,2) \rangle$. Right: $S \subset b^2a^{-1}\langle ab^{-1} \rangle$ $S' \subset \langle ab^{-1}\rangle$.}
    \label{fig:translate}
\end{figure}

\begin{definition}
Let $G$ be a group and $S \Subset G$. We say $G$ has the \emph{finite $S$-filling property} if the $S$-filling closure of every finite set is finite. We say $G$ has the \emph{finite filling property} if this holds for all non-linear $S \Subset G$. We say $G$ has the \emph{finite $S$-solitaire property} if the $S$-solitaire orbit of every finite set is finite, and \emph{finite solitaire property} if this holds for all non-linear $S$.
\end{definition}

The assumption of non-linearity is necessary:

\begin{lemma}
If $S \Subset G$ is linear, and $|S| \geqslant 2$, then there is a finite set with infinite $S$-filling closure and infinite $S$-solitaire orbit.
\end{lemma}

\begin{proof}
By Lemma~\ref{lem:Linear} above we may suppose $S \subset \langle b \rangle$ with $b$ of infinite order. Let $m$ be maximal and $n \geqslant m$ minimal such that $S \subset T = \{b^m, \ldots, b^n\}$. Then $T \setminus \{b^m\}$ is easily seen to have infinite $S$-filling orbit and infinite $S$-solitaire orbit: one can slide $S$ along $\cyc{b}$, adding points in one direction or moving $b^{k+i}$ to $b^{k+j}$ where $i$ (resp. $j$) is the minimum (resp. maximum) of $\{k\in \Z\mid b^k \in S\}$. \qed
\end{proof}

\begin{figure}[ht]
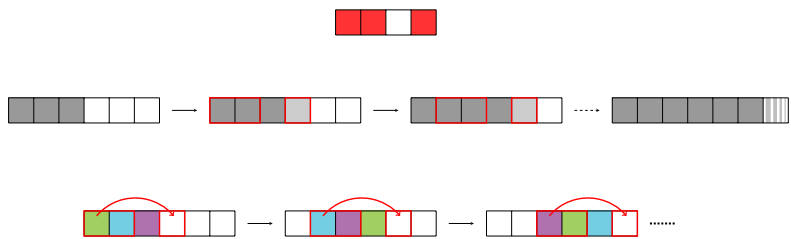

    \centering
    \includestandalone[width =0.85\textwidth]{Figures/General/linear}
    \caption{An example of linear solitaire on $\Z^2$. In red, the shape $S$, on the first line a pattern with infinite filling and on the second line an infinite run of solitaire with the same patterns (the three points having different colors).}
    \label{fig:linear}
\end{figure}

The following is immediate from Lemma~\ref{lem:ContainedInFilling}.

\begin{lemma}
If $G$ has the finite filling property, then it has the finite solitaire property.
\end{lemma}
           
\begin{remark}
    The converse is not true. For example, take $G$ the cylinder $\Z \times (\Z/d\Z)$ and $S$ the $T$ shape $\{(0,0), (0,1), (0,2), (1,1)\}$. Then for a ring pattern $P = \{0\}\times (\Z/d\Z)$, $\phi_S(P) = \N \times (\Z/d\Z)$ but the orbit of $P$ is finite. (See Figure~\ref{fig:cylinder})
\end{remark}

\begin{figure}[ht]
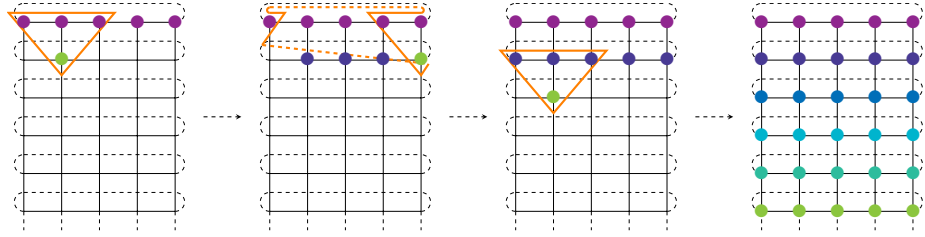

    \centering
    \includestandalone[width = \textwidth]{Figures/General/cylinder}
    \caption{An example of solitaire with finite orbit but infinite filling on $\Z \times (\Z/d\Z)$. The filling extends infinitely downward but the orbit is contained in the first three rings.}
    \label{fig:cylinder}
\end{figure}

The following is a simple geometric observation.

\begin{proposition}
\label{prop:ZdFFP}
The groups $\Z^d$ have the finite filling property.
In particular, they have the finite solitaire property.
\end{proposition}

\begin{proof}
The case $d = 1$ is trivial (and rather meaningless) since there are no non-linear subsets.

We briefly explain the geometric intuition. Suppose $S \Subset \Z^d$ is non-linear. As we explain below, this implies $d \geqslant 2$, or we can reduce the problem to one with smaller $d$. Next we observe that a hyperplane with the same normal direction as a face of the convex hull of $S$ cannot be crossed by the filling process, i.e.\ if there is no point on the other side, the filling process cannot add one. The proof amounts to eliminating various boundary cases (like indeed shapes that fit into a proper affine subspace) and then using basic facts from convex geometry. 

Without loss of generality, assume that $\vec 0 \in S$ (this does not change the $S$-filling nor the $S$-solitaire), and that $S$ generates $\Z^d$ as a group. Otherwise, since all subgroups of $\Z^d$ are of the form $\Z^k$ by the fundamental theorem of abelian groups, we can reduce $d$ and apply the solitaire process separately in each coset. (See Remark~\ref{rem:fg}.)

Now, since $S$ generates $\Z^d$, its convex hull $H$ in $\R^d$ has nonempty interior, as it is well-known \cite{ConOpt} that a convex set without interior lies in a proper affine set $F$, and then $S$ then sits inside the proper subgroup $\Z^d \cap F$, so we could have again taken a smaller $d$. Now replace $S$ by $kS$ for large $k \in \N$ (i.e.\ dilate the shape; note that this preserves the finite filling property because $kS$-solitaire consists of independent $S$-solitaires on cosets of $k\Z^d$), and then translate it, after which we may actually assume that $\vec 0 \in \overset{\circ}{H}$ (after this, we may no longer have $\vec 0 \in S$).

By basic convex geometry (see \cite{ConAna}, Theorem 19.1), such a set is a bounded intersection of a finite number of closed half-spaces, i.e.\ we have $H = \bigcap_i \ell_i^{-1}((-\infty, r_i])$ where $\ell_i : \R^d \to \R$ are linear and nontrivial. Note that $r_i > 0$ since $\vec 0$ is in the interior of $C$. Furthermore, the half-spaces $\ell_i((-\infty, r_i])$ can be taken to satisfy that each $\ell^{-1}(r_i) \cap H$ is a face (i.e.\ a set spanned by vertices) of $H$ of dimension $d-1$.

In particular, the faces are spanned by $d$ vertices of $H$, which must be elements of $S$, thus, $|\ell_i^{-1}(r_i) \cap S| \geqslant d \geqslant 2$ for all $i$. In particular, it follows that $P_i = \ell_i^{-1}((-\infty, r])$ is closed under $S$-filling for all $r \in R$. Namely, we have $\varphi(S) \subset (-\infty, r_i]$ and $\varphi(s) = r_i$ for at least two distinct $s \in S$. Thus the intersection $\vec v + S \cap P_i$ is always of cardinality $|S|$ (when $r < r_i$) or of cardinality at most $|S|-2$ (when $r > r_i$), and thus Definition~\ref{def:Filling} does not apply for any $\vec v \in \Z^d$.

Now let $T \Subset \Z^d$ be arbitrary. Since $T$ is finite, it is bounded and thus $\ell_i(T)$ is bounded for any $i$, say $\forall t \in T: \ell_i(t) \leq t_i$. Then $T \subset D = \bigcap_i \ell_i^{-1}((-\infty, t_i])$. As we showed, each $\ell_i^{-1}((-\infty, t_i])$ is closed under $S$-filling, thus so is their intersection.

It suffices to show that $D$ is bounded. If it is not, then there exist arbitrarily large (in norm) $\vec v$ such that $\ell_i(\vec v) \leqslant t_i$ for all $i$. By scaling, we obtain that there exist arbitrarily large $\vec v$ such that $\ell_i(\vec v) \leqslant r_i$ for all $i$. Such vectors $\vec v$ are then in $H$, so in fact $S$ itself is unbounded, a contradiction. \qed
\end{proof}

\begin{figure}[ht]
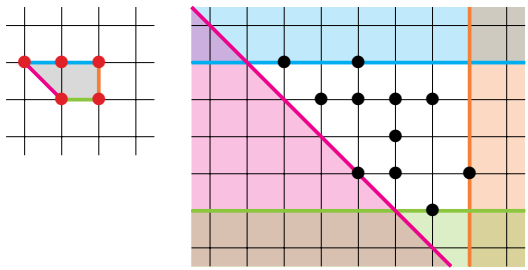

    \centering
    \includestandalone[height = 3.5cm]{Figures/General/boundary}
    \caption{With $S$ in red on the left, colored lines are parallel to the edges of $S$. For each of them, the boundary contains two points or more, so no point can enter the colored half-planes.}
    \label{fig:boundary}
\end{figure}

\subsection{Solitaire on exotic groups}

We mostly concentrate on the free group and the group $\Z^2$, since these groups are easy to draw, and should be familiar (or easy to learn) for most readers. However, there are many other interesting groups where the solitaire can be played. The following example illustrates solitaire on the lamplighter group, which is a group of great symbolic dynamical interest \cite{Cohen17,BaSa22,BaSa24}.

\begin{example}
\label{ex:lamplighter}
The lamplighter group is the wreath product $\Z_2 \wr \Z$, in other words it is the semidirect product $\bigoplus_{i \in \Z} \Z_2 \rtimes \Z$ where $\Z$ acts on the infinite direct sum $V = \bigoplus_{i \in \Z} \Z_2$ by shifting. We take as generators $a, b$ where $a$ is the generators of $\Z$, and $b = a \circ c$ where $c \in V$ is the element with $1$ at only the origin.

As in \cite{BaSa24} we prefer to shift the indices defining $V$, and think of $V$ as the set of configurations in $\Z_2^{\Z + \frac12}$ with finitely many non-zero coordinates, which we think of as lamps on an infinite vertical road. Then the nodes of the Cayley graph (i.e.\ group elements) can be taken to be pairs $(v, n)$ where $v \in V$ and $n \in \Z$, meaning a configuration of lamps, and the lamplighter between two of the lamps. Then generator $a$ adds $1$ to $n$ ("the lamplighter moves upward") and $b$ adds $1$ to $n$ and toggles the value of $v_{n+1/2}$ ("the lamplighter moves upward and toggles the lamp it moves over").

The typical finite subgraphs of the Cayley graph that one considers are the ``tetrahedra'', which correspond to the subsets of the group where the head is in a particular interval $[m, n]$, and only lamps inside $(m, n)$ have been flipped. These graphs can be relatively nicely illustrated on the plane, see Figure~\ref{fig:lamplighter}. This is a so-called horocyclic product of two binary trees, meaning upward-growing binary trees are joined with downward growing binary trees. It is useful to think of this as a three-dimensional embedding of the group, with the $a$ and $b$ edges moving ``upward'', and the elements colored black as being on the bottom, forming a front-to-back column.

The simplest interesting shape for the solitaire is $S = \{e, a, c\}$ where $c = ab^{-1}$. Note that $c$ has order $2$ and we have $cS = \{e, b, c\}$, see Figure~\ref{fig:lamplighter} (top). One can check that the filling closure of the bottom elements is the entire tetrahedron. Figure~\ref{fig:lamplighter} (bottom) shows another pair of $S$-solitaire-equivalent subset. \qee
\end{example}

\begin{figure}[ht]
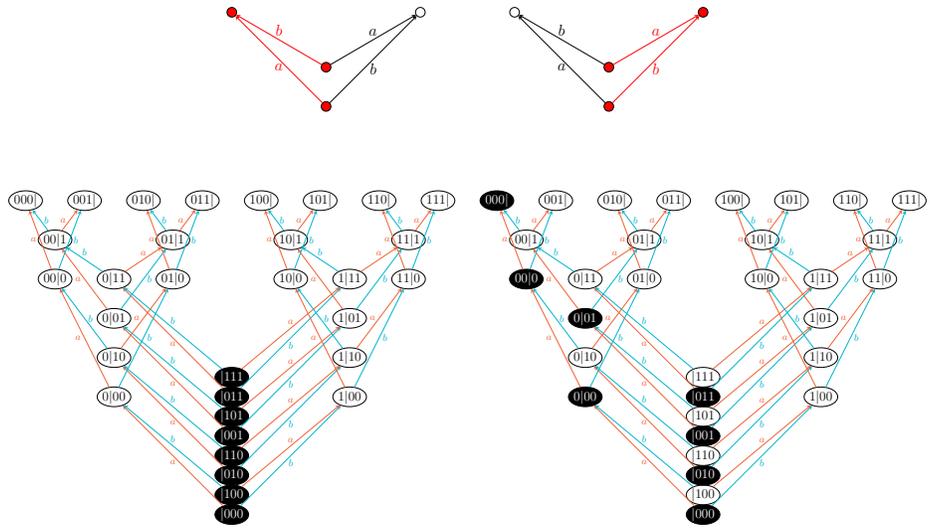

    \centering
    \includestandalone[width = \textwidth]{Figures/General/lamplighter}
    \caption{Top: The shape $S = \{e, a, c\}$ (left) and its translation $cS$ (right). Bottom: Two $S$-solitaire equivalent patterns on the lamplighter group.}
    \label{fig:lamplighter}
\end{figure}

There are many other groups where solitaire would be interesting to study in detail. One source are nilpotent and solvable groups, which are intuitively close to abelian (and more precisely, built by group extensions from abelian groups). The lamplighter group from the previous example is solvable (in fact metabelian). Another interesting group in this class is the Heisenberg group (the upper triangular integer matrices with unit diagonal). Another interesting class of groups to look at would be free products like $\Z * \Z^2$ (intuitively, this group combines the abelian case $\Z^2$ with the free group $F_2 = \Z * \Z$) and more generally graph groups.

\section{Excess}
\label{sec:Excess}

Much of our interest is in understanding the connection between the filling process and the solitaire process. We measure the difference between a pattern and a minimal pattern with the same filling closure by the notions of excess and  excess sets.

In the case of the triangle, we will see that for each filling closed set, all patterns with minimal cardinality that generate it lie in the same solitaire orbit.

\subsection{Excess in a global sense}

Fix $C \Subset S \Subset G$ for this section, and suppose $G$ has finite $S$-fillings.

If $P \Subset G$, then the \emph{rank}, $\rank(P)$, of $P$ is the minimal cardinality of $R \Subset G$ such that $\varphi(R) = \varphi(P)$. Note that the rank of $P$ is at most $|P|$. Note also that the only  candidate subsets $R \Subset G$ we need to consider are subsets of $\varphi(P)$.

\begin{definition}
The \emph{excess} of $P \Subset G$ is defined as $e(P) = |P| - \rank(P)$.
\end{definition}

This is excess in a global sense, since we are counting how much smaller we could make $P$ by replacing it by another (possible quite different-looking) set, without changing its filling closure.

Note that every filling closed set $F$ has by definition at least one subset with numerical excess $0$.

\begin{figure}[ht]
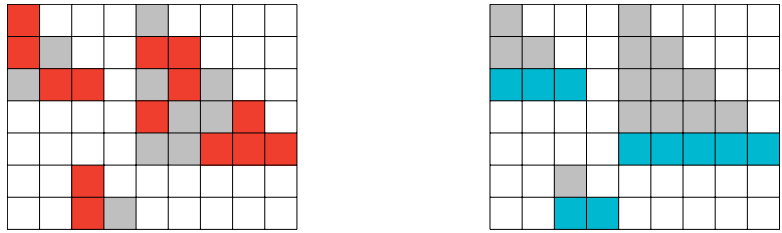

    \centering
    \includestandalone[height = 3cm]{Figures/Triangle/ex_excess}
    \caption{Left: A pattern $P$ on $\Z^2$ in red with its filling for the triangle shape $\{(0,0), (0,1), (1,0)\}$ drawn in gray. $P$ has size $14$, rank $10$ and excess $4$. Right: A minimal set $R$ in blue with the same filling.}
    \label{fig:ex excess triang}
\end{figure}

\begin{figure}[ht]
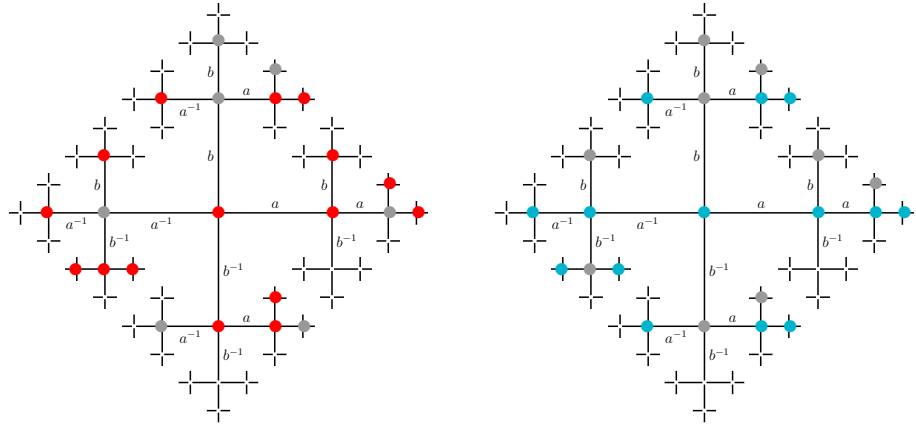

    \centering
    \includestandalone[width=\textwidth]{Figures/FreeGroup/excess}
    \caption{Left: A pattern $P$ on $F_2$ in red with its filling for the pyramid shape $\{e, a, b, a^{-1}\}$ drawn in gray. $P$ has size $16$, rank $14$ and excess $4$. Right: A minimal set $R$ in blue with the same filling.}
    \label{fig:ex excess free}
\end{figure}

\begin{lemma}
\label{lem:ExcessMonotone}
Excess is monotone: if $P \subset Q \Subset G$, then $e(P) \leqslant e(Q)$.
\end{lemma}

\begin{proof}
Let $R$ be a set of size $\rank(P)$ with filling $\varphi(P)$. 
Then $\varphi(Q) = \varphi(P \cup (Q \setminus P)) = \varphi(R \cup (Q \setminus P))$ by Lemma~\ref{lem:ChangeSubset}. Thus $\rank(Q) \leqslant \rank(P) + |Q \setminus P|$ so
\[ e(Q) = |Q| - \rank(Q) \geqslant |P| + |Q \setminus P| - \rank(P) - |Q \setminus P| = e(P).  \]
\qed
\end{proof}

\subsection{Excess in a local sense}

\begin{definition}
Let $P \Subset G$. We say $Q \subset P$ is an \emph{excess set} if $\varphi(P \setminus Q) = \varphi(P)$. The set of excess sets of $P$ is denoted $E(P)$. The \emph{visible excess} of $P$ is the maximal cardinality of an excess set in $P$, denoted $\hat E(P)$. The \emph{phantom excess} of $P$ is $e(P) - \hat{E}(P)$.
\end{definition}

\begin{figure}[ht]
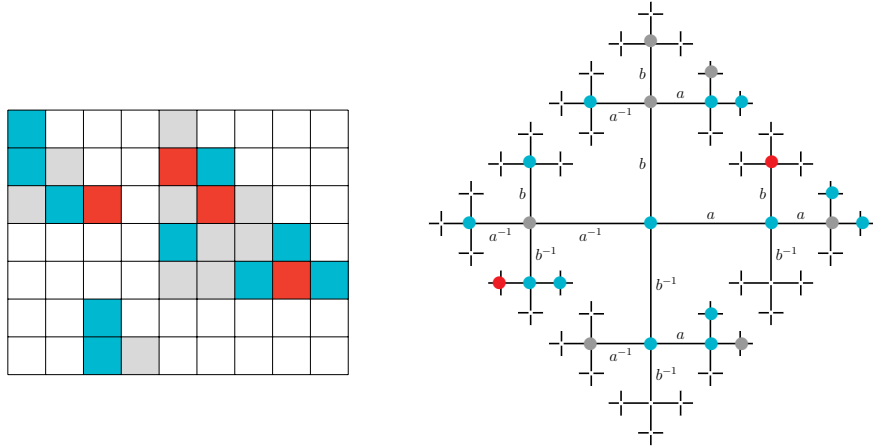

    \centering
    \includestandalone[width=0.95\textwidth]{Figures/General/excess_sets}
    \caption{Two examples of patterns with an excess set highlighted. Left: On $\Z^2$ with $S = \{(0,0), (0,1), (1,0)\}$. Right: On $F_2$ with $S = \{e, a, b, a^{-1}\}$. The pattern is composed of the blue and red points, the red points form an excess set and the grey points are the filling of the pattern.}
    \label{fig:ex excess set}
\end{figure}

Note that excess sets are defined for a given set $P$ (and of course they depend on the choice of $C, S$). Clearly if $P$ contains an excess set of cardinality $k$, then the $e(P) \geqslant k$. Thus, excess is always at least as large as visible excess, and phantom excess is always nonnegative.

It is tempting to try to prove that phantom excess is always zero, or even that all maximal excess sets have the same cardinality, which equals excess. As we will see in Section~\ref{sec:TriangleExcess}, this fails badly even on the group $\Z^2$, and for the simplest possible non-linear shape.

\begin{lemma}
The excess sets of $P$ form a down set.
\end{lemma}

\begin{proof}
If $Q$ is an excess set of $P$ and $Q' \subset Q$, then $\varphi(P) = \varphi(P \setminus Q) \subset \varphi(P \setminus Q') \subset \varphi(P)$. \qed
\end{proof}

\begin{lemma}
Visible excess is monotone: if $P \subset Q \Subset G$, then $\hat{E}(P) \leqslant \hat{E}(Q)$.
\end{lemma}

\begin{proof}
If $R\subset P$ is an excess set then $\varphi(P) = \varphi(P \setminus R)$, so 
\[ \varphi(Q) = \varphi(P \cup (Q \setminus P)) = \varphi((P \setminus R) \cup (Q \setminus P)) = \varphi(Q \setminus R) \]
so $R$ is an excess set for $Q$, implying $\hat{E}(Q) \geqslant \hat{E}(P)$. \qed
\end{proof}

\begin{lemma}
\label{lem:ContainsNoPhantom}
If a pattern contains a subpattern with no excess and with the same filling closure, then it has no phantom excess.
\end{lemma}

\begin{proof}
If $P \subset Q$ where $P$ has no excess and $\varphi(P) = \varphi(Q)$, then $P$ can be used to calculate excess for $Q$, and we have $e(Q) = |Q \setminus P|$. On the other hand, $Q \setminus P$ is clearly an excess set of this cardinality, thus visible excess matches excess. \qed
\end{proof}

Interestingly, visible excess is \emph{not} preserved under the solitaire process. We will see an example in Section~\ref{sec:TriangleExcess}.

\subsection{Monotonicity of solitaire, and ``ignoring elements''}
\label{sec:Pretending}

Fix $C \Subset S \Subset G$ for this section.

We show that the solitaire process is monotone in the sense that after adding elements to a set, we can still apply all moves we could previously. Intuitively, we can always pretend not to see some of the elements in a pattern, and play solitaire on the subpattern. The elements we pretended not to see are then permuted in some (possibly nontrivial) way.

Recall that we may think of a solitaire move as applying a permutation of $C$ in $gC$ (through the natural identification $c \leftrightarrow gc$), and we are allowed to apply this permutation to a pattern $P$ if $g(S \setminus C) \subset P$ and $|gC \cap P| = |C| - 1$. We note that it does not hurt if we also allow applying the permutation when $gS \subset P$, as in this case the permutation fixes the pattern anyway.

\begin{definition}
We say that \emph{solitaire applies} at $g$ to pattern $P$ if $g(S \setminus C) \subset P$ and $|gC \cap P| \geqslant |C| - 1$. Then for any permutation $\pi : C \to C$ we have the \emph{associated solitaire move} $g\pi$ defined by
\[ P \rightarrow P' = (P \setminus gC) \cup g\pi(g^{-1}P \cap C). \]
We also use function notation $g\pi(P) = P'$.
\end{definition}

Note that $g^{-1}P \cap C$ is in bijection with $P \cap gC$, and this bijection is the natural one, $c \leftrightarrow gc$. Thus we are indeed applying $\pi$ to the subset of $P$ intersecting $gC$, and finally moving it back to a subset of $gC$.

\begin{figure}[ht]
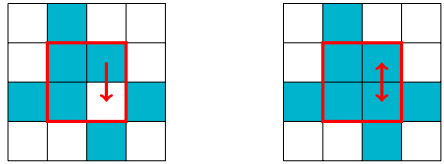

    \centering
    \includestandalone[height=2.1cm]{Figures/General/monotonicity}
    \caption{Left: A standard solitaire move for the square shape on $\Z^2$. Right: A trivial application of the solitaire. (Here, $\pi$ is the transposition exchanging $(1,0)$ and $(1,1)$.)}
    \label{fig:monotonicity}
\end{figure}

\begin{lemma}
Let $P \subset G$. If solitaire applies at $g$ to pattern $P$, then it applies to any $P \cup R$. Furthermore, $g\pi(P \cup R) \supset g\pi(P)$.
\end{lemma}

\begin{proof}
Permutations at $g$ stay applicable when the set increases: the only possible change in the intersection with $gS$ is that $(P \cup R) \cap gS$ becomes $gS$ (because already $P \cap gS$ misses at most one element). Both $g\pi(P \cup R)$ and $g\pi(P)$ are produced by applying the same permutation in the same set of cells. \qed
\end{proof}

\begin{lemma}
\label{lem:MonotonicityUpToPermutation}
Suppose $P \rightarrow^* P'$. Then there exists a bijection $\pi : G \setminus P \to G \setminus P'$ with support contained in $\varphi(P)$, such that $P \cup Q \rightarrow^* P' \cup \pi(Q)$ for all $Q \subset G \setminus P$.
\end{lemma}

\begin{proof}
Since $P \rightarrow^* P'$, there is a sequence of patterns $P = P_0 \rightarrow P_1 \rightarrow \cdots \rightarrow P_m = P'$, thus there is a sequence of permutations $g_1, g_2, \ldots, g_m$ and a sequence of permutations $\pi_1, \ldots, \pi_m$ of $C$ such that solitaire applies at $g_i$ in $P_{i-1}$, and $P_i = g_i\pi_i(P_{i - 1})$.

Let $R \cap P = \varnothing$, and define $Q = Q_0 = P \cup R$. Induction and the previous lemma show that the sequence $Q_i = g_i\pi_i(Q_{i - 1})$ is well-defined (because solitaire applies at $g_i$ to pattern $Q_{i - 1}$), and $Q_i$ contains $P_i$.

The entire evolution of the pattern comes from a sequence of permutations applied to $G$, proving the claim about there existing of a single permutation of $\pi$.

For the claim about the support, observe that each $P_i$ is contained in the filling closure $\varphi(P)$. The only way solitaire could apply at $g_i$ in $P_{i-1}$, without $g_iS$ being contained in $\varphi(P)$, is that actually $g_i\pi_i$ moves an element outside of $\varphi(P)$, which is impossible. \qed
\end{proof}

In particular:

\begin{lemma}
\label{lem:MonotoneUpToPermutation}
Suppose $P \rightarrow^* P'$. Then for any $Q$, we have $P \cup Q \rightarrow^* R$ for some pattern $R$ containing $P'$.
\end{lemma}

\begin{proof}
Take $R = P' \cup \pi(Q)$ in the previous lemma. \qed
\end{proof}

\subsection{Transporting excess}
\label{sec:TransportingExcess}

Fix again $C \Subset S \Subset G$ for solitaire purposes.

Sometimes a filling closed set $F$ contains a filling subset $P$ (i.e.\ $F = \varphi(P)$), which allows us to freely perform transformations on the excess sets. Intuitively, $P$ can be moved around $F$ by the solitaire freely enough that it can ``drag other elements along'' anywhere in $F$, and ultimately permute them arbitrarily.

\begin{definition}
Suppose $P, R \Subset G$. The \emph{$R$-restricted orbit} of $P$ is the set of patterns $Q \Subset G$ obtained from $P$ by applying only solitaire moves only at $g \in G$ such that $gS \subset R$.
\end{definition}

\begin{definition}
Suppose $P \Subset G$ and $F = \varphi(P)$. We say $P$ \emph{transports $k$-excess inside $R$} if all of the patterns $P \sqcup Q \subset R$ where $|Q| \leqslant k$ are in the same $R$-restricted $S$-orbit. We say $P$ \emph{transports $k$-excess} if it transports $k$-excess inside $F$. We say $P$ transports excess (inside $R$) if it transports $k$-excess (inside $R$) for all $k$.
\end{definition}

A \emph{super-$P$} pattern is $Q \supset P$ such that $\varphi(Q) = \varphi(P)$. In terms of this notion, $P$ transports $k$-excess if and only if every for every $\ell \leq |P|+k$, every super-$P$ pattern with cardinality $\ell$ is in the same orbit, and $P$ transports excess if and only if every super-$P$ pattern with the same cardinality is in the same orbit. Note that in this definition $P$ itself is allowed to have excess, although in our applications it will not. To explain the terminology, it is the excess \emph{on top of} the excess of $P$ itself that its transported.

\begin{figure}[ht]
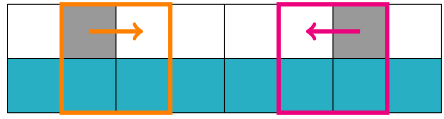

    \centering
    \includestandalone[height=1.5cm]{Figures/General/transport}
    \caption{A horizontal line transports excess on the lines above and under it for the square solitaire on $\Z^2$.}
    \label{fig:transport}
\end{figure}

By definition, if $P$ transports $k$-excess inside $F$, then it transports $(k-1)$-excess. We show that in at least very special situations, the converse is true.

\begin{lemma}
\label{lem:Ordering}
Suppose $P \Subset G$ and $F = \varphi(P)$. Suppose $F \setminus P = \{p_1, \ldots, p_n\}$ and $P$ transports $1$-excess in $R_k = P \cup \{p_1, \ldots, p_k\}$ for all $k$. Then $P$ transports excess.
\end{lemma}

\begin{proof}
By induction on $k$ and $n$, we show that $P$ transports $n$-excess in $R_k$ for all $k, n$. For $k = 0$, this is trivial, and for a fixed $k$, it is trivially true for $n = 0$, since in these cases $P$ is the only pattern to consider.

Consider then a pattern $Q$ of size $n + 1$ inside $R_k$, such that the claim holds for patterns up to size $n$, and holds for all smaller values of $k$. If $n+1 = k$, then $Q$ is the only pattern of its size, and the claim holds trivially, so suppose $n+1 < k$. If $Q$ does not contain $p_k$, we may think of it as being contained in $R_{k-1}$, inside which we can transport $(n+1)$-excess freely. Thus, it suffices to show that $p_k$ can be moved inside $R_{k-1}$ without involving elements outside $R_k$.

Actually, it is easier to show that we can move an element into $p_k$. Namely, since $n + 1 \geqslant 1$, we know that $Q$ contains some $p_i$. Since $P$ transports $1$-excess, we have $P \cup \{p_i\} \rightarrow^* P \cup \{p_n\}$.

By Lemma~\ref{lem:MonotoneUpToPermutation}, we then have $P \cup Q \rightarrow^* P \cup R$ for some $R \ni p_k$. In fact, by Lemma~\ref{lem:MonotonicityUpToPermutation}, we can choose $R \setminus \{p_k\}$ freely by choosing $Q$ freely, and the latter is possible since $P$ transports $(n+1)$-excess inside $R_{k-1}$. \qed
\end{proof}

\begin{lemma}
\label{lem:TransportsExcessSolitaire}
The solitaire preserves the set of patterns that transport excess.
\end{lemma}

\begin{proof}
Suppose $P$ transports excess, and suppose $Q \rightarrow^* P$. Let $Q \subset R \sqcup Q \subset \varphi(Q)$. By Lemma~\ref{lem:MonotoneUpToPermutation}, we have $Q \cup R \rightarrow^* P'$ for some pattern $P_R$ containing $P$. Since $P$ transports excess, all such $P_R$ with equal cardinality are in the same orbit, therefore all patterns $Q \cup R$ with the same cardinality are in the same orbit. \qed
\end{proof}

\subsection{Nice solitaire theories}

\begin{definition}
Let $G$ be a group, $C \Subset S \Subset G$. We say $(C, S, G)$ has a \emph{nice solitaire theory} if the following holds: if $P, Q$ have the same filling closure, and neither of them has excess, then $P \rightarrow_{C, S}^* Q$.
\end{definition}

In our case studies below, we will show that $G = \Z^2$ has nice solitaire theory for $C = S$ either the triangle shape $\{(0,0), (1,0), (0,1)\}$ or the square shape $\{0,1\}^2$. In both cases, excess can also be fully understood.

Note that, at least if $G$ has finite fillings, then it is never true that all patterns with the same filling closure and the same excess are in the same orbit. Namely, let $Q$ be any pattern with excess, and $P$ any smaller pattern with $\varphi(P) = \varphi(Q)$ (for example, $P = S \setminus \{c\}$, $Q = S$ for some $c \in C$). Then it is easy to see that for $g \in G$ ``large'' enough, the filling closures of $gQ \cup P$ and $gP \cup Q$ are equal, and both have the same excess, but they are not in the same solitaire orbit.

\begin{example}
    Let $k>1$ and $l>2k+2$ be integers. Let $G$ be a group with a subgroup $H$ isomorphic to $\Z/k\Z \times \Z/l\Z$, i.e.\ a torus with one side sufficiently larger than the other. Denote $a$ and $b$ a pair of generators of $H$. Then $S = \{e, a, b\}$ does not have a nice solitaire theory.

    Indeed, consider $P = \langle a \rangle$ and $Q = b^{\lfloor\frac{l}{2}\rfloor}\langle a \rangle$. Moving $S$ along $\langle a \rangle$, one can perform filling steps to add $b\langle a \rangle$ to $P$. By induction, $H \subset \varphi(P)$, and in fact $H = \varphi(P)$. Similarly, $\varphi(Q) = H$.
    
    Neither of the patterns has excess. This can be seen as follows: If we have any set $R \subset H$ with less than $k$ points, then one can see analogously to the arguments we give for $\Z^2$ that the filling closure will also involve less than $k$ cosets of $\langle b \rangle$. In particular, the filling closure cannot be $H$.
    
    However, $\mathcal{O}(P) \subset \langle a \rangle(e + b + b^2 + \ldots b^{k-1})$ since to move a point along $b$, there need to be another point to ``push'' it, and similarly $\mathcal{O}(Q) \subset \langle a \rangle b^{\frac{l}2}(e + b + b^2 + \ldots b^{k-1})$ so $P$ and $Q$ cannot be in the same orbit.\qee
\end{example}



\section{Triangle solitaire on the plane}
\label{sec:Triangle}

The \emph{triangle} is the set $T = \{(0,0), (1,0), (0,1)\}$. A $T$-solitaire move is called a \emph{triangle move}, and the $T$-solitaire process is referred to as \emph{triangle solitaire}.

We have $T = T_2$ where $T_n = \{(a, b) \in \{0,\ldots,n-1\}^2 \;|\; a + b \leqslant n-1\}$ is the more general \emph{$n$-triangle}. Sometimes, we refer to a generic $n$-triangle as simply a triangle, and ``the triangle'' or ``triangle shape'' to refer to $T$.

In this section, we completely characterize the orbits of the triangle solitaire, and prove various things about our favorite orbit, namely \emph{line orbit}, which is the orbit of the \emph{line} $L_n = \{0, \ldots, n-1\} \times \{0\}$. Sometimes ``a line'' can also refer to a translate of the line. By the \emph{edges} of the $n$-triangle we refer to the intersections of the edges of its convex hull with the lattice $\Z^2$; the line is one of the three edges of the $n$-triangle.

The set $L_n$ has the $n$-triangle as its filling closure. One of our main results is that all sets of cardinality at most $n$ whose filling closure is $T_n$ are in the solitaire orbit of $L_n$.  

The results of this section were already stated in our previous paper \cite{SaSc23}, with elementary and triangle solitaire specific proofs. Here, we rely on the general theory developed above to deduce immediately some results and make other proofs faster.

\subsection{Definitions and first properties of the triangle solitaire}

A useful graph structure to consider on $\Z^2$ is the triangular lattice. Specifically, the \emph{neighborhood} of a point $x \in \Z^2$ is the one depicted in Figure \ref{fig:nbh}, it corresponds to the points which can be involved in a triangle move with $x$. The neighborhood of a pattern $A$, denoted $N(A)$ is the union of the neighborhoods of its points. Two patterns A and B \emph{touch} if $A \cap N(B) \neq \varnothing$ or equivalently $B\cap N(A) \neq \varnothing$. 
\begin{figure}[ht]
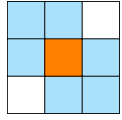

    \centering
    \includestandalone[height = 1.5cm]{Figures/Triangle/fig_neighborhood}
    \caption{The neighbors of the orange cell are the blue ones.}
    \label{fig:nbh}
\end{figure}
\spacc{}

Note that by squishing vertically by a factor of $1/\sqrt{2}$, and then shearing horizontal lines by $0.5$, the triangular lattice becomes invariant under rotation by $120$ degrees, and triangles become equilateral. Although we find it more convenient to work with coordinates in $\Z^2$, we apply this observation in symmetry considerations.

Note that when expressed in coordinates $\Z^2$, the $120$-degree rotation counterclockwise corresponds to the automorphism of $\Z^2$ given by the matrix $M = \left(\begin{smallmatrix} -1 & -1 \\ 1 & 0\end{smallmatrix}\right)$ (when multiplying column vectors from the left). Formally, we may take the phrase ``up to rotation'' to refer to an applications of $M$ or $M^{-1}$.

We begin with the observation that the choice of orientation of the line does not matter:

\begin{proposition}
\label{prop:lines}
    For every $n$, the three edges of $T_n$ are in the same orbit.
\end{proposition}

\begin{proof}
    The first line of Figure \ref{fig:lines_transformation} explains by example how to transform the horizontal edge into the diagonal, and the second one how to transform the diagonal into the vertical edge (of course this is just a rotated inverse of the first transformation). 
\begin{figure}[ht]
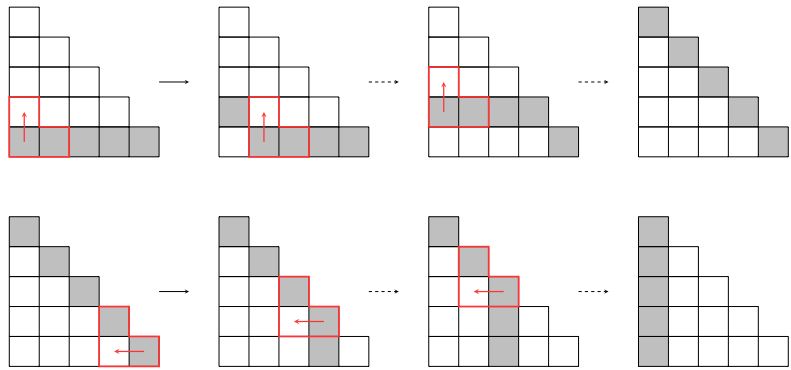

    \centering
    \includestandalone[width=0.85\textwidth]{Figures/Triangle/fig_lines}
    \caption{How to transform one edge into another for $n=5$.}
    \label{fig:lines_transformation}
\end{figure}
\spacc{}
\qed\end{proof}

This proof shows that the number of solitaire moves needed to go from one line to another is at most $O(n^2)$. We will see later that $\Omega(n^2)$ steps are necessary, and that the diameter of the orbit of the line, i.e. the maximal number of steps needed to go from one pattern of the orbit to another, is $\Omega(n^3)$.

\subsection{The filling process for the triangle}

We now specialize the filling process for the triangle. This specific process has been studied previously, in \cite{OEISfillmatrices}. An example of triangle filling was given in Figure~\ref{fig:fill process}


Since $\Z^2$ has the finite filling property and $T$ is non-linear, the filling process is terminating, and thus the filling closure $\varphi(P)$ is reached in finitely many steps. 

Let us say that $P$ \emph{fills} if $\varphi(P) = T_n$, where $n = |P|$. In \cite{OEISfillmatrices}, filling sets $P$ were called \emph{fill matrices}. We will show that they in fact correspond to the elements of the line orbit.



\begin{lemma}
\label{lem:fill_shape}
    For any pattern $P$, there are unique integers $k_1, \ldots, k_r$ and vectors $\vec{v_1}, \ldots, \vec{v_r}$ such that $\varphi(P) = \bigcup_{i = 1}^r \vec{v_i} + T_{k_i}$, $\sum_{i = 1}^r k_i \leqslant |P|$ and $N(\vec{v_i}+T_{k_i}) \cap (\vec{v_j}+T_{k_j}) = \varnothing$ for each $i \neq j$.
\end{lemma}

We refer to $\varphi(P) = \bigcup_{i = 1}^r \vec{v_i} + T_{k_i}$ as the \emph{fill decomposition} of $P$. (Note that this refers to the formal union rather than the set, but it is easy to see that the set determines the values $v_i$ and $k_i$ up to a permutation.)

\begin{proof}
    We prove this by induction on $|P|$. The case $|P| = 1$ is trivial.
    
    Now assume the result is true for patterns of size at most $n$ and let $P$ be a pattern of size $n +1$. Then if $x \in P$, $P\setminus\{x\}$ satisfies the induction hypothesis so we can write $\varphi(P \setminus \{x\}) = \bigcup_{i= 1}^r \vec{v_i} + T_{k_i}$. We now have three cases to consider. 
    
    First, if $x \in \varphi(P\setminus \{x\})$ then $\varphi(P) = \varphi(P\setminus \{x\})$, and the inequality on the size of $P$ clearly continues to hold, as the right-hand size increase, but the left-hand side does not. 
    
    Second, if $x$ is not in the neighborhood of $\varphi(P \setminus \{x\})$ then no additional filling can be done with it, therefore $\varphi(P) = \varphi(P \setminus \{x\}) \cup \{x\}$ and, as $\{x\}$ is a triangle (translate of $T_1$), we have the appropriate decomposition, and both sides of the inequality are increased by $1$.
    
    Finally, assume $x \in N(\vec{v_1} + T_{k_1})$. Then we can extend $\vec{v_1} + T_{k_1}$ as in Figure \ref{fig:fill_triangle}. By doing so we may lose the property that the triangles do not touch, but if some do so we can merge them by repeating the extension process. Notice that if two triangles are merged, then the new triangle cannot be larger than the sum of the sizes of the initial triangle, so the inequality on the triangles' sizes is still satisfied. (Merges may be triggered recursively, but nevertheless no merge can increase the sum of triangle sizes.)
    \begin{figure}[ht]
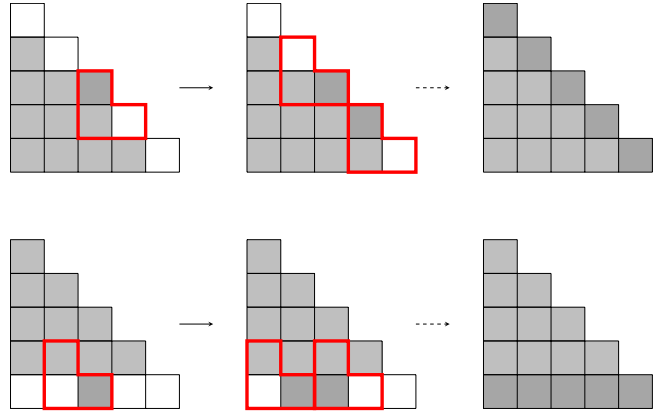

        \centering
        \includestandalone[width = 0.7\textwidth]{Figures/Triangle/fig_filling_shape}
        \caption{How to extend a triangle with a top neighbor or a subdiagonal neighbor. The right neighbor case is symmetric to the top neighbor case.}
        \label{fig:fill_triangle}
    \end{figure}
    \spacc{}
\qed\end{proof}

Define the \emph{triangle excess} of $P$ as the difference $e_T(P) = |P| - \sum_{i=1}^r k_i$.

\begin{lemma}
Triangle excess is equal to excess.
\end{lemma}

\begin{proof}
Excess is defined as
\[ e(P) = |P| - \rank(\varphi(P)), \]
and triangle excess is defined as
\[ e_T(P) = |P| - \sum_{i=1}^r k_i. \]
where $k_i$ come from the fill decomposition of $P$. To show these equivalent is the same as showing $\rank(\varphi(P)) = \sum_{i=1}^r k_i$.

Since $\varphi(\bigcup_{i=1}^r \vec v_i + L_{k_i}) = \varphi(P)$, we have $\rank(\varphi(P)) \leqslant \sum_{i=1}^r k_i$. On the other hand, if we have fill decomposition $\varphi(R) = \varphi(P) = \bigcup_i \vec{v}_i + T_{k_i}$ then $|R| \geqslant \sum_i k_i$ by Lemma~\ref{lem:fill_shape}, so indeed $\rank(\varphi(P)) \geqslant \sum_i k_i$ as desired. \qed
\end{proof}




\subsection{Characterisation of the solitaire orbit through the filling}

We can mimic the filling process using the solitaire, by using lines in place of triangles, and pretending the excess elements are not there, as in Section~\ref{sec:Pretending}.

In line with the terminology in Section~\ref{sec:TransportingExcess}, we say a pattern $P$ is a \emph{superline} if $\varphi(P) = \vec v + T_n$ for some $n$, and $P \supset \vec v + L_n$. 

\begin{lemma}
\label{lem:ExtendSuperline}
Let $P$ be a superline, and suppose $x$ touches $P$. Then $P \cup \{ x\}$ contains a superline in its solitaire orbit.
\end{lemma}

\begin{proof}
Suppose $P$ contains $\vec{v} + L_m$ and $\varphi(P) = \vec{v} + T_m$. By translational symmetry we may assume $\vec v = (0,0)$. First suppose the exact equality $P = L_m$ holds. If $x \in T_m$, the claim is clear as $P$ is already a superline.

Otherwise, up to rotational symmetry (and because of Proposition~\ref{prop:lines}), we may assume $x$ is below $P$, i.e.\ $x = (k, -1)$ for $0 \leqslant k \leqslant m$. Figure~\ref{fig:merge_line} shows how we can turn $P \cup \{x\}$ into the line $(0,-1) + L_{m+1} = \{0, \ldots, m\} \times \{-1\}$ by applying solitaire moves, again proving the claim. 

    \begin{figure}[ht]
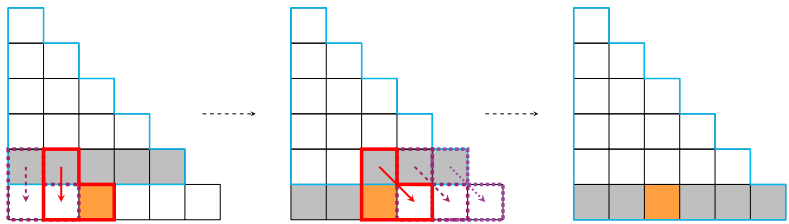

        \centering
        \includestandalone[width=0.85\textwidth]{Figures/Triangle/fig_merge_lines}
        \caption{How to extend a line with a bottom neighbor.}
        \label{fig:merge_line}
    \end{figure}
    \spacc{}

Now consider the general case $P \supset L_m$ and $\varphi(P) = T_m$, where we may again assume $x \notin T_m$. We have $\varphi(P) = \varphi(L_m)$ so $\varphi(P \cup \{x\}) = \varphi(L_m \cup \{x\})$ by Lemma~\ref{lem:ChangeSubset}. Denote $T' = \varphi(L_m \cup \{x\})$.

We have $L_m \cup \{x\} \rightarrow^* Q$ for some superline $Q$ with of course $\varphi(Q) = T'$, so by Lemma~\ref{lem:MonotonicityUpToPermutation} we have $P \cup \{x\} \rightarrow^* Q'$ for some pattern $Q' \supset Q$, and $\varphi(Q') = T'$. 
Thus $Q'$ is also a superline. \qed
\end{proof}

\begin{lemma}
\label{lem:merging}
Suppose that $P, Q$ are (not necessarily disjoint) superlines, and their filling closures touch. Then $P \cup Q$ contains a superline in its solitaire orbit.
\end{lemma}

\begin{proof}
Up to rotational symmetry, translation, and exchanging the roles of $P$ and $Q$, we may assume $\varphi(P) = T_m$ and $\varphi(Q) = \vec v + T_n$ where $\vec v = (a, b)$ with $a \geqslant 0$, $b \leqslant -1$, and $n + b \leqslant m - a$. To see this, we recommend the reader draw a picture and argue geometrically: if $n+b > m-a$, then the top corner of $Q$ sticks above the diagonal of $P$, in which case the right corner of $P$ will be inside $Q$ or just next to it, and we may exchange the roles of $P$ and rotate and translate to get to the situation $n + b \leqslant m - a$. 

Since the filling closures touch, we have $n \geqslant -b$ and $a \leqslant m$. Since $Q$ is a superline and all edges of the triangle are in the same orbit, $Q' = \vec v + \{0\} \times \{0, \ldots, n-1\}$ is contained in some pattern in the orbit of $Q$. Note that the assumption on $n$ is precisely equivalent to the fact that vectors in $Q'$ with positive second coordinate are contained in $T_n$. Now we can simply apply the previous lemma successively to the vectors in $Q'$ with negative second coordinate. \qed
\end{proof}

\begin{lemma}
Let $P$ be a pattern with fill decomposition $\varphi(P) = \bigcup_i \vec{v}_i + T_{k_i}$. Then the orbit of $P$ contains a union of superlines whose filling closures do not touch, and where the corresponding lines are precisely the lines $\vec{v}_i + L_{k_i}$. 
\end{lemma}

\begin{proof}
We start by observing that every pattern $P$ can be thought of as a union of superlines $\bigcup_i \vec{v}_i + T_{k_i}$, for example by taking $k_i = 1$ for all $i$, and having the $\vec v_i$ enumerate $P$.

Let $P' = \bigcup_{i=1}^r Q_i$ be in the orbit of $P$ with $Q_i$ superlines, such that $r$ is minimal. Then the filling closures of distinct $Q_i$ do not touch: suppose $\varphi(Q_j)$ and $\varphi(Q_k)$ do touch for $j \neq k$. Then by Lemma~\ref{lem:merging}, $Q_j \cup Q_k$ contains a superline $Q$ in its solitaire orbit. By Lemma~\ref{lem:MonotoneUpToPermutation}, $\bigcup_{i=1}^r Q_i$ then contains a set $R \supset Q' = Q \cup \bigcup_{i \in [1, r] \setminus \{j, k\}} Q_i$ in it solitaire orbit. Each element of $R$ which is not in $Q'$ can be added in one of the superlines $Q_i$ or $Q$, since the union of their filling closures is $\varphi(P)$. Thus, we have found a pattern in the orbit of $P$ which can be written as a union of strictly less than $r$ superlines.

For the final claim, we observe that the filling closure of $P$ must be precisely the union of the filling closures of the superlines, which determines the lines to be those corresponding to the fill decomposition. \qed
\end{proof}

Next, we show that the non-touching superlines can be put in normal form.

\begin{lemma}
The line transports excess.
\end{lemma}

\begin{proof}
By
Lemma~\ref{lem:Ordering} it suffices to show that there is a linear order $p_1, p_2, \ldots, p_m$ on $\varphi(L_n) \setminus L_n = T_n \setminus L_n$ such that $L_n$ transports $1$-excess in $R_k = P \cup \{p_1, \ldots, p_k\}$ for all $k$. Note that $m = |T_n \setminus L_n| = n(n-1)/2$. We use the order $(a, b) \leqslant (c, d)$ if $b < d$ or $b = d \wedge a < c$. Note that extending this formula to all of $T_n$, $L_n = \{0,\ldots, n-1\} \times \{0\}$ would form the minimal elements.

Figure~\ref{fig:fetch excess} now explains the procedure for retrieving an element from anywhere in the triangle to a position of the form $(k, 1)$, without involving any positions larger than the retrived element. Figure~\ref{fig:push excess} in turn explains how to move an element from position $(k, 1)$ to position $(0,1)$ (the reverse can be done by inverting this process). This concludes the proof.\qed
\end{proof}

    \begin{figure}[ht]
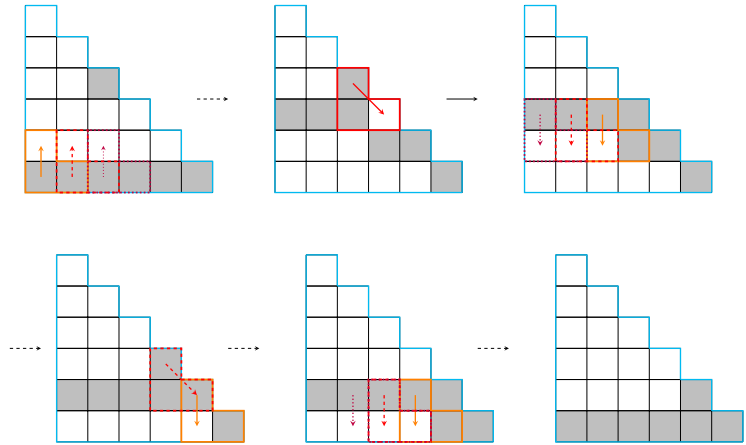

        \centering
        \includestandalone[width=0.8\textwidth]{Figures/Triangle/fig_fetch_excess}
        \caption{Fetching an excess point.}
        \label{fig:fetch excess}
    \end{figure}
    \spacc{}
    
        \begin{figure}[ht]
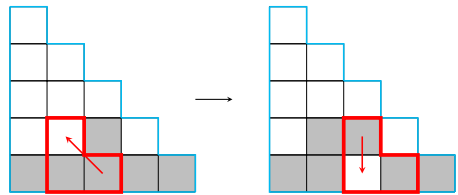

        \centering
        \includestandalone[height=2.5cm]{Figures/Triangle/fig_push_excess}
        \caption{Pushing excess to the left.} 
        \label{fig:push excess}
    \end{figure}
    \spacc{}
    
    Since the line transports excess, we obtain that all superlines with the same filling closure and same cardinality are in the same orbit. The representatives informed by the previous proof are as follows: for $0 \leqslant k \leqslant n(n-1)/2$, we denote by $P_{n, k}$ the shape composed of a line of length $n$ to which $k$ points were added by filling the triangle under the line from top to bottom, and at each height from left to right. Examples are shown in Figure~\ref{fig:Pnk}. 
\begin{figure}[ht]
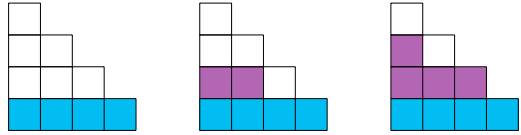

    \centering
    \includestandalone[height = 1.7cm]{Figures/Triangle/fig_Pnk}
    \caption{From left to right: $P_{4, 0}$, $P_{4, 2}$ and $P_{4, 4}$. The purple cells are the excess.}
    \label{fig:Pnk}
\end{figure}
\spacc{}

\begin{theorem}[Characterisation of the orbits]
\label{thm:orbits_charac}
    If $P$ is a finite pattern then there are integers $n_1, \ldots n_r$ and $k_1, \ldots k_r$ and vectors $\vec{v_1}, \ldots \vec{v_r}$ (uniquely determined by $P$) such that $P \rightarrow^* \bigcup_{i = 1}^r (\vec{v_i} + P_{n_i, k_i})$, and the filling closures $\varphi(P_{n_i, k_i} + \vec{v_i}) = T_{n_i}+\vec{v_i}$ do not touch each other. Furthermore, $\sum_{i=1}^r n_i = |P| - e(P)$ and $\sum_{i=1}^r k_i = e(P)$.
\end{theorem}

\begin{proof}
The claim $P \rightarrow^* \bigcup_{i = 1}^r (\vec{v_i} + P_{n_i, k_i})$ follows by first observing that the orbit contains a union of superlines whose filling closures do not touch, and then using the fact the line transports excess to put each superline in normal form $P_{n_i, k_i}$. For uniqueness, observe that (up to permutation) the vectors $\vec v_i$ and values $n_i$ are determined by the fill decomposition, and $k_i$ by the cardinalities of the sets $P \cap \vec{v_i} + T_{k_i}$.

For the last claim, we recall that $e(P)$ is equal to the triangle excess defined as $e_T(P) = |P| - \sum_{i=1}^r n_i$, thus $\sum_{i=1}^r n_i = |P| - e(P)$. Since solitaire preserves cardinality,
\[ |P| = |\bigcup_{i = 1}^r (\vec{v_i} + P_{n_i, k_i})| = \sum_i (n_i + k_i) \]
so $e(P) = \sum_i k_i$. \qed
\end{proof}

\begin{corollary}
\label{cor:line_orbit}
Let $P$ be a pattern with fill decomposition $ \varphi(P) = \bigcup_{i = 1}^r \vec{v_i} + T_{k_i}$.
Then $P$ has no excess if and only if $P$ is in the solitaire orbit of $\bigcup_i \vec v_i + L_i$.
\end{corollary}

\begin{corollary}
\label{cor:excess_orbit}
    If $P$ is a pattern, then $P \in \orb(P_{n, k})$ if and only if $\varphi(P) = T_n$ and $e(P) = k$.
\end{corollary}

Note that we now know the orbits for solitaire processes on the plane $\Z^2$ with all shapes of size 3. Indeed, such a shape is either linear (and thus easy to analyse), 
or it is a triangle shape on a finite index subgroup of $\mathbb{Z}^2$, and the orbits in different cosets of this subgroup are completely independent (Remark~\ref{rem:fg}), and are individually described by the triangle solitaire.

\subsection{Size of the line orbit}
\label{sec:size line orbit}

In what follows, we build a subset and a superset of the line orbit and compute their size so as to get bounds on the size of the line orbit.

Lets first build a subset of $L_n$. One can build an element of the orbit of the line by choosing first one of the three corners and then choosing a point on each line parallel to the edge opposed to the corner as illustrated in the left part of Figure~\ref{fig:lower bound}. Such pattern is called a horizontal, vertical or diagonal \emph{stack} depending on if the chosen corner is the top one, right one or bottom left one. One can easily check by induction on $n$ that thos patterns are in the line orbit.

\begin{figure}[ht]
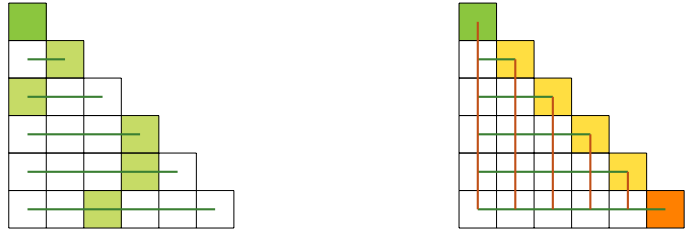

    \centering
    \includestandalone[height = 3cm]{Figures/Triangle/3nf}
    \caption{Left: An example of vertical stack. Right: A pattern that is two kind of stacks is a line.}
    \label{fig:lower bound}
\end{figure}

Assume that a pattern $P$ is for instance both a vertical and a horizontal stack, then $P$ has exactly one point in each line and in each column, so it is the diagonal line (see the right side of Figure~\ref{fig:lower bound}). Therefore, when counting all three kinds of stacks, only $3$ patterns are counted twice. Since there are $n!$ stacks of one kind, there are exactly $3n!-3$ stack patterns. Using Stirling's approximation, this gives the following lower bound on the size of the line orbit:

\begin{lemma}
    For all $n \geqslant 1$, $|\mathcal{O}(L_n)| \geqslant 3n!-3 \geqslant c_1e^{-n}n^{n + \frac12}$ for a certain constant $c_1 >0$.
\end{lemma}

Now let us build a superset of $L_n$. Let $A_n$ be the set of patterns $P$ with size $n$ contained in the triangle $T_n$ such that for each $1 \leqslant j \leqslant n$, the number of points of $P$ in the $j$ right-most columns of $T_n$ is at most $j$. $L_n \subset A_n$ since if a pattern $P$ is not in $A_n$, then there is a $j$ such that the subpattern of $P$ composed of the $j$ right-most columns of $T_n$ has excess, so so does $P$ (see Figure~\ref{fig:upper bound}). 

\begin{figure}[ht]
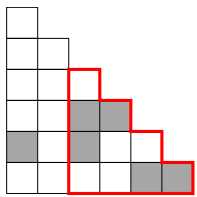

    \centering
    \includestandalone[height = 2.5cm]{Figures/Triangle/upper_bound}
    \caption{A pattern with $5$ points in the $4$ rightmost columns. The subpattern delimited by the red line has excess so so does the whole pattern.}
    \label{fig:upper bound}
\end{figure}

The sets $A_n$ have been studied by \cite{OEISupperbound}, who proved that their size is
\[ c\left(\frac{e}2\right)^n(n-1)^{n-\frac52} \]
with $c = \frac{4 + 2W(-2e^{-2})}{e^3\sqrt{2\pi}}$ where $W$ is the inverse function of $f : z \mapsto ze^z$ (Lambert's W function). Since the line orbit is a subset of $A_n$, this equivalent translates to an upper bound on its size up to increasing the constant:

\begin{lemma}
    For all $n \geqslant 1$, $|\mathcal{O}(L_n)| \leqslant c_2 \left(\frac{e}2\right)^n(n-1)^{n-\frac52}$ for a certain constant $c_2 >0$.
\end{lemma}

Combining these, we get the following bounds on the size of the line orbit.

\begin{theorem}
\label{thm:orbit size}
    There are constants $c_1, c_2 >0$ such that $c_1e^{-n}n^{n + \frac12} \leqslant |\mathcal{O}(L_n)| \leqslant c_2 \left(\frac{e}2\right)^n(n-1)^{n-\frac52}$.
\end{theorem}

\subsection{Excess sets for the triangle}
\label{sec:TriangleExcess}

We show that excess sets do not behave as intuitively as one might hope, even for our simple triangle example. Namely, it is tempting to think that if a set has excess, then we can remove some of its points to remove the excess, i.e.\ that the size of a maximal excess set in $P$ always matches the excess $e(P)$. This is not true.




\begin{theorem}
For the triangle solitaire,
\begin{enumerate}[label={(\roman*)}]
\item there exist patterns with arbitrarily large excess, which contain no excess sets (in particular, phantom excess is unbounded),
\item visible excess (equivalently phantom excess) can vary by an unbounded amount within one solitaire orbit,
\item there exist patterns where there are maximal excess set of different cardinality.
\end{enumerate}
\end{theorem}

\begin{proof}
For the first claim, consider the pattern $P$ in Figure~\ref{fig:no excess point}. It is clear that its filling closure is the triangle, so $e(P) = 1$. However, this pattern has no nonempty excess sets. To see this, observe that the process of joining triangles always merges two existing triangles to a larger one, namely two $T_1$s to $T_2$, then two $T_2$s to $T_4$, and then two $T_4$s to $T_7$. If any element is removed, then one of these pairs cannot be joined.

\begin{figure}[ht]
    \centering
    \includestandalone[height=2.5cm]{Figures/Triangle/excess_no_excess_point}
    \caption{Here, $e(P) = 1$ but $E(P) = \{ \varnothing \}$ so $\hat E(P) = 0$.}
    \label{fig:no excess point}
\end{figure}

We see that in fact when any element is removed, the entire bottom left corner of the triangle is left empty. This allows us to continue the construction inductively: We can add another triangle below whose \emph{top} does not become filled if any element is removed. Then together these triangles will fill a large triangle to the bottom right from $P$, but if any element is removed, then this triangle will stay empty in the filling closure. We obtain more and more excess at each step, but still have no nonempty excess sets (see Figure~\ref{fig:still no excess point}).

\begin{figure}[ht]
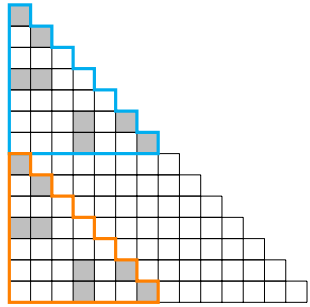

    \centering
    \vspace*{-0.3cm}
    \includestandalone[height=4cm]{Figures/Triangle/more_excess}
    \caption{The pattern obtained after one induction step. The patterns highlighted are the two smaller patterns combined. $e(P) = 2$ but we still have $E(P) = \{ \varnothing \}$.}
    \label{fig:still no excess point}
\end{figure}

For the second claim, observe that the pattern we constructed (with arbitrarily large phantom excess) has in its orbit a pattern containing the corresponding line, and by Lemma~\ref{lem:ContainsNoPhantom} any such pattern has excess equal to its visible excess, since the line itself has no excess.

For the third claim, consider Figure~\ref{fig:max excess set}.
\begin{figure}[ht]
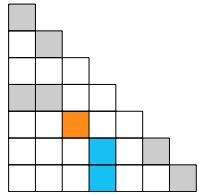

    \centering
    \includestandalone[height=2.5cm]{Figures/Triangle/max_excess_set}
    \caption{The blue and orange sets are both maximal excess sets but do not have the same cardinality.}
    \label{fig:max excess set}
\end{figure}
\qed
\end{proof}


\section{A brief look at other shapes on the plane}
\label{sec:Plane}

\subsection{The square shape}
\label{sec:Square}

In this section we study the solitaire action induced by the $2 \times 2$ square shape $\{0,1\}^2$. We show that the theory can be developed entirely analogously as for the triangle. 

The action of the square rotates a square with a single empty cell as depicted in Figure \ref{fig:square action}.
\begin{figure}[ht]
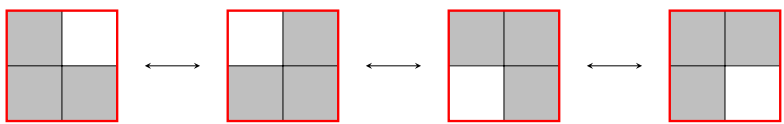

    \centering
    \includestandalone[height = 1.5cm]{Figures/Square/square_action}
    \caption{The action of the square shape.}
    \label{fig:square action}
\end{figure}

The filling process for the square shape of course completes squares which are missing a single point. By analogy with the triangle shape, we expect the orbit of a finite pattern to be contained in non touching rectangles which can be computed by filling the initial pattern, and those orbits to be recognisable with only the shape of the filling and by counting the amount of excess elements in each rectangle.

With respect to the square, a good notion of \emph{neighbors} of a point are the eight points surrounding it, and the neighborhood of a pattern is the set of points which have two adjacent points of the pattern as neighbors. (Note that this is precisely the condition for being able to apply a square solitaire move.)

Examples are given in Figure \ref{fig:square nbh}.
\begin{figure}[ht]
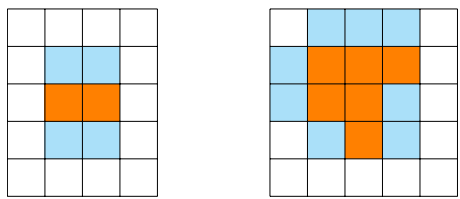

    \centering
    \includestandalone[height = 2.5cm]{Figures/Square/square_neighbors}
    \caption{The neighbors of the orange pattern are the blue cells.}
    \label{fig:square nbh}
\end{figure}
The notion of touching is then again defined as $A$ touches $B$ if $A \cap (B\cup N(B)) \neq \varnothing$.

\subsubsection{Rectangles and crosses.}

As already indicated, the natural analog of the triangle $T_n$ for the square solitaire is a \emph{rectangle} $R_{m,n} = \{0, \ldots, m-1\} \times \{0, \ldots, n-1\}$. A natural analog of a line $L_n$ for square solitaire is in turn a horizontal line intersecting a vertical line. We call such a pattern a \emph{cross}. Notice that the square-shape filling of a cross is a rectangle of width the length of the horizontal line and height the length of the vertical line. 

\begin{lemma}
\label{lem:movecross}
    All crosses that fill the same rectangle are in the same orbit.
\end{lemma}

\begin{proof}
    Using the moves described in Figure \ref{fig:move cross}, one can shift the vertical line to the left. By reversing the moves, the line shifts to the right. By rotating them, the horizontal line can be shifted up or down.
    \begin{figure}[ht]
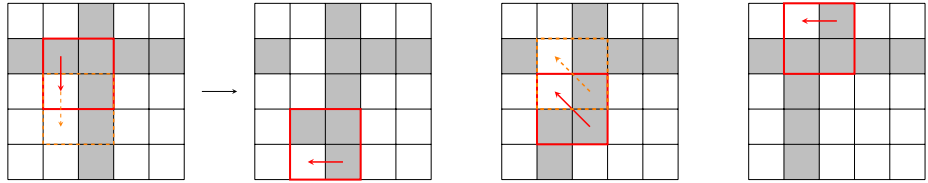

    \centering
    \includestandalone[width = \textwidth]{Figures/Square/move_cross}
    \caption{How to shift the vertical line of a cross to the left. All the other movements are symmetrical.}
    \label{fig:move cross}
    \end{figure}
\qed\end{proof}

\subsubsection{Shape of the square filling}

\begin{lemma}
    With respect to the square shape, the filling of any pattern $P$ is a set of non touching rectangles $R_1, \ldots, R_k$ whose sizes satisfy $\sum_{i=1}^k (w(R_i) + h(R_i) -1) \leqslant |P|$, where $w(R_i)$ and $h(R_i)$ are respectively the width and height of $R_i$.
\end{lemma}

Again, we refer to $\bigcup R_i$ as the fill decomposition of $P$, and its uniqueness is clear.

\begin{proof}
    Let $P$ be a pattern. First, notice that a maximal connected subpattern $Q$ of $\varphi(P)$ has no neighbor because otherwise a cell could be filled.
    
    Now consider a maximal connected subpattern $Q$ of $\varphi(P)$ that is not a rectangle. Then $Q$ has an inner corner, thus $Q$ has a neighbor. Hence the result on the shape of the filling.
    
    We now prove the inequality on the size of the rectangles by induction on $|P|$.
    
    If $P$ is a single point, then $\varphi(P) = P$ is a rectangle of length and height 1 so the inequality holds.
    
    Now assume the inequality holds for all patterns of size $n$, and let $P$ be a pattern of size $n+1$. Let $x \in P$, $|P \setminus \{x\}| = n$, so the rectangles that compose $\varphi(P \setminus \{x\})$ satisfy the inequality. If $x$ is in one of those rectangles, then $\varphi(P) = \varphi(P \setminus \{x\})$, and we have the inequality. If $x$ is not in the neighborhood of a rectangle then $\varphi(P) = \varphi(P\setminus\{x\}) \cup \{x\}$, and again the inequality holds.
    
    Lastly, if $x$ is in the neighborhood of rectangle $R_i$, for example $x$ is adjacent to the right size of $R_i$, then the filling adds a column to rectangle $R_i$, augmenting its length by 1.
    
    The rectangles may intersect, triggering a sequence of merges. But if $R_i$ intersects the neighborhood of another rectangle $R_j$, then the filling merges them into a rectangle $R$ that satisfies $w(R) \leqslant w(R_i) + w(R_j)$ and $h(R) \leqslant h(R_i) + h(R_j)$, and at least one of these inequalities is strict. From this, it follows that the inequality continues to hold when rectangles are merged.
\qed\end{proof}


\subsubsection{Characterisation of the orbits.}

If the filling of pattern $P$ is $\bigsqcup_{i=1}^r R_i$ then its square excess is $e_S(P) = |P| - \sum_{i=1}^r (w(R_i) + h(R_i) - 1)$. We can see as in the case of the triangle, that square excess is equal to excess in the abstract sense.

For $a, b, k$ integers, let $L_{a, b, k}$ be the L-shape of length $a$ and height $b$ with $k$ points in the rectangle in generates, filled from left to right and bottom to top. Examples are given Figure \ref{fig:Rabk}.

\begin{figure}[ht]
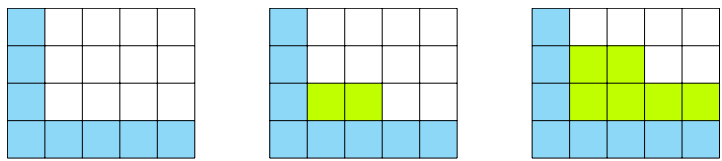

    \centering
    \includestandalone[height = 2cm]{Figures/Square/Rabk}
    \caption{From left to right, $L_{5, 4, 0}, L_{5, 4, 2}$ and $L_{5, 4, 6}$.}
    \label{fig:Rabk}
\end{figure}

\begin{lemma}
Every cross transports excess.
\end{lemma}

\begin{proof}
By Lemma~\ref{lem:TransportsExcessSolitaire} and Lemma~\ref{lem:movecross}, it suffices to show this for one cross. We use the cross with the left and bottom edges of the rectangle, i.e.\ $L_{m, n, 0}$.

Excess transportation can be now shown using Lemma~\ref{lem:Ordering} by using the lexicographic order, and is even easier than the corresponding process for the triangle. Namely, by the procedure in Lemma~\ref{lem:movecross} we can move the vertical line of a cross so that it is next to an excess point, then move it down, then move the vertical line back to the left, and transport the excess point to get the final pattern $L_{m,n,1}$.

This process only uses elements below the excess point in lexicographic order, and we conclude from Lemma~\ref{lem:Ordering} that the cross transports excess. \qed
\end{proof}

\begin{theorem}[Characterisation of the orbits of the square]
\label{thm : square orbit}
    If $P$ is a pattern, then it is in the orbit of the pattern composed of $L_{a, b, k}$-shapes 
    corresponding to the fill decomposition of $\varphi(P)$ and adding the excess present in the rectangle. I.e., if $\varphi(P) = \bigcup_{i=1}^k [(x_i,y_i), (x_i + a_i,y_i)]\times [(x_i,y_i), (x_i,y_i + b_i)]$ and $k_i = |P \cap R_i| - (a_i + b_i -1)$ then $P \in \orb\left(\bigcup_{i=1}^k {(x_i, y_i)} + L_{a_i, b_i, k_i}\right)$.
\end{theorem}

\begin{proof}
This is proved analogously as for the triangle: we first show that we can apply the solitaire process to turn $P$ into a disjoint union of supercrosses (patterns containing a cross, whose filling closure is the same as that of the cross) whose filling closures do not touch. The union of the filling closures of these super-crosses form the filling closure of $P$.

For this, we observe analogously to the case of the triangle, that two touching supercrosses can be merged into a supercross. Thus, a pattern in the orbit of $P$ which is a union of a minimal number of supercrosses will in fact be a union of non-touching disjoint supercrosses.

Finally, we can use the previous lemma (the excess transportation property of crosses) to put the super-crosses in the stated normal form.
\qed\end{proof}

\subsubsection{Size of a cross orbit}

One may be able to build subsets and supersets of a cross orbit with known cardinality in the same fashion as for the triangle shape. 

For the subset, a good candidate is the set of patterns built recursively by adding a point on the column to the right or the line to the top. More precisely, define $\mathcal{R}_{1,1} = \{(0,0)\}$ and then for $a,b>0$, patterns of $\mathcal{R}_{a+1,b}$ are obtained by adding a point in $\{(a, j) \mid 0 \leqslant j < b\}$ to a pattern of  $\mathcal{R}_{a,b}$ and those of $\mathcal{R}_{a,b+1}$ are obtained by adding a point in $\{(i, b) \mid 0 \leqslant i < a\}$ to a pattern of  $\mathcal{R}_{a,b}$. One may check that these patterns are in the orbit of $L_{a,b,0}$. We however do not know how to count them. 

For the superset, one may consider the patterns such that no rectangle subpattern with bottom left corner $(0,0)$ has excess, similarly to the overset for the triangle. It is clear that this is an overset of the cross orbit but again, we do not know how to count them.

\subsection{General shapes $S$ admit interaction without merging}


\subsubsection{Polygons and fillings.}


Consider a general non-linear shape $S \Subset \Z^2$, i.e.\ not contained in any affine line. We say $S \subset \Z^2$ is \emph{discrete-convex} if it is the intersection of a geometrically convex subset of $\mathbb{R}^2$ and $\Z^2$. 

Recall that $P \rightarrow Q$ is an $S$-solitaire move if there is a vector $v \in \Z^2$ such that $|P \cap (\vec{v}+S)| = |Q \cap (\vec{v}+S)| = |S| - 1$ and $P \triangle Q \subset \mathcal{P}_2(\vec{v}+S)$. An $S$-filling step consists in adding to a pattern such that $|P \cap (\vec{v}+S)| = |S|-1$ the missing point of $\vec{v} + S$. This process is confluent and when it exists, we denote its limit $\varphi_S(P)$.

Previously, the filling of a pattern had the ``same shape'' as the shape used for the solitaire moves. The following definition aims at generalising this idea for arbitrary $S$. 

\begin{definition}
Given a shape $S$, an \emph{$S$-polygon} is a pattern obtained by intersecting $\mathbb{Z}^2$ with a geometric polygon in $\mathbb{R}^2$ with the same edge directions as the convex hull of $S$, each edge being at least as long as the corresponding edge in $S$'s convex hull (in terms of the number of lattice points). 
We define the \emph{$S$-hull} of $P$ as $H_S(P) = Q \cap \Z^2$ where $Q \subset \mathbb{R}^2$ is the smallest polygon with (at most) the same edge directions as $S$ containing $P$.
\end{definition}

Notice that an $S$-hull is not necessarily an $S$-polygon since some edges might be too short.

\begin{figure}[ht]
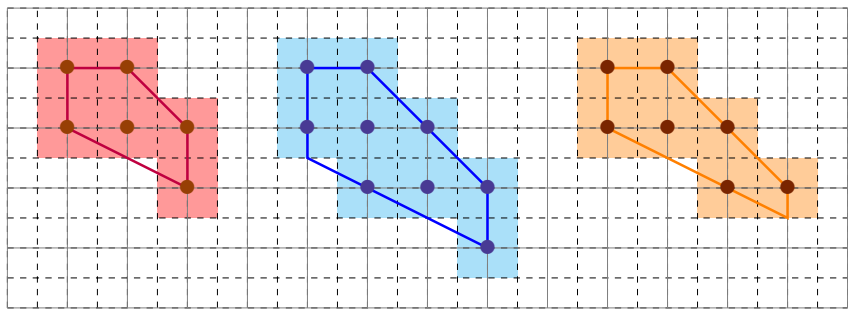

    \centering
    \includestandalone[height = 4cm]{Figures/Convex/polygon}
    \caption{With $S$ the red shape, the blue pattern is an $S$-polygon but not the orange one because the right side is too short. Here the solid lines correspond to standard generators of $\Z^2$, and dotted lines show the dual lattice of $\Z^2$-centered unit squares.}
    \label{fig:polygon}
\end{figure}

The following two lemmas are essentially the two-dimensional version of Proposition~\ref{prop:ZdFFP}. We give direct (slightly informal) proofs.

\begin{lemma}
The $S$-hull of every set is well-defined, and is finite.
\end{lemma}

\begin{proof}
For each edge $e$ of $S$, there is a corresponding normal vector $\vec v$ (pointing outside), and points with positive dot product with $\vec v$ define a half-space $H_e$.

Given $P \subset \mathbb{R}^2$, we can see $P$ as a subset of $\mathbb{R}^2$ and slide copies of all these half-planes  $H_e$ as close to $P$ as possible, and finally take the convex polygon spanned by their edges.
\qed\end{proof}

Unlike in the case of the triangle and square, there nevertheless does not always exist a smallest $S$-polygon containing a given finite set $P \subset \mathbb{Z}^2$, because the edges in the $S$-hull can be too short. This is in particular the case when $P$ is smaller than $S$, an example is given in Figure ~\ref{fig:hull}.

\begin{figure}[ht]
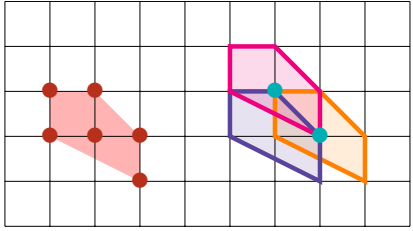

    \centering
    \includestandalone[height = 3cm]{Figures/Convex/S-hulls}
    \caption{With $S$ the red shape, the blue pattern has 3 smallest $S$-polygons containing it.}
    \label{fig:hull}
\end{figure}

\begin{lemma}
\label{lem:fill in hull}
    The $S$-filling of pattern $P$ is contained in its $S$-hull.
\end{lemma}

\begin{proof}
     Let $L$ be a line parallel to an edge on $S$, then the filling process can add points on the outer side of the line only if there is already some point on the outer side. Therefore the filling process cannot add points outside the $S$-hull.
\qed\end{proof}

\begin{remark}
    The notion of $S$-hull can be generalised to higher dimensions, using faces instead of edges. Lemma~\ref{lem:fill in hull} still holds for higher dimension, simply replace lines by hyperplanes in the proof. 
\end{remark}

In the case of the triangle shape and the square shape, we were able to identify a set of domains (triangles and rectangles) and a notion of touching, such that patterns contained in domains that do not touch evolve entirely independently in the solitaire and patterns filling domains that touch can merge into a pattern that fill a bigger domain.



In fact, these sets were precisely $S$-polygons: for a triangle, the sets $T_n$ are precisely the triangle-polygons, and rectangles $\{0, \ldots, m-1\} \times \{0, \ldots, n-1\}$ are precisely the square-polygons.

The following example illustrates that, unlike in the case of the triangle and the square. However, for a general shape $S$ we can have two $S$-polygons sets which do not ``merge'' (in that their union is closed under filling, and is not an $S$-polygon), yet the solitaire process can in a nontrivial way share a hole between the two sets.

\begin{example}
Consider the red shape $S$ in Figure~\ref{fig:weird}, and the pattern $P$ formed by the orange, green and blue points. One can apply solitaire moves to the left (orange and green) of right (blue and green) part separately, moving the hole arbitrarily.

\begin{figure}[ht]
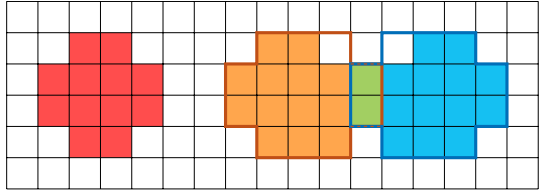

    \centering
    \includestandalone[height = 2.5cm]{Figures/Convex/Weird_merging}
    \caption{An example of ``non-merging shapes'' with non-trivial interaction.}
    \label{fig:weird}
\end{figure}

If we move the hole in the orange part to the left, then we can borrow a green point from the middle. This can freeze the process in the blue part, until the element is returned. 

The filling closure of the colored area is the area with the orange and blue outlines, i.e.\ the filling process adds only two points. This limit does not fit our notion of an $S$-polygon. \qee
\end{example}

\begin{remark}
For large $S$, one can form long chains out of such non-merging $S$-polygons, and even cycles and other planar graphs, where elements can be exchanged between the various parts, but which do not merge into a discrete-convex pattern. 
\end{remark}

We do not yet have clean descriptions of the processes described in the previous example and remark, but we hope to find a good way of cutting fillings into understandable components.

\begin{question}
    Is there a notion of $S$-components such that the maximal $S$-components of a pattern closed under filling are either $S$-polygons or single points (those component being not necessarily disjoint), and the solitaire can be understood as in the triangle and square case, except that distinct $S$-components may share points?
\end{question}

\subsubsection{Contours and Orbits.}

One thing that does work with general shapes is the notion of a contour from \cite{Sa22}. It provides a natural generalization of the line in triangle solitaire, and the cross (or an L-shape) in the case of square solitaire.

Here we find it helpful to work with more general groups $G$, to avoid getting caught up in geometric details, although our particular emphasis will be on the groups $\Z^d$.

Recall that a total order $<$ on (the elements of) a group $G$ is \emph{bi-invariant} if $a < b \iff ca < cb \iff ac < bc$ for all $a, b, c \in G$. Such orders exists for example for free abelian groups and free groups.

\begin{definition}
Let us call $s \in S$ a \emph{corner} if there is a bi-invariant order on $G$ such that $s$ is the maximal element of $S$ (it is then minimal for the inverted order, which is also bi-invariant).\footnote{It seems that here we need a bi-invariant order, while \cite{Sa22} only needs a left-invariant order.}
\end{definition}

The following fact is well-known:

\begin{lemma}
\label{lem:Orders}
The (bi-)invariant orders on $\Z^d$ are fully described as follows: First order vectors according to their dot product with a unit vector $\vec v$. If the kernel of the projection to the line spanned by $\vec v$ has non-trivial intersection with $\Z^d$, then this intersection is a finitely-generated abelian group of smaller rank. Order it inductively.
\end{lemma}

In particular, our notion of corner agrees with the usual notion of a corner in convex geometry. Figure~\ref{fig:border} presents few a examples on $\Z^2$.

\begin{definition}
Suppose $G$ is a group, $S \Subset G$ and suppose $c \in S$ is a corner of $S$. If $P$ is a pattern, its $S$-\emph{contour} with respect to $c$ is
\[ C_{S,c}(P) := \{x \in P \;|\; xc^{-1}S  \not\subset P \}. \]
\end{definition}

Note that the set of $x$ such that $xc^{-1}S \not\subset P$ can be expressed as ``those $x$ such that if $g$ is the translation such that $gc = x$, we have $gS \not\subset P$'', since $gc = x \iff g = xc^{-1}$.

Note also that if $c = e_G$ (which we may assume without affecting the solitaire process by translating the shape), then this simplifies to $\{x \in P \;|\; xS \not\subset P\}$.

Examples of contours on the plane are shown in Figure~\ref{fig:border}.

\begin{figure}[ht]
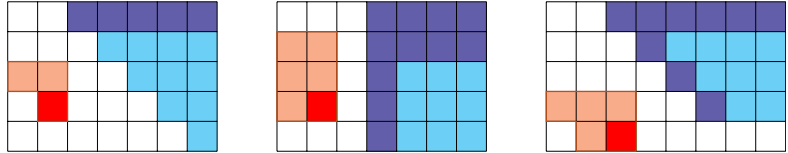

    \centering
    \includestandalone[height = 2cm]{Figures/Convex/border}
    \caption{The dark blue points are the contour of the light blue polygon for shape $S$ in orange with respect to the red corner.}
    \label{fig:border}
\end{figure}

The following result is essentially from \cite{Sa22}. 

\begin{proposition}
\label{prop:AnyTwo}
    Any two contours of a given pattern with the same shape have the same cardinality.
\end{proposition}

\begin{proof}
By \cite[Lemma 28]{Sa22}, if we pick any TEP subshift with the given shape $S\Subset G$, then all contours of a domain $P \subset G$ can be filled arbitrarily and their contents uniquely determine a valid $P$-pattern. (In the terminology of Section~\ref{sec:Permutive}, the contours of $P \Subset G$ are independent and $P$-spanning sets.)

The number of patterns with shape $C_{S,c}(P)$ is then precisely the logarithm (with base the size of the alphabet) of the number of legal patterns on the domain $P$. In particular, it does not depend on the choice of $c$.
\qed\end{proof}

The above proof is rather indirect, as it involves the choice of an arbitrary TEP subshift. We show that under some conditions, one can apply the solitaire process to move between two contours.

\begin{definition}
 Suppose $c, c' \in S$ are two distinct corners of $S$. We say $c$ and $c'$ are \emph{sweep swappable} if there exists a bi-invariant total order on $G$ such that $c = \min S$ and $c' = \max S$ with respect to this order. 
\end{definition}

\begin{lemma}
\label{lem:swap}
Let $c, c' \in S$ be sweep swappable. Then for any $P \Subset G$, the contours $C_{S, c}$ and $C_{S, c'}$ are in the same solitaire orbit.
\end{lemma}

\begin{proof}
Let $<$ denote a total bi-invariant order on $G$ such that $c = \min S$ and $c' = \max S$. By translating $S$ on the left (which preserves solitaire moves and contours) we may assume $c' = e_G$. Let $C, C' \Subset P$ be the contours corresponding to $c, c'$ respectively. Order $S \setminus C$ in increasing order as $g_1, \ldots, g_n$. Now define $C_0 = C$ and inductively $C_i = (C_{i-1} \cup \{g_i\}) \setminus \{g_i c\}$. We claim that this is a valid solitaire move, namely the solitaire move at $g_i$ that moves the hole from $g_ic$ to $g_i$.

For this, we need to show that $|g_i S \cap C_{i-1}| = |S| - 1$, and that $g_i \notin C_i$. The latter fact is clear since $g_i$ is larger than any $g_j$ added previously. For the former, observe first that the filling process following this solitaire process (i.e.\ the result of applying filling moves at each $g_jS$ for $j < i$, which can be proven well-defined by induction) would certainly have added the remaining $|S| - 1$ elements, since they are smaller, thus either on the list of the $g_j$, or in the initial contour $C$.

Thus it suffices to show that the application of a previous solitaire move at $g_j S$ taking us from $C_{j-1} \to C_j$ did not remove any of the elements in $g_iS$. But the removed element when applying a move at $g_j$ is always $g_j c$. Since $c = \min S$, and the order is bi-invariant, $g_j c \leq g_j s \leq g_i s$ for any $s \in S$, with equality only if $j = i$ and $c = s$.

We claim that we now have $C_n \supset C'$, equivalently $P \setminus C_n \subset P \setminus C'$. To see this, suppose that $gc \in P$ but $gc \notin C_n$. From $gc \in P$ we have that $gc \in C_k$ for some $k$. From $gc \notin C_n$ we have that $gc$ is of the form $g_ic$ for some $i$ (since it was removed at some point). The fact it was indeed removed means that $g S \subset P$ so $gc \notin C'$. \qed
\end{proof}

In particular, by the previous proposition all contours are in the same solitaire orbit whenever all the corners are in the same equivalence class under the transitive closure of sweep swappability.

The interested reader can try to order $\Z^2$ so as to swap any two borders in Figure~\ref{fig:border} by a sweep swap. We give a more interesting example of this phenomenon on the free group.

\begin{lemma}
\label{lem:DErp}
    If $S \subset \Z^2$ is a shape whose convex hull does not have two pairs of parallel edges, then all its contours are in the same solitaire orbit. 
\end{lemma}

\begin{proof}
    Consider the following representation of polygon $S$: let $\vec{u_1}, \ldots, \vec{u_k}$ be the vectors normal to the edges of $S$, oriented toward the exterior of $S$. Take a circle $C$ with centre $O$ and represent edge $e_i$ by the intersection between $C$ and the half line with origin $O$ and direction $\vec{u_i}$, and each corner by the arc between the two edges. We claim that corners $a$ and $b$ are sweep swappable if a diameter of $C$ joins there arcs.
    
    \begin{figure}[ht]
        \centering
        \includegraphics[height = 4cm]{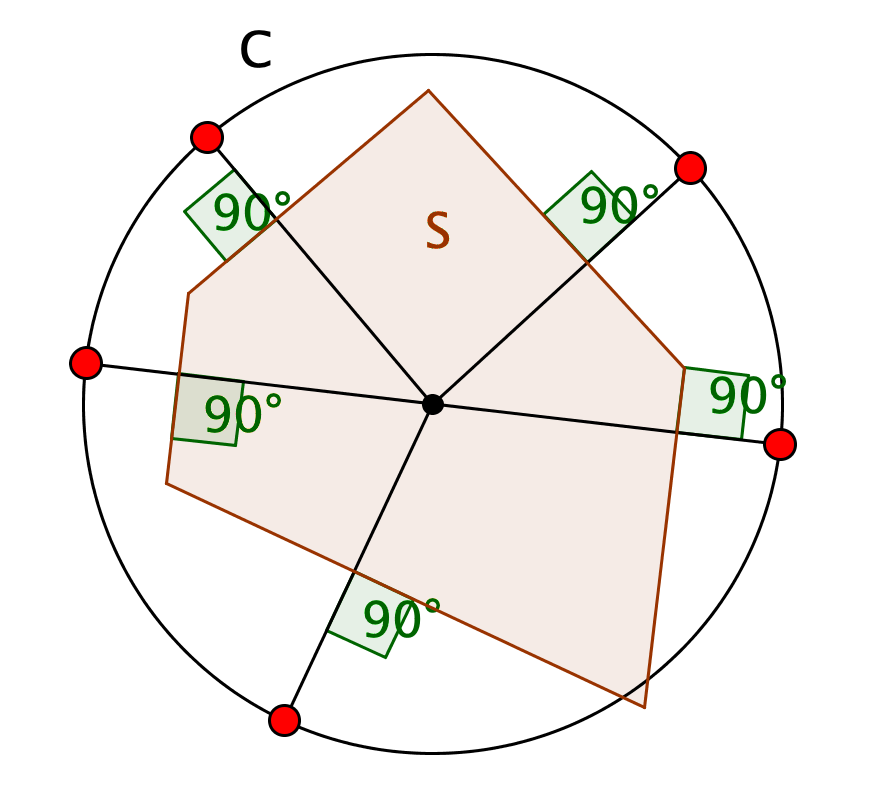}
        \includegraphics[height = 4cm]{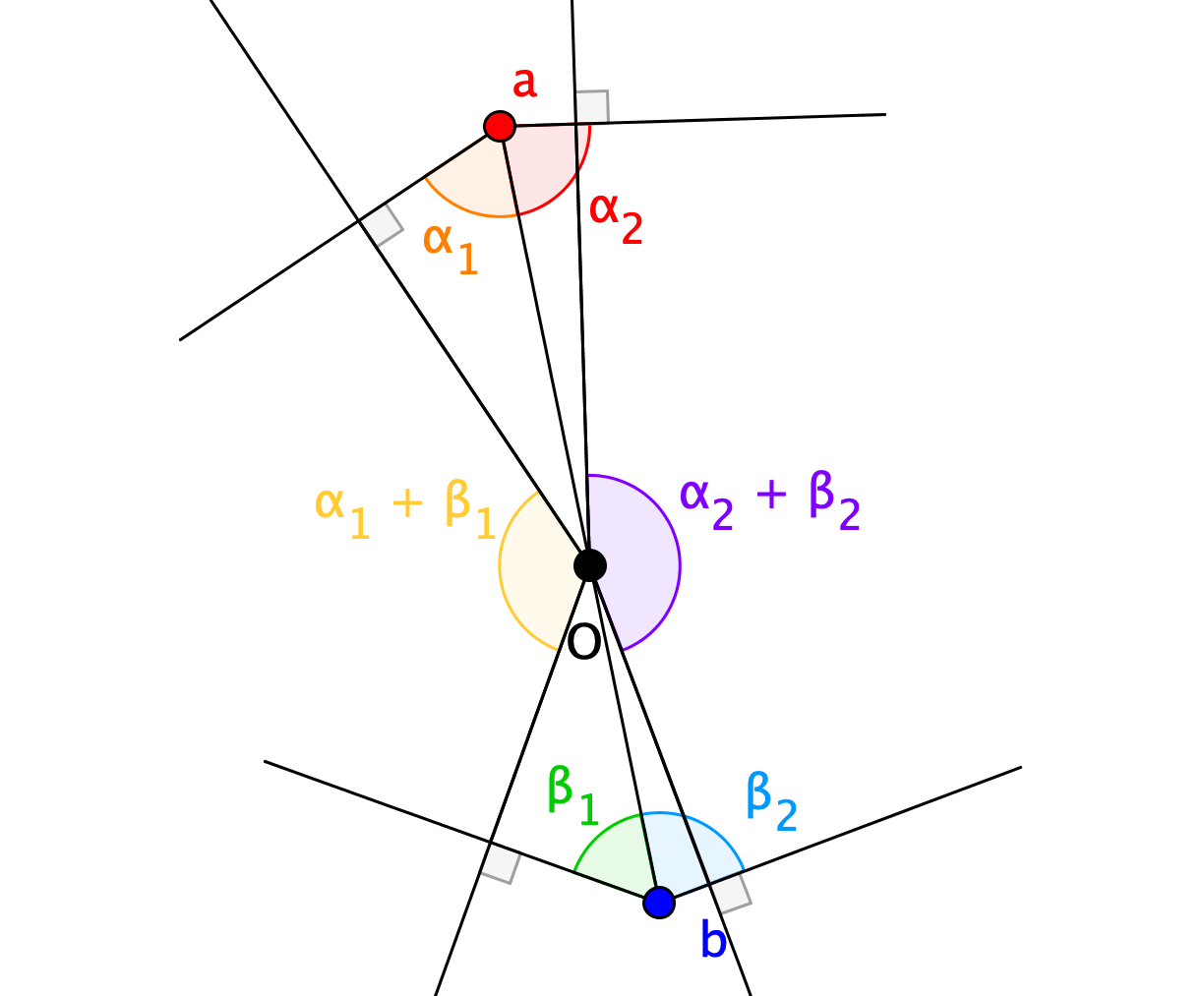}
        \caption{On the left, an illustration on the circle representation of a polygon. On the right, the notations used for the proof.}
        \label{fig:exchange corners}
    \end{figure}

    Note that the open arcs corresponding to corner $a$ consists of the possible directions of unit vectors $\vec v$ which make $a$ a minimal (or maximal) element of the corresponding order Lemma~\ref{lem:Orders}. Thus if a diameter joins two arcs, then the corners are sweep swappable, thus the corresponding contours are in the same solitaire orbit for all patterns.
    
    %
    Now notice that if there are no opposite points (which corresponds to a pair of parallel edges) on the circle, then any border can be crossed by crossing the circle twice. One pair of opposite points is not a problem because we can simply move along the circle in the opposite direction.\qed
\end{proof}

In the case of $\Z^2$, we can in fact describe a process that allows to transform one contour into another it cannot be swapped with.

\begin{theorem}
\label{thm:AllInSameOrbit}
    In dimension 2, all the contours of a pattern are in the same orbit.
\end{theorem}

\begin{proof}
If $S$ has two parallel edges $e$ and $e'$, let $c$ and $c'$ be the extremities of respectively $e$ and $e'$ on the same side (such that the line from $c$ to $c'$ is not a diameter). Then the contours $C$ and $C'$ with respect to $c$ and $c'$ can be exchanged as follows. In $C$, as the other end of $e'$, $b'$, is diametrically opposed to $c$, an $S$-solitaire step can be done there. Follow edge $e'$ from $b'$ to $c'$ and at each step move the point in corner $b'$ to $c$. Then follow $e'$ from $c'$ to $b'$ and move the point in corner $c'$ to $b'$. Repeat at each line until you get $C$.

If looked at in the circle representation, this process allows to exchange corners whose section can be joined by a line parallel to a diameter, thus allow to cross it. Combined with Lemma~\ref{lem:DErp}, it follows that all contours are in the same orbit. \qed
\end{proof}

The process described in the previous proof is the one that we used to move a line along another with the square shape. This is why studying this other example brings dynamics the triangle didn't reveal and in a sense those two examples capture all the processes to convert a contour into another.

\begin{remark}
Aside from the fact that small shapes present nice filling decompositions, another reason why things are easier with small shape is that their solitaire moves can be understand more intuitively. For instance the triangle solitaire can ``rotate a point around another'' while the square solitaire ``slides points along lines''. As the size of the shape grows, it becomes harder both to understand what the movements are and how to get to a configuration where a particular move can be done. 
\end{remark}






\section{Quantitative and algorithmic aspects of convex solitaires}
\label{sec:Algo}

\subsection{Complexity of the identification of the orbit of a pattern}

The characterisation of the orbits through filling and excess provides a polynomial time algorithm to identify to which orbit a given pattern belongs for the triangle and square solitaires. 

The algorithm is the following for the triangle shape:
\begin{alg}[Identify orbit]
\label{algo:ident orbit}
Data: pattern $P$. Result: the canonical representative of the orbit of $P$.
\begin{enumerate}
    \item Fill the pattern.
    \item Divide the filling into triangles $\vec{v_1} + T_{k_1}, \ldots, \vec{v_r} + T_{k_r}$.
    \item Count the excess in each triangle, the canonical representative of the orbit of the pattern is $\bigcup_{i = 1}^r \vec{v_i} + P_{k_i, e(P \cap (\vec{v_i}+T_{k_i}))}$. 
\end{enumerate}
\end{alg}

It can easily be adapted to the square shape, and more generally to shapes on $\Z^2$ that have a nice solitaire theory. This includes all convex $3$ points triangles and $4$ points convex parallelograms by twisting the triangle and square cases. Experiments suggest that some $5$ points shapes such as the trapezoid $\{(0,0), (0,1), (0,2), (1,0), (1,1)\}$ are also nice.


This algorithm has a total time complexity of $O(n^2)$ where $n = |P|$: The first two steps are linear in the number of points in $\varphi(P)$ and $|\varphi(P)| \leqslant \frac{n(n+1)}2$ for the triangle shape so they run in time $O(n^2)$. Step 3 is then linear so the total time complexity of the algorithm is $O(n^2)$.

For shapes other than the triangle, it still holds for $S$-polygons that $|P| \leqslant|C_S(P)|^2$ where $C_S(P)$ is a contour of $\varphi_S(P)$ with respect to $S$ so the bound remains the same. 

\subsection{Number of steps needed to put a pattern in normal form}

One can use the argument in Lemma~\ref{lem:ExtendSuperline} and Lemma~\ref{lem:merging} to successively merge superlines following the process of merging triangles in the filling process. 

A single step of adding a new element into an existing superline takes at most $O(n^2)$ steps: the process of turning one edge into another using the solitaire process clearly takes $O(n^2)$ steps, and adding an element once we have the correct edge takes $O(n)$ steps. Thus, after $O(n^3)$ steps we have turned any pattern of cardinality $n$ into a union of superlines.

It suffices to show that a superline can be turned into the canonical representative $P_{n,k}$ in $O((n+k)^3)$ steps.

For this, the argument in Lemma~\ref{lem:Ordering} can be unraveled to give a practical procedure for fetching the excess one element at a time, at all times keeping the excess lined up, i.e.\ so that at all times the pattern is of the form $P_{n,k} \sqcup R$.

One can use the procedure in Figure~\ref{fig:push excess} to move all excess to the left on the bottom line, and the one in Figure~\ref{fig:push excess right} to move excess from the bottom line to the line above, if it is not yet full. Figure~\ref{fig:fetch more excess} illustrates how one can fetch elements if excess is already lined up.
    
    \begin{figure}[ht]
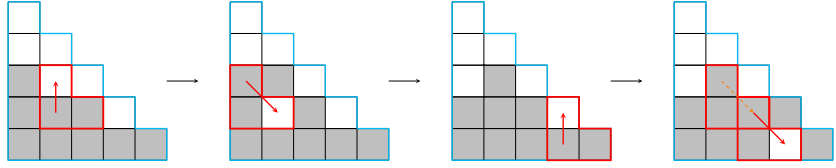

        \centering
        \includestandalone[width=0.9\textwidth]{Figures/Triangle/fig_push_excess_twolines}
        \caption{Pushing excess to the right.} 
        \label{fig:push excess right}
    \end{figure}
    \spacc{}

    \begin{figure}[ht]
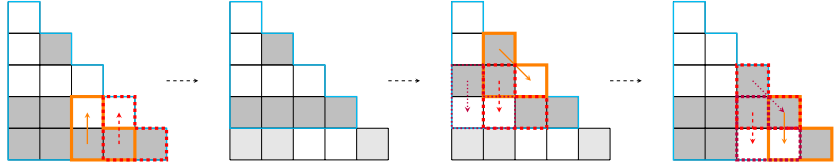

        \centering
        \includestandalone[width=0.9\textwidth]{Figures/Triangle/fig_fetch_excess2}
        \caption{Fetching an excess point with some excess already lined up.}
        \label{fig:fetch more excess}
    \end{figure}


\begin{alg}[Transform a pattern to its canonical form]
Data: pattern $P$. Result: a sequence of solitaire steps that turns $P$ into the canonical representative of its orbit.
\label{alg:path}
\hfill
\begin{enumerate}
    \item Merge the different components and form superlines using the process described in Lemma~\ref{lem:merging} 
    \item Fetch the excess with the process described in Lemma~\ref{lem:Ordering}. 
\end{enumerate}
\end{alg}

\begin{figure}[ht]
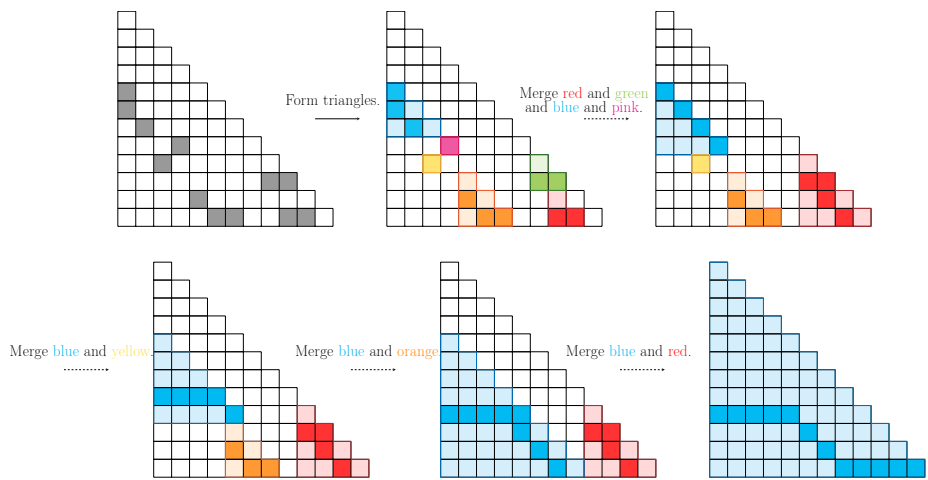

    \centering
    \includestandalone[width =\textwidth]{Figures/Triangle/return}
    \caption{An example of how to get back to the line from a random element of its orbit.}
    \label{fig:return}
\end{figure}

For the square shape, the same algorithm can be applied by replacing superlines by supercrosses and noticing than an excess point can very easily be moved along a line.

Naively implemented, this algorithm takes $O(n^2 (n + k))$, where $k$ denotes the excess: The first step takes $O(n^3)$, since each merging takes $O(n^2)$ time and we merge at most $n$ times. The second step takes $O(n^2 k)$ if we fetch the $k$ many excess points one by one, as each fetch takes $O(n^2)$.

One can in fact achieve $O(n^2 + nk)$. Such a process for the triangle (resp.\ square) is as follows:
\begin{enumerate}
    \item Form a vertical line (resp. a bottom left cross).
    \item Move the line from left to right, and after moving it by one column, use it to move down all the excess point to its left to the bottom of there column. You now have a horizontal line (resp.\ a bottom right cross) with piles of excess on it.
    \item Move the excess points to the right or left to form horizontal lines on top of the bottom one.
\end{enumerate}

Step one requires $O(n^2)$ moves. Step 2 requires again $O(n^2)$ moves for the line movement from left to right and at most $n$ moves to move each excess point down since one point cannot go down by more than $n$ so a total of $O(kn)$ moves. Step 3 takes $O(kn)$ moves for the same reason.

\begin{theorem}
\label{thm:diameter bound}
    For the triangle and the square solitaire, the orbit of a pattern with $n$ elements has diameter at most $O(n^3)$.
\end{theorem}

This is in fact optimal for the line orbit in the triangle solitaire.

\begin{theorem}
\label{thm:diameter}
    The diameter of the orbit of the line of length $n$, seen as a graph, is $\Theta(n^3)$.
\end{theorem}

\begin{proof}
    We are going to build an infinite family of patterns that require $\Omega(n^3)$ steps to get back to the line.
    
    Let $P_0$ be the empty pattern. $P_{n+1}$ is inductively built by extending $P_n$ as described in Figure \ref{fig:far pattern} where the grey triangle is the triangle in which  $P_n$'s orbit is confined.
    \begin{figure}[ht]
        \centering
        \includegraphics[height = 2.5cm]{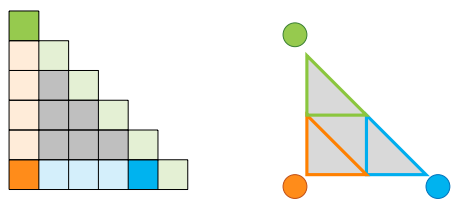}
        \caption{Left: The extension of $P_n$ into $P_{n+1}$. Right: A schematic representation of $P_{n+1}$ used in the proof.}
        \label{fig:far pattern}
    \end{figure}
    \spacc{}
    In pattern $P_{n+1}$, $\Omega(n^2)$ steps are required to fetch the three colored points. 
    
    Indeed, first notice that up to renaming points, the blue point has to move left to fetch the orange one, then the orange one will have to move up to fetch the green one and finally the green one will have to go down to prepare for the next extension. Now let us analyse the movement of the blue point, starting from the first moment it touches another triangle. To move a point left, one need a point in the column at its left. We'll prove that this means that at some point in the process, half of the points of the pattern have to be in the left half of the triangle.
    
    Mark the blue point, and whenever an unmarked point is in the same column as a marked one, mark it. Consider column $i$ from the right, and the first marked point $x$ to reach it. Clearly the point $x$ moves from column $i-1$ to column $i$, and this requires us to have an unmarked point in column $i$ to allow this. This point was unmarked, and is now marked. Therefore, when the blue point has reached column $\frac{3n}2$ (the middle column), there are at least $\frac{3n}2$ marked points, all of which were in the blue triangle in Figure~\ref{fig:far pattern} at some point during the journey of the blue point.
    
    The same reasoning on the orange and green point gives that at some point of the process, between the first moment the blue point moves, until the point where we are ready to move the blue point of the next level, $\frac{3n}2$ points were in the orange triangle and the same amount in the green one. As the pattern $P_n$ only has $3n$ points, at least $\frac{3n}2$ need to be moved from one subtriangle to another, therefore there is a pair of triangles that will share at least $\frac{n}2$ points. Those at least $\frac{n}2$ points will need to be moved by a mean distance of at least $\frac{n}4$, which requires at least $\frac{n^2}8$ steps. Thus pattern $P_n$ requires $\Omega(n^3)$ moves to be transformed into a line.
\qed\end{proof}

\begin{conjecture}
The orbit of a contour of cardinality $n$ of an $S$-polygon has diameter $\Theta(n^3)$
\end{conjecture}


A good candidate for patterns far from a contour are the ones built by adding points around a pattern so that each point extends the filling closure of the pattern by a line, and adding the points on successive sides so that the contour has to be changed between the mergings. Such a pattern for the square is shown in Figure ~\ref{fig:far square}. The natural method to put it in normal form takes $\Theta(n^3)$ moves. We have not proved that this is optimal, but suspect that it is. 
The proof does not seem to directly generalize, however. 

\begin{figure}[ht]
    \centering
    \includegraphics[height = 3cm]{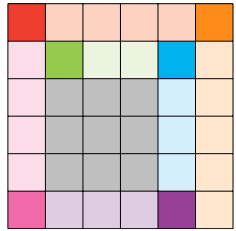}
    \caption{A pattern we conjecture is at $\Theta(n^3)$ steps from the cross. $P_n$ (gray) is extended into $P_{n+1}$ by adding the green, blue, purple and pink points, red and orange are here to show how it keeps growing.}
    \label{fig:far square}
\end{figure}

\subsection{Distance between two contours}

\begin{proposition}
    For any shape $S$, the transformation from one contour to another cannot be done in $o(n^2)$ steps, for some $S$-polygons of cardinality $n$.
\end{proposition}

\begin{proof}
    We define a metric on the set of finite subsets of $\mathbb{Z}^2$ with a given cardinality by $\Delta(A, B) = \min_{f : A \sim B} \sum_{a \in A} d(a, f(a))$ where $d$ is the Euclidean distance on $\mathbb{Z}^2.$ As any move with the solitaire moves exactly one point, by a distance at most the diameter of $S$, the number of moves needed to transform pattern $A$ into pattern $B$ is $\Omega(\Delta(A, B))$.
    
    Let us first consider the triangle. Let $n \in \mathbb{N}$, denote $H_n = [|0; n-1|] \times {0}$ and $V_n = {0} \times [|0, n-1|]$ respectively the horizontal and the vertical line of length $n$ with origin 0. $$\Delta(H_n, V_n) = \min_{f : H_n \sim V_n} \sum_{x \in H_n} d(x, f(x)) = \min_{\sigma \in \mathfrak{S}} \sum_{k = 0}^{n-1} \sqrt{k^2 + \sigma(k)^2} \geqslant \sum_{k = 0}^{n-1} k = \frac{n(n-1)}2$$.

Let $S$ be now a general convex shape, and suppose $S$ has at least 4 sides (the case of 3 edges being similar).
    Notice that because of the triangular inequality, any bijection that reaches $\Delta(A, B)$ must fix every point of $A\cap B$. Consider a family of $S$-polygons $P_n$ for $n \in I \subset \mathbb{N}$ infinite such that the contours of $P_n$ have size $n$ and there is a constant $c$ such that every edge of $P_n$ has length at least $\frac{n}c$ and two opposite edges have distance at least $\frac{n}c$. One can be built by choosing any $S$-polygon, then taking the family of polygons obtained by multiplying the lengths of its edges by an integer.
    
    Let $C$ and $C'$ be two contours of $P_n$. $P_n$ is convex, so the hypothesis on the length and distance of edges implies that any two non consecutive edges of $P$ are at distance at least $\frac{n}{c}$ of each other.
    \begin{itemize}
        \item If $C$ and $C'$ differ by exactly one edge, then there has to be three consecutive edges of $P_n$, $e, f, e'$ such that $e \in C \setminus C'$, $e' \in C' \setminus C$ and $f\notin C \cup C'$. Since $e$ and $e'$ are not consecutive so $\Delta(C, C') \geqslant \Delta(e, e') \geqslant d(e, e')\min(|e|, |e'|) \geqslant \left(\frac{n}{c}\right)^2$. 
        \item Else, $C$ and $C'$ differ by two edges, $e, f \in C$ and $e', f' \in C'$. 
        If they are not consecutive, the inequality is obtained just as above. 
        \item If they are consecutive, for example $f$ follows $f'$, then by the triangular case $\Delta(f, f') \geqslant \alpha\left(\frac{n}{c}\right)^2$ with $\alpha > 0$, by the remark above $\Delta(e, f'), \Delta(e', f) \geqslant \left(\frac{n}{c}\right)^2$ and combining both (depending on whether $e$ and $e'$ are consecutive or not), $\Delta(e, e') \geqslant \alpha'\left(\frac{n}{c}\right)^2$. Thus $\Delta(C, C') \geqslant \Delta(e\cup f, e'\cup f') \geqslant \alpha'' \left(\frac{n}{c}\right)^2$
    \end{itemize}
\qed\end{proof}

This shows in particular that the paths of Proposition~\ref{prop:lines} have the minimal length up to a constant:

\begin{corollary}
    For the triangle solitaire, the distance between two distinct edges of $T_n$ is $\Omega(n^2)$.
\end{corollary}

In particular, while we would like to think of the lines as the ``center'' of the line orbit, they are not unless we consider $O(n^2)$ to be a small distance (which makes some sense because the diameter $\Omega(n^3)$ is much larger). A better candidate for that would be the patterns that are invariant by rotation by 120 degrees. Those patterns are depicted in Figure~\ref{fig:triang center} for $n = 3,4$. We added the lines directed by $-1$ to emphasize the 120 degree symmetry.

\vspace{-0.3cm}
\begin{figure}[ht]
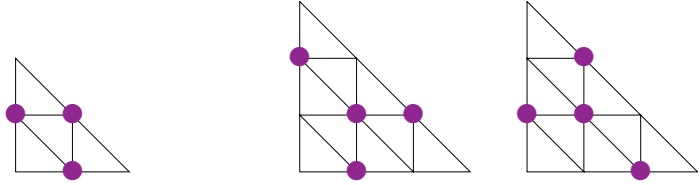

    \centering
    \includestandalone[height=2.5cm]{Figures/Triangle/center}
    \caption{The symmetric patterns of $L_3$ on the left and $L_4$ on the right.}
    \label{fig:triang center}
\end{figure}

\section{A few results on the free group}
\label{sec:Free}

We denote by $F_k$ the free group on $k$ generators and represent it by its Caley graph, i.e. the graph with vertex set $F_k$ and edges between two elements $x$ and $y$ if and only if there is a generator $a$ such that $x = ay$ or $y = ax$. It is the $2k$-regular infinite tree. 

\subsection{Convex shapes on the free group have finite filling}

A shape $S$ on the free group is \emph{connected} if it is a connected component of the Caley graph. 

In what follow, we take $k = 2$ for the sake of simplicity but the results also hold for greater $k$. 
We denote $a$ and $b$ the generators of $F_2$.

\begin{theorem} If $S$ is connected, then $S$ has finite filling if and only if $|S| = 1$ or $S$ is not contained in a right coset of $\langle a \rangle$ or $\langle b \rangle$.
\end{theorem}
\begin{proof} 
    The forward implication is straightforward.

    For the backward implication, consider a finite initial pattern $P$. Take any $b$-edge $e$ in the Cayley graph, such that $P$ has nodes only on one side of it. Let $E \subset F_2$ be the empty side. Suppose we are able to get infinitely many nodes into $E$. To do this you have to first position $S$ so that you have exactly one node in $E$ in $gS$, then you have to position $S$ to get exactly one new node in $E$, and so on. Observe that:
\begin{enumerate}
    \item As long as you have less than $|S|-1$ nodes in $E$, to add a node in $E$ you have to position $S$ so that you use the edge $e$.
    \item All the edges of $S$ that you so use have to be distinct, when you add distinct nodes.
\end{enumerate}
    All in all, $S$ has to have at least $|S|-1$ vertical edges. So it is just a vertical (= $b$-directional) interval.

    Similarly, unless $S$ is a horizontal interval you cannot send information to an empty side of an $a$-edge.

    So if $S$ is neither, then you can at most fill the geodesic closure of $A$, and then at most $|S|-1$ nodes away from its nodes.
\end{proof}

We believe that the free group has the same ``finitely filling shapes'' as $\Z^d$, i.e. those are the non-linear ones, but we do not have a proof yet.

\begin{conjecture}
    Let $S \subset F_2$ be a shape, $S$ has finite filling if and only if it is not contained in a coset of a cyclic subgroup $\langle g \rangle$.
\end{conjecture}

\subsection{Orbits of contours in the free group}

The free group admits many bi-invariant orders, indeed a whole Cantor set of orders \cite{DeNaRi14}. The best-known one is the one dubbed Magnus order in \cite{Sa22}. We omit the definition here,\footnote{In brief, the order comes from the lexicographic order, where the sequences used for the ordering comes from the upper central series, i.e.\ nilpotent quotients of the group.} but the Figure~\ref{fig:free order} gives the order of the ball of radius $2$.
This order can be used to perform a sweep swap between two contours of any shape.

\begin{figure}[ht]
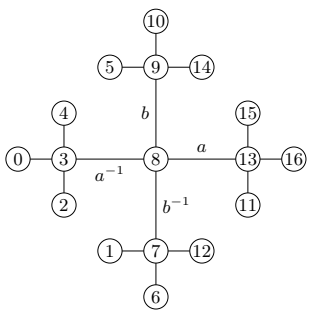

    \centering
    \includestandalone[height=4cm]{Figures/General/freeorder}
    \caption{The Magnus order on the ball of radius $2$ of the free group on two elements, with $0$ as the minimum and $16$ as the maximum}
    \label{fig:free order}
\end{figure}

For example, let $S = \{e, a, b, a^{-1}, b^{-1}\}$. Then the maximal element of $S$ is $a$ and the minimal element is $a^{-1}$. In the sweep swap from the $a$-contour to the $a^{-1}$-contour, we go down this order, and swap the maximal element of $gS$ with its minimal element whenever possible. This is shown in Figure~\ref{fig:free swap}. 

\begin{figure}[ht]
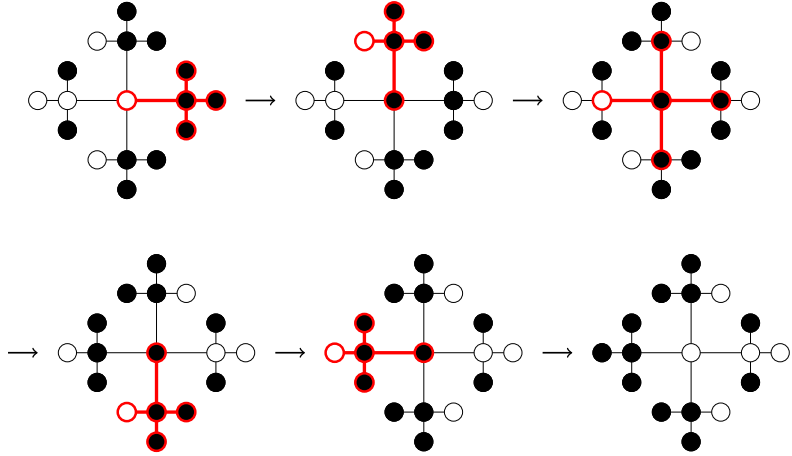

    \centering
    \includestandalone[height=6cm]{Figures/General/freeswap}
    \caption{Sweep swap from the $a$-contour to the $a^{-1}$-contour.}
    \label{fig:free swap}
\end{figure}

In particular, in any TEP subshift with the shape $S$, following the solitaire gives a bijection between the contents of the two contours.

Since the free group has many orderings, the following seems likely.

\begin{conjecture}
    For any shape $S$ on $F_k$ and any $S$-filling closed set $P$, all the $S$-contours of $P$ computed with respect to a corner, are in the same $S$-solitaire orbit. 
\end{conjecture}

Here, by corner we mean again a maximal (or minimal) element with respect to a left-invariant order. By Lemma~\ref{lem:swap}, it would suffice to show that for any two corners  $x, y$ of a set $S$, we can find a sequence of corners $x_1= x, x_2, \ldots, x_k =y$ such that $x_i, x_{i+1}$ are sweep swappable, i.e.\ some left-invariant order has minimal element $x_i$ and maximal element $x_{i+1}$ in $S$.

In \cite{Sa22}, a family of convex sets called \emph{tree convex sets} is defined for subsets of a free group. This gives another natural notion of corner for a finite set. We do not at present know if this coincides with the notion of corner as defined through orders.

\subsection{Triangle solitaire on the free group}

The tree structure of the free group limits the size of the orbit. Intuitively, contrary to the solitaire on $\Z^d$, it is no longer possible to bring points from different area and use them together to make new moves. We illustrate this by studying the orbit of the line for the triangle solitaire on $F_2$.

Let $S = \{e, a, b\}$ with $a$ and $b$ the generators of $F_2$. We call $\cyc{a}$ the horizontal direction and $\cyc{b}$ the vertical one. Let $L_n = \{e, a, a^2, \ldots a^{n-1}\}$ the horizontal line of length $n$. Notice that $\varphi(L_n) = L_n \cup L_{n-1}$, i.e. the filling only adds one line on top on the original one. This is because contrary to $\Z^2$, $a$ and $b$ do not commute so we cannot use $b$ and $ab$ to add a point at $b^2$.

Now consider the line orbit. From $L_n$, solitaire moves can move a point at $a^i$ to either $a^ib$ or $a^{i-1}b$ (except $e$ can only be moved to $b$ and $a^{n-1}$ only to $a^{n-2}b$). If such a move is done, for example $a^i \rightarrow a^ib$, then the moves $a^{i+1} \rightarrow a^ib$ and $a^{i-1} \rightarrow a^ib$ are no longer possible, all the other moves remaining possible, and two new moves become possible: $a^{i-1} \rightarrow a^i$ and $a^{i+1} \rightarrow a^i$, but doing one of those would lead to the same result as replacing the first move by $a^{i\pm1} \rightarrow a^i$. Therefore, an element of the line orbit is obtained by choosing for each point if we move it up, left up or not.

\begin{lemma}
    Encoding $E$ for the point was not moved, $B$ for movement along $b$ and $A$ for movement along $a^{-1}b$, the line orbit under the triangle solitaire on the free group is described by the set of word of length $n$ on $\{E, A, B\}$ with no subword $BA$, starting with $E$ or $B$ and ending with $E$ or $A$. 
    
    This can be rewritten as a regular language equation:
    \[\orb(L_n) = ((\{a\} + \{ab\} + \{a, ab\}\cup a(\{ab\}\cup a)^*\cup a)\cup a)^*(\{a\} + \{a, ab\}\cup a(\{ab\}\cup a)^*\varnothing)\]
    where $(X\cup a)^k Y = X \cup aX \cup a^2X \cup \ldots \cup a^kY$.
\end{lemma}


\begin{remark}
    This is equal to the line orbit under the triangle solitaire on $\Z^2$ when restrained to two horizontal lines. This particular case is regular and was studied in \cite{OEISfillmatrices}. An automaton recognising it is depicted Figure~\ref{fig:auto2}. 
\end{remark}

\begin{figure}[ht]
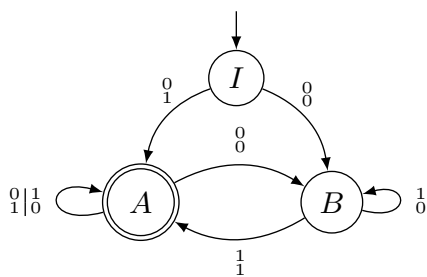

    \centering
    \includestandalone[width=0.5\textwidth]{Figures/Triangle/auto2}
    \caption{An automaton recognising $\orb(L_n)$. The automaton read the pattern from right to left, considering the content of $a^k$ and $a^kb$ as a letter with $1$s standing for points and $0$s for empty cells. }
    \label{fig:auto2}
\end{figure}

\begin{corollary}
    The size of the line orbit on the free group is $$|\orb(L_n)| = \frac{1}{2^n\sqrt{5}}\left( (3+\sqrt{5})^n-(3-\sqrt{5})^n \right).$$
\end{corollary}

Notice that this is very small compared to the line orbit on $\Z^2$.

\section{Connection to TEP and permutive subshifts}
\label{sec:Permutive}

\subsection{TEP and examples}

Recall that for $G$ a group, $C \Subset S \Subset G$ and $A$ a finite alphabet, a $(C,S)$-TEP set of patterns is a set $\mathcal{T} \subset A^S$ such that for each $c\in C$, for all $\pi : S\setminus \{c\} \to A$, there is a unique $p \in T$ such that $p_{|S\setminus\{c\}} = \pi$, and a $(C,S)$-TEP subshift is a set $X$ of configuration $x \in A^G$ such that for all $g \in G$, $x_{|gS} \in \mathcal{T}$ for some $(C,S)$-TEP set of patterns.

TEP stands for ``totally extremally permutive'', a term we explain in the following section. The TEP subshifts (with a specific choice of $C$ for a given shape $S$, discussed below) were introduced in \cite{Sa22} by the first author. Here, we develop the connection to solitaire in slightly more generality. (In \cite{Sa22}, the definition of TEP is in fact generalized in another direction, to ``$k$-TEP subshifts'', but they do not seem to admit a notion of solitaire.)

Next we recall a result from \cite{Sa22} which is one of the main  motivations for studying TEP subshifts.

\begin{definition}
Let $G$ be a group. A \emph{convex geometry} on $G$ is a set $\C$ of finite subsets of $G$ such that:
\begin{enumerate}[label=(\roman*)]
    \item $\C$ contains $\emptyset$.
    \item $\C$ is closed under finite intersections.
    \item $G$ can be written as an increasing union of sets from $\C$.
    \item If $C, D \in \C$ and $C \subsetneq D$, then we can find $g \in D \setminus C$ such that $C \cup \{g\} \in \C$.
\end{enumerate}
This last property is called the \emph{corner addition property}. We say $\C$ is \emph{invariant} if $C \in \C \implies \forall g \in G :gC \in\C$ and \emph{midpointed} if $gh^{-1}, gh \in C \wedge C \in \C \implies g \in C$, where $g,h \in G$. The \emph{corners} of $S \Subset G$ with respect to $\C$ are the elements $s \in S$ such that $S \setminus \{c\} \in \C$.
\end{definition}

On the group $\Z^d$, the main example of a midpointed invariant convex geometry is the family of intersections of real convex sets with $\Z^d$, already discussed in previous sections. In \cite{Sa22}, convex geometries are constructed on many other groups as well, in particular on the free group we have the geometry of tree convex sets (whose definition we omit here).


One of the main result of \cite{Sa22} is the following:

\begin{theorem}
Let $G$ be a group and $\C$ a midpointed invariant convex geometry on $G$. Let $S \Subset G$ and $C$ the corners of $S$ with respect to $\C$. If $X$ is a $(C, S)$-TEP-subshift (for example any $S$-TEP subshift) on the group $G$ defined by a TEP set $T \subset A^S$, then whenever $D \in \C$, a pattern $p \in A^D$ appears in $X$ if and only if $p$ is locally valid, meaning whenever $gS \subset D$ we have $g^{-1}x|_S \in P$.
\end{theorem}


In other words, in TEP subshifts, \emph{locally valid} patterns (meaning patterns whose $S$-shaped subpatterns are in $T$ up to a shift) on convex domains are \emph{globally valid}, meaning extend to a configuration of $X$.

\begin{remark}
This theorem shows that TEP subshifts are a ``syntactically defined'' class of SFTs which have decidable languages (indeed uniformly in the defining family of forbidden patterns $A^S \setminus P$), assuming the convex geometry is nice enough. This is in contrast to general SFTs, which are known to have undecidable languages in general. Given an SFT, even whether it is empty is a well-known classical undecidable problem \cite{Be66}.
\end{remark}

\subsection{Permutivity}

Much of the theory makes sense under the weaker assumption of permutivity, which we define below.

\begin{definition}
Let $G$ be a group, and let $C \Subset S \Subset G$. Let $A$ be a finite alphabet and $X \subset A^G$ be a closed set. We say $X$ is $(C, S)$-permutive if for each $c \in C$ and $g \in G$ there exists a function $f_g : A^{S \setminus \{c\}} \to A$ such that for every $g \in G$, $x \in X$ and $p \in A^S$ we have
\[  g^{-1}x|_S = p \implies f_g(p|_{S \setminus \{c\}}) = p_c. \]
An $(S, S)$-permutive subshift is called simply $S$-permutive.
\end{definition}

In the term TEP (totally extremally permutive), ``permutivity'' refers to permutivity, ``extremal'' to the fact that $C$ is thought of as the set of ``extremal vertices'' or ``corners'' of $S$. ``Total'' refers to the fact that the restriction to $S \setminus \{c\}$ is always $A^{S \setminus \{c\}}$ when $c \in C$, which is not required in the previous definition of permutivity.

A set $X$ that is permutive in the previous sense may not be a subshift. When all the functions are the same, i.e.\ $\forall g \in G: f_g = f$, the condition simplifies to
\[ \forall p \in X|_S: f(p|_{S \setminus \{c\}}) = p_c, \]
which is a property of TEP subshifts (but the assumption about surjectivity of projection to $S \setminus \{c\}$ is missing).

A special case of permutivity was studied by \cite{FrKr20} under the name polygonal subshifts. This special case is $S \Subset \Z^d$ with $C$ the set of corners of $S$ in the geometric sense, and $X$ is a subshift. (Note that for the solitaire theory to cover this family, it is again important that we study general $(C, S)$-solitaire instead of just the case $C = S$.)

The term \emph{permutive} comes from the following lemma (whose proof is straightforward and omitted).

\begin{lemma}
Let $G$ be a group, and let $C \Subset S \Subset G$. Let $A$ be a finite alphabet and $X \subset A^G$ be a closed set. If $X$ is $(C, S)$-permutive, then for any $g \in G$ and distinct $c, c' \in C$, writing $T = G \setminus \{gc, gc'\}$, we have for all $x \in A^T$ a bijection $\pi : A \to A$ such that
\[ x \sqcup a^{gc} \in X|_{T \cup \{gc\} }\iff x \sqcup \pi(a)^{gc'} \in X|_{T \cup \{gc'\}}. \]
Furthermore, the permutation $\pi$ only depends on $x|_{gS \setminus \{gc, gc'\}}$.
\end{lemma}

\subsection{Independent sets}
\label{sec:indep}

\begin{definition}
Let $G$ be a group, and let $X \subset A^G$ be a closed set. The \emph{topological independence sets} $\TI(X)$ are the sets $T \in \pow(G)$ satisfying $X|_T = A^T$. 

If $\mu$ is a Borel probability measure on $X$, then a finite set $T \Subset G$ is \emph{equidistributed} if the $\mu([p]) = \mu([q])$ for any $p, q \in A^T$. The \emph{$\mu$-independence sets} $\I_\mu(X)$ are the closure of the equidistributed finite sets $T \Subset G$ in the product topology of $\pow(G)$.
\end{definition}

Note that as stated, both types of independence sets depend on the alphabet $A$. One may always use the effective alphabet (symbols that actually appear in configurations) unless stated otherwise. Note that the assumption on $\mu$ in particular implies that if $T$ is equidistributed then all patterns $A^T$ are in the support of the marginal distribution of $\mu$ on $T$, in other words, $\TI_\mu(X) \subset \TI(X)$ for any Borel probability measure $\mu$ on $X$.

For simplicity, we concentrate mostly on sets of topological independence, but shift-invariant measures are of great interest in symbolic dynamics, so we include them in Theorem~\ref{thm:TISolitaire} below.

\begin{example}
    Consider the triangle shape $\{(0,0), (0,1), (1,0)\}$, an abelian group alphabet, and the TEP rule where the value at $(0,1)$ must be the sum of the values in the other positions, on $\Z^2$. Consider the set $T = \{(0,0), (2, 0), (0,2)\}$. It is not independent for the group $A' = \Z/2\Z$ with addition modulo $2$; one can check by a calculation that the values must sum to $0$.
    
    However, it is independent for $A = \Z/3\Z$ (still with addition). Using the notation depicted in Figure~\ref{fig:indep}, let $a,b,c$ be arbitrary. The rule gives $a = x+y$, $x = b+z$ and $y = c+z$ so $2z = a-b-c$. Now remember that in $\Z/3\Z$, we have $2=-1$. This gives $z = c+b-a$, then $x = c-b-a$ and $y=b-c-a$ which does not contradict $a=x+y$.  \qee
\end{example}
\vspace*{-0.5cm}
\begin{figure}[ht]
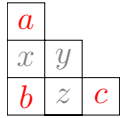

    \centering
    \includestandalone[height = 1.5cm]{Figures/TEP/indep}
    \caption{The red letters correspond to elements of $T$, which is an independent set for $A = \Z/3\Z$ but not $A' = \Z/2\Z$.}
\label{fig:indep}
\end{figure}

\begin{theorem}
\label{thm:TISolitaire}
Let $G$ be a group, let $C \Subset S \Subset G$ be as above, and let $X \subset A^G$ be a $(C, S)$-permutive closed set. Then the sets of topological independence $\TI(X)$ are closed under the $(C, S)$-solitaire. If $\mu$ is a Borel probability measure on $X$, then the $\mu$-independence sets are closed under $(C, S)$-solitaire. Both sets are down sets under inclusion, and are topologically closed subsets of $\pow(G)$. If $X$ is shift-invariant, then $\TI(X), I_\mu(X)$ are shift-invariant sets.
\end{theorem}

\begin{proof}
We start by proving the last three claims. First, the set $\TI(X)$ is obviously down, and $\I_\mu(X)$ is down because it is the closure of finite $\mu$-independent sets, which in turn are down because for $T' \subset T \Subset G$, the measure on $A^{T'}$ is determined by the one on $A^T$ by the natural projection from $A^T$ to $A^{T'}$, where all fibers have equal cardinality.

Topological closedness of $\I_\mu(X)$ is true by definition. For $\TI(X)$ it suffices to show that $X|_P = A^P$ if and only if $X|_Q = A^Q$ for any $Q \Subset P$. But this is immediate from the fact that $\pi : X \to X|P$ is a continuous map between two closed subsets of Cantor space (which are compact Hausdorff).

If $X$ is shift-invariant, then shift-invariance of $\TI(X)$ and $\I_\mu(X)$ are trivial.

Now we show closure under solitaire. Let $T \subset G$ be a set of topological independence, meaning $X|_T = A^T$. Suppose $T \rightarrow_{C, S} T'$ meaning
\[ |T \cap gS| = |T' \cap gS| = |S| - 1 \wedge T \triangle T' \in \pow_2(gC). \]
Let $T \setminus T' = \{gc\}$, $T' \setminus T = \{gc'\}$
By the definition of $(C, S)$-permutivity of $X$, there are functions $f = f_{g,c} : A^{S \setminus \{c\}} \to A$ and $f' = f_{g,c'} : A^{S \setminus \{c'\}} \to A$, such that for every $g \in G$ we have
\[ g^{-1}x|_S = p \implies f(p|_{S \setminus \{c\}}) = p_c \]
and
\[ g^{-1}x|_S = p \implies f'(p|_{S \setminus \{c'\}}) = p_{c'}. \]

Let $U = T \cap T'$ and $V = T \cup T'$. Note that $T \setminus U, T' \setminus U$ are singleton sets contained in $gC$, and similarly for $V \setminus T, V \setminus T'$. Now, fix $x|_U \in A^U$, let $E \subset A$ be the set of symbols $a \in A$ such that $x|_U \sqcup a^{gc} \in X|T$, and $E' \subset A$ the set of symbols $a \in A$ such that $x|_U \sqcup a^{gc'} \in X|T'$. Clearly $E = A$ by the assumption that $T \in \TI(X)$. Our task is precisely to show that $E' = A$.

We give a bijection between these sets. For this, observe that for $x \in X$, the restriction $x|T$ uniquely determines $x|V$ and thus $x|T'$ (using the property of $f'$), and the restriction $x|T'$ uniquely determines $x|T$. In particular, there is a bijection between $X|T$ and $X|T'$, which can only change symbols in the symmetric difference of $T, T'$. This is only possible if indeed $E' = A$ (since $E = A$), concluding the proof for the sets of topological independence.

If $\mu$ is a probability measure, the proof is similar. Namely, suppose $T$ is a limit of finite patterns $Q$ such that $\mu$ is uniform on the finite patterns $X|_Q$. Then fix $Q \supset gS$ and let $\mathcal{P} = X|_Q$. We have that the distribution on the legal patterns $\mathcal{P}' = X|_Q \setminus \{gc\} \cup \{gc'\}$ is also uniform, because as above we have a bijection between $\mathcal{P}$ and $\mathcal{P}'$. Now $T'$ is a limit of such patterns $Q \setminus \{gc\} \cup \{gc'\} = T'$. \qed
\end{proof}


The solitaire also has applications also to sets that are not independence sets for a TEP subshift:

\begin{remark}
If $\mu$ is a Borel probability measure on a $(C, S)$-permutive $X$, then the set of $P \Subset G$ such that $\mu$ is equidistributed on the patterns $X|_P$ (without assuming $X|_P = A^P$!) is also closed under the solitaire process, by a similar proof as above.
\end{remark}




The fact the topological independence subshift is closed under solitaire is the original motivation for the study of the solitaire process. The $S$-TEP subshifts are an interesting class of subshifts (generalizing the Ledrappier subshift and spacetime subshifts of bipermutive cellular automata), and the sets of topological independence are interesting class of subshifts of the group (which in the Ledrappier case forms a matroid \cite{Sa22}).


\subsection{Spanning sets}
\label{sec:spanning}

We now define a kind of dual notion to topological independent sets.

\begin{definition}
Let $G$ be a group, $D \subset G$, and $X \subset A^G$ closed. Then $P \subset G$ is a \emph{$D$-spanning} if there is a function $f : X|_P \to X|_D$ such that for all $x \in X$, $x|_D = f(x|_P)$.
\end{definition}

We use the term ``spanning'' as an analog of ``independence'', roughly as the terms are used in linear algebra and matroid theory. However, the $G$-spanning sets are often called \emph{expansive} or \emph{coding} sets. In the case $G = \Z^d$ an \emph{expansive subspace} $H \leqslant \Z^d$ is precisely one such that $H + B_r \cap \Z^d$ is $G$-spanning \cite{BoLi97}, where $B_r$ is the Euclidean ball of radius $r$. For the cellular automata minded it might also be natural to call this the \emph{determined set}, as we have deterministic (local) rules for it in the following sense:

\begin{lemma}
Let $G$ be a group, and let $X \subset A^G$ be a closed set. Let $P$ be a $D$-spanning set. Then for all $g \in D$ there exists a finite set $P_g \Subset P$ and a function $f_g : X|_{P_g} \to A$ such that for all $x \in X$, we have $x_g = f_g(x|_{P_g})$.
\end{lemma}

\begin{proof}
Suppose there is a function $f : X|_P \to X|_D$ such that for all $x \in X$, $x|_D = f(x|_P)$. Note that the graph of this function is just $Y = \{(x|_P, x|_D) \;|\; x \in X\}$, which is the continuous image of the compact space $X$, thus $Y$ is compact. A function with compact Hausdorff codomain and closed graph is continuous (this is one of the closed graph theorems from general topology \cite{Ke17}), thus we conclude that $f$ is continuous.

Now the claim is proved as the Curtis-Hedlund-Lyndon theorem: the clopen partition $\{[a]_g \;|\, a \in A\}$ has clopen preimage in $f$, and inclusion in a clopen set depends on only finitely many coordinates.
\qed\end{proof}

It is obvious that if $P$ is both $Q$-spanning and $R$-spanning, then it is $(Q \cup R)$-spanning. It is also clear that the set of sets $Q$ for which $P$ is $Q$-spanning is closed in the product topology of $\pow(G)$. This gives rise to the following definition.

\begin{definition}
Let $G$ be a group, $P \subset G$, and $X \subset A^G$ closed. The \emph{spanned set} $\psi_X(P)$ of $P$ is the maximal subset $D$ of $G$ such that $P$ is $D$-spanning.
\end{definition}

\begin{lemma}
Spanned sets form a closure system.
\end{lemma}

\begin{proof}
By definition, we need to show $P \subset \psi_X(P)$ (extensivity), $P \subset P' \implies \psi_X(P) \subset \psi_X(P')$ (monotonicity) and $\psi_X(\psi_X(P)) = \psi_X(P)$ (idempotency). All of these properties are obvious from the semantics of the operator (and also easy to verify from the formal definition).
\qed\end{proof}

Sets $P$ such that $\psi_X(P) = P$ are called \emph{spanning-closed}. Our particular interest is in spanning sets in TEP subshifts. The following is immediate from the definitions, but we give a detailed proof.

\begin{lemma}
\label{lem:SpannedContainsFilled}
Let $G$ be a group, let $C \Subset S \Subset G$ and let $X \subset A^G$ closed and $(C, S)$-permutive. Let $P \Subset G$. Then $P$ is a $\varphi(P)$-spanning set. In other words, the spanned set always contains the filling closure, or in a formula $\psi_X(P) \supset \varphi(P)$ for all $P \subset G$.
\end{lemma}

\begin{proof}
By definition of the filling process, we produce $\varphi(P)$ from $P$ by constructing a sequence of sets $P_0 = P, P_1, \ldots$ such that $P_{i+1}$ is obtained from $P_i$ by taking $g_i$ such that $|g_iS \cap P_i| = |S|-1$ and $|g_iC \cap P_i| = |C|-1$, and letting $P_{i+1} = P_i \cup \{g_i c_i\}$ where $g_i c_i \notin P_i$.

Now let $x \in X$ be arbitrary. We prove that $x|_{\varphi(P)}$ is uniquely determined by $x|_P$ by following the chain $P_i$. First, $x|_{P_0} = x|_P$ is trivially determined by $x|_P$. Now suppose $x|_{P_i}$ is determined already.

By the definition of $(C, S)$-permutivity, there is a function $f_{g_i} : A^{S \setminus \{c_i\}} \to A$ such that
\[ g_i^{-1} x|_S = p \implies f_g(p|_S \setminus \{c_i\}) = p_{c_i}. \]
Note that $g_i^{-1} x|_S$ contains the same data as $x|_{g_iS}$, and $p_{c_i} = x|_{g_ic_i}$. so we are precisely saying that the pattern $x|_{P_{i+1}}$ is uniquely determined by $x|_{g_iS}$, in particular by $x|_{g_iS \setminus \{g_ic_i\}}$. The set $g_iS \setminus \{g_ic_i\}$ is contained in $P_i$ since because $|g_iS \cap P_i| = |S|-1$ and  $g_i c_i \notin P_i$. Therefore, $x|_{P_{i+1}}$ is determined by $x|_{P_i}$. 
\qed\end{proof}

\begin{lemma}
Let $G$ be a group, let $C \Subset S \Subset G$ and let $X \subset A^G$ closed and $(C, S)$-permutive. Then $(C, S)$-solitaire preserves the spanned set.
\end{lemma}

\begin{proof}
Suppose $P \rightarrow Q$. We know that solitaire preserves the filling closure, $\varphi(P) = \varphi(Q)$. Then we calculate
\[ \psi_X(P) = \psi_X(\psi_X(P)) \supset \psi_X(\varphi(P)) \supset \psi_X(Q). \]
The other inclusion is symmetric. 
\qed\end{proof}

\begin{lemma}
Let $G$ be a group, let $C \Subset S \Subset G$ and let $X \subset A^G$ closed and $(C, S)$-permutive. Then every spanning-closed set is filling-closed.
\end{lemma}

\begin{proof}
Suppose $\psi_X(P) = P$. Then Lemma~\ref{lem:SpannedContainsFilled} we have $P \subset \varphi(P) \subset \psi_X(P) = P$ so $\varphi(P) = P$. 
\qed\end{proof}

It is tempting to conjecture that in natural situations, the spanned sets of a TEP subshift would in fact correspond to filling closures. We will see in Section~\ref{sec:Ledrappier} that this fails in the very classical Ledrappier subshift (which is the spacetime subshift of the two-neighbor XOR CA).

There is an obvious connection between the cardinalities of independent and spanning sets in a general subshift:

\begin{definition}
Let $G$ be a group, and let $X \subset A^G$ be a closed set. For $P \Subset G$, write $\rank_{\mathrm{span}}(P)$ for the minimal cardinality of a set $R \subset P$ such that $\psi_X(R) \supset P$. Write $\rank_{\mathrm{indep}}(P)$ for the maximal cardinality of a set $Q \subset P$ such that $X|_Q = A^Q$.
\end{definition}

\begin{proposition}
Let $G$ be a group, and let $X \subset A^G$ closed set. Then
\[ \rank_{\mathrm{indep}}(P) \leqslant \rank_{\mathrm{span}}(P) \]
\end{proposition}

\begin{proof}
The inequality $\rank_{\mathrm{indep}}(P) \leqslant \rank_{\mathrm{span}}(P)$ holds in general (Without assuming permutivity). Namely, suppose $R$ is such that $\psi_X(R) \supset P$. This means that $x_{|R}$ determines $x_{|P}$ for all $x \in X$, so in particular $|X_{|R}| \geqslant |X_{|P}|$. If $X|_Q = A^Q$ and $Q \subset P$, then certainly $|X_{|P}| \geqslant |A^Q|$, so we must have $|X_{|R}| \geqslant |A^Q|$ and taking base-$|A|$ logarithms we get $|R| \geqslant |Q|$.

This means every spanning set is larger than any independent set, in particular the minimal cardinality of a spanning set is larger than the maximal cardinality of and independent set, proving the inequality. 
\qed\end{proof}

Similarly, there is a connection between the filling rank and the spanning rank.

\begin{proposition}
Let $G$ be a group, $C \Subset S \Subset G$, and let $X \subset A^G$ be a $(C, S)$-permutive closed set. If $P$ is closed under filling, then $\rank_{\mathrm{indep}}(P) \leqslant \rank(P)$, where $\rank(P)$ is computed with respect to $(C, S)$-solitaire.
\end{proposition}

\begin{proof}
Let $R \subset P$ satisfy $\varphi(R) = P$. Then in particular $R$ is a $P$-spanning set contained in $P$. Thus its cardinality is at least as large as that of a minimal $P$-spanning set.
\qed\end{proof}

Alternatively, we can say that the filling process is only one way to find spanning sets, and thus we expect that minimal spanning sets of a pattern can be smaller than the minimal sets that fill it.

\subsection{Bases and filling bases}
\label{sec:Bases}

\begin{definition}
Let $P \Subset G$ and let $X \subset A^G$ be closed. We say $T \Subset P$ is a \emph{topological (resp.\ $\mu$-) basis} of $P$ (with respect to $X$) if it is topologically (resp.\ $\mu$-)independent and $P$-spanning. 
\end{definition}

We concentrate on topological bases, but by using the TEP measure from \cite{Sa22} one has analogous results for the $\mu$-version.


While the bases of a TEP subshift are of great interest, there is not much we can say about them in general. Indeed, even in TEP subshifts with the triangle shape, the family of bases depends on the subshift. The following related definition is more amenable to analysis through the solitaire, and we will show that for the triangle shape on the plane, it does not depend on the subshift.

\begin{definition}
Let $P \Subset G$ be a filling closed set, and let $X \subset A^G$ be closed. We say $T \Subset P$ is a \emph{filling basis} of $P$ if it is independent and its filling closure equals $P$.
\end{definition}

In the conference version \cite{SaSc23}, we called ``filling bases'' as ``bases'', but we feel now that it is better to use the generic term for a concept that is not TEP or solitaire specific.

Now, the following is obvious:

\begin{lemma}
Any filling basis for a filling closed set is a basis for it.
\end{lemma}

\begin{lemma}
\label{lem:BasesPreserved}
Let $P$ be a (filling) basis for $F = \varphi(P)$. Then every pattern in its solitaire orbit is also a (filling) basis for $F$.
\end{lemma}

\begin{proof}
Independent sets are closed under solitaire, filling closure is preserved under solitaire, and $F$-spanning is preserved under solitaire.  \qed
\end{proof}

We prove that in TEP subshifts, every convex set has a filling basis (in particular, a basis).

\begin{theorem}
Every convex set $P$ in a TEP subshift contains a filling basis.
\end{theorem}

\begin{proof}
Let $G$ be the group, $\C$ the convex geometry, $P \in A^S$ be the $(C, S)$-TEP set (so $C$ are the corners of $S$ with respect to $\C$).

Let $F$ be finite convex set. By iterating the corner addition property from $\emptyset$, adding elements from $F$, we obtain an increasing sequence of convex sets $F_0 = \emptyset, F_1, \ldots, F_k = F$ where $k = |F|$ and $F_{i+1} \setminus F_i$ is a singleton. It is shown in \cite{Sa22} that for any pattern $D \in A^{F_i}$ the set of locally valid patterns on $A^{F_{i+1}}$ which extend $D$ are either completely free (any symbol can be chosen at $g \in F_{i+1} \setminus F_i$), or are uniquely determined (and there is a possible choice).

Let $g_1, \ldots, g_\ell$, $\ell \leq k$, be the elements where we have a free choice. Then no matter how we make the free choices, we end up with a locally valid pattern on $F$. By the previous theorem, these patterns are globally valid, so the set $\{g_1, \ldots, g_\ell\}$ is independent.

When we add a corner whose contents cannot be freely chosen, the reason for this is precisely that a filling move can be applied. Thus, the $g_i$ form a filling basis.
\end{proof}

Recall that if $S \Subset G$ and and $c \in S$ is a corner of $S$. The $S$-contour with respect to $c$ of a pattern $P$  is $C_{S,c}(P) := \{x \in P \;|\; xc^{-1}S  \not\subset P \}$. The following theorem is immediate from the proof of Lemma 5.21 in \cite{Sa22} and is another source of bases in the orderable case.

\begin{theorem}
Let $G$ be any left-orderable group, let $X$ be an $(C, S)$-TEP subshift with $C$ the corners of $S$. Then the $S$-contour of every convex set is a filling basis for it.
\end{theorem}

Since solitaire preserves bases, one can now compute many bases by simply applying solitaire moves to a basis, for instance the contour of a convex set, which is easy to compute. 

\begin{example}
We continue the lamplighter example with the same shape $S = \{e, a, ab^{-1}\}$. Note that this group is not torsion-free ($ab^{-1}$ is of order $2$), so it is not left-orderable.

It is easy to show that in any TEP subshift, the bottom elements (those with the lamplighter at position $0$, labelled $1$ to $8$ in Figure~\ref{fig:Lamplighter2}) actually form a filling basis: we showed that the filling closure contains the tetrahedron (and in fact it contains nothing more). 

A simple ad hoc argument shows that indeed any contents is possible in this set (in that it extends to a global configuration), i.e.\ that the set is independent. First, to fill the tetrahedron from any given values of the bottom nodes, we simply note that the shape deterministically gives unique values for each node (in that it applies exactly once with the top node $a$ translated to any particular non-bottom node).

To extend to a global configuration, we note that geometrically the group's full Cayley graph is obtained by repeatedly extending the tetrahedron upward (by adding a copy of the existing tetrahedron ``behind'' the present one, and joining them by new top nodes); and extending it downward by adding a copy on the right, and joining them by new bottom vertices. For the copies, we can always simply copy the existing values. When extending upward, one can again use determinism, and to extend downward one can pick every second value arbitrarily, and use the TEP rule to fill the rest.

TEP subshifts with this shape all allow the same solitaire to be played, but they can have somewhat different properties dynamically. If we use an abelian alphabet $A$ and require that $gS$-contents sum to $0$ for all $g$, then it is easy to see that all cosets of $\bigoplus_{i < 0} \Z_2$ are constant (these are indicated in Figure~\ref{fig:Lamplighter2}, and geometrically mean moving to ``horizontal neighbors''). However, this is no longer true if we use a noncommutative group structure, let alone in a general TEP subshift. \qee
\end{example}

\begin{figure}[ht]
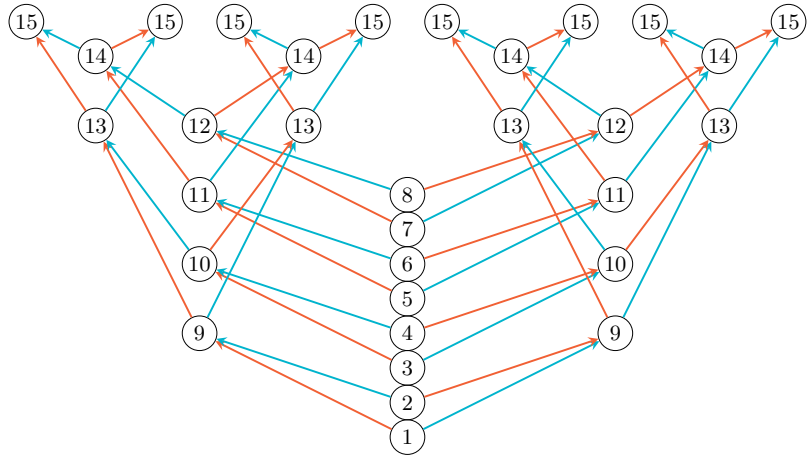

    \centering
    \includestandalone[height = 6cm]{Figures/General/lamp_tetrahedron}
    \caption{A tetrahedron in the lamplighter group, with nodes on the same ``row'' (coset of $\bigoplus_{i < 0} \Z_2$) labelled the same.}
\label{fig:Lamplighter2}
\end{figure}

\subsection{Computing translations between patterns memorilessly}

\begin{definition}
Let $N = \{0, \ldots, n-1\}$ and let $A$ be an alphabet. A \emph{simple $k$-permutation} of $A^N$ is a permutation $\pi$ that only reads and modifies at most $k$ elements of $N$. More precisely, it is one for which there exists $\pi' : A^M \to A^M$ with $M \subset N$ and $|M| \leqslant k$, such that
\[ \pi(x)_i = \begin{cases}
x_i & \mbox{ if } i \notin M \\
\pi'(x|_M)_i & \mbox{ if } i \in M.
\end{cases}
\]
If $|A| \geqslant 3$, then a \emph{simple permutation} refers to a simple $2$-permutation. If $|A| = 2$, then it refers to a simple $3$-permutation. 
\end{definition}

The reason that in the case $|A| = 2$ we use simple $3$-permutations is that in this case simple $2$-permutations are affine for the natural vector space structure, and thus do not generate all permutations.

If $X$ is a TEP-subshift and $P, Q$ are bases for the same set, then there is a natural bijection $\pi : X|_P \to X|_Q$, namely the one with graph $\{(x|_P, x|_Q) \;| \; x \in X\}$.

\begin{lemma}
\label{lem:SimplePermutations}
Let $P \rightarrow^n Q$, then for any identifications of $P$ and $Q$ with $N = \{0, \ldots, |P|-1\}$, the permutation of $A^N$ corresponding to the natural bijection $\pi : X|_P \to X|_Q$ can be computed with $O(n + |P| \log |P|)$ many simple permutations.
\end{lemma}

\begin{proof}
We first show that this is true for simple $|S|$-permutations.

We start by showing that this is true for some identifications between $P$ and $N$, and between $Q$ and $N$. Let $P_0 = P \rightarrow P_1 \rightarrow \cdots \rightarrow P_n = Q$ be a sequence of solitaire moves.

Start with any identification between $P_0$ and $N$, and when applying a solitaire move, copy the identification for all cells except the ones in the symmetric difference. When we use this identification of $P_{i+1}$ and $N$, it is clear that the operation of determining the contents of $P_{i+1}$ from those of $P_i$ is a $k$-simple permutation.

Finally, for any other identification between $Q$ and $N$, it suffices to sort this final result, for which $|Q| \log |Q|$ simple $2$-permutations suffice.

To show that simple $|S|$-permutations can be turned into simple permutations, one can apply Lemma 3.2 of \cite{Sa22a}. \qed
\end{proof}

\subsection{Bipermutive cellular automata and the Ledrappier subshift}
\label{sec:Ledrappier}

A standard source of TEP subshifts with the triangle shape are bipermutive cellular automata. Recall that a \emph{cellular automaton} (here, one-dimensional) is a function $f : A^\Z \to A^\Z$ (for finite alphabet $A$) that is continuous and shift-commuting.

A cellular automaton always has a \emph{local rule} $F : A^{\{\ell,\ldots, r\}} \to A$ such that $f(x)_i = F(x|_{\{i+\ell, \ldots, i + r\}})$ for all $x \in A^\Z, i \in \Z$. We say a cellular automaton is \emph{bipermutive} if, taking $\ell$ maximal and $r$ minimal, we have $\ell < r$ and $F$ satisfies that $F(au) \neq F(bu)$ and $F(ua) \neq F(ub)$ whenever $u \in A^{r-\ell}$ and $a, b \in A$ are distinct letters. (Here, we identify words and patterns with an interval shape in the natural way.)

\begin{lemma}
Let $f : A^\Z \to A^\Z$ be a bipermutive cellular automaton and let $F : A^{\{\ell,\ldots, r\}} \to A$ be a local rule for it with $\ell$ maximal and $r$ minimal. Then
\[ X = \{x \in A^{\Z^2} \;|\; x_{(a, b+1)} = F(x_{(a+\ell, b)}, \ldots x_{(a+r, b)}) \} \]
is a $(C, S)$-TEP subshift, where
\[ S = (\{\ell, \ldots, r\} \times \{0\}) \cup \{(0,1)\} \]
and 
\[ C = \{(\ell,0), (r, 0), (0,1)\}. \]
\end{lemma}

\begin{proof}
    This result follows from the observation that bipermutivity implies that $a$ is determined by the pair $(u, F(ua))$ and by the pair $(u, F(bu))$; and that for each $u \in A^{r-\ell}$, $F(ua)$ and $F(au)$ can take any values (since an injective function between finite sets of equal cardinality is surjective).
\end{proof}

\begin{remark}
\label{rem:HorizontalSurjective}
The subshift $X$ can be defined for any cellular automaton, and is called the \emph{spacetime subshift}. It follows from the theory in \cite{Sa22} that every element of $A^\Z$ appears on rows of this subshift in the TEP case. More generally, this holds for surjective cellular automata. For general cellular automata, the set of configurations appearing on the rows is precisely the limit set of the cellular automaton.
\end{remark}

The theory developed in Section~\ref{sec:Triangle} applies to all bipermutive cellular automata with $r - \ell = 1$, as in these cases (up to shearing) the shape $S$ above is precisely the triangle, and $C = S$.

A TEP subshift of special interest, and having the triangle shape, is the \emph{Ledrappier subshift}, which is the spacetime subshift of the XOR cellular automaton on alphabet $A = \{0,1\}$, defined by $f(x)_i = x_i \oplus x_{i+1}$ where $\oplus$ is addition modulo $2$. Equivalently, this is the TEP subshift where the allowed $T$-patterns are $p \in A^T$ such that $\sum p = 0$ where addition is again modulo $2$. 

This example shows that spanned sets can be larger than filled sets.

\begin{proposition}
There exists a triangular TEP subshift $X$ such that we can find $n$, and $P \subset T_n$ such that $\psi_X(P) = T_n$ but $|T_n| - |\varphi(P)|$ is arbitrarily large. Indeed, the Ledrappier subshift is such.
\end{proposition}

\begin{proof}
Observe that the restriction of the Ledrappier subshift to the subgroup $2\Z^2$ of index $4$ is $2T$-permutive. This can be verified by a direct computation ($x_{0,2} = x_{0,1}\oplus x_{1,1} = x_{0,0} \oplus x_{0,1} \oplus x_{0,1}\oplus x_{0,2} = x_{0,0} \oplus x_{0,2}$ and symmetries). (Abstractly, we may deduce it from the fact the subshift is algebraic: we can see configurations as infinite series in polynomials in two commuting variables with coefficients in $\Z_2$, and $T$ corresponds to an annihilating polynomial $p$. Then $p(x, y)^2 = p(x^2, y^2)$ because the field has characteristic $2$, see e.g.\ \cite{KaSz20}.)

It follows that $P = L_{2n} \cup \{(2n+1, 0), (0,2n+1)\}$ spans $T_{2n+2}$ : $L_{2n}$ spans $T_{2n}$, then one can add the diagonal edge of $T_{2n+2}$ using $2T$ and the last diagonal can then be added with regular filling steps.

On the other hand, for sufficiently large $n$, the filling closure of $P$ is $T_{2n} \cup \{(2n+1, 0), (0,2n+1)\}$, as this set is closed under $T$-filling. The difference in cardinality can be arbitrarily large.
\qed\end{proof}

\begin{figure}[ht]
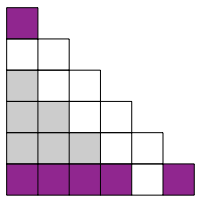

    \centering
    \includestandalone[height=2.5cm]{Figures/TEP/span_not_fill}
    \caption{A $T_6$-spanning pattern that does not fill $T_6$. The filling of the pattern is shown in grey.}
    \label{fig:span not fill}
\end{figure}

\begin{remark}
    Notice that the spanning sets depend on the TEP interpretation and not only on the shape $S$. For example, keeping the triangle shape but taking the $S_3$-TEP subshift shown in Figure~\ref{fig:S3triangle}, one can no longer deduce $x_{0,2}$ from $x_{0,0}$ and $x_{2,0}$ since we no longer have $x_{1,0} \circ x_{1,0} = x_{1,0}$.
\end{remark}

\subsection{Filling bases for triangular TEP subshifts}

We continue with TEP subshifts with the triangle shape, and show that the filling bases of a triangle (with respect to a TEP subshift with the triangle shape) correspond to the line orbit.


Let $X_R$ be a TEP subshift for some $R \subset A^T$ and fix some $n \in \N$. Recall from \cite{Sa22} that the pattern $q \in A^{T_n}$ appears in $X_R$ (i.e.\ is equal to the restriction of some $x \in X_R$ to $T_n$) if and only if it does not explicitly contain a translate of a pattern in $A^T \setminus R$.


\begin{theorem}
\label{thm:Bases}
The following are equivalent for $P \subset T_n$, and a TEP subshift $X \subset A^{\Z^2}$ with the triangle shape:
\begin{enumerate}
    \item $P$ is a filling basis (with respect to $X$),
    \item $P$ fills (i.e.\ has filling closure $T_n$ and no excess),
    \item $P$ is in the line orbit.
\end{enumerate}
\end{theorem}

\begin{proof}
The equivalence of the last two items is a special case of Corollary~\ref{cor:line_orbit}. We show that (1) implies (2) and (3) implies (1). 

 Suppose thus first that $P$ is a filling basis. Then by definition $\varphi(P) = T_n$. Since the line is a filling basis, it is $T_n$-spanning, and thus $X|_{T_n}$ cannot have cardinality larger than $|A|^n$. Thus we cannot have more than $n$ elements in $P$ since $P$ is independent. 
 Now, we know that any filling set with at most $n$ element is in the line orbit.

The line is an independent set by Remark~\ref{rem:HorizontalSurjective}, and of course fills, thus it is a filling basis. The property of being a filling basis is preserved under solitaire (Lemma~\ref{lem:BasesPreserved}), thus every pattern in the line orbit is a filling basis.
\qed\end{proof}

Note that the first item talks about a specific (but arbitrary) TEP subshift, since in the definition of the filling basis we require that $X|_T = A^T$ for the specific subshift. The other two items only talk about the abstract solitaire process, i.e.\ the set of bases is independent of $A$ and $R$.

The solitaire process allows us to translate patterns on one basis to ones on another, more space efficiently than the direct method suggests, indeed we give a polynomial time in-place algorithm for this.

We now specialize Lemma~\ref{lem:SimplePermutations} to the triangle case. If $P, Q \subset T_n$ are bases, any pattern $p \in A^P$ uniquely determines a pattern in $q \in A^Q$ in the natural way, by deducing the unique extension of one pattern to $T_n$ and then restricting to the domain of the other. If we biject $P$ and $Q$ with $N = \{1, \ldots, n\}$ we obtain a bijection $\varphi : A^n \to A^n$.

Recall that when $|A| \geqslant 3$, a simple permutation of $A^n$ is one that ignores all but two cells. If the cells are $1, 2$, this means that for some $\hat\pi \in \mathrm{Sym}(A^2)$ we have
$\pi(a_1 a_2 a_3 a_4 \cdots a_n)_i = \hat\pi(a_1 a_2) a_3 a_4 \cdots a_n$
for all $a_1a_2\cdots a_n \in A^n$. In general, one conjugates by a reordering the cells. If $|A| = 2$, a simple permutation may look at three cells.

\begin{theorem}
\label{thm:Memoryless}
The bijection $\varphi$ can be computed with $O(n^3)$ simple permutations.
\end{theorem}

\begin{proof}
This is a special case of Lemma~\ref{lem:SimplePermutations}, observing that $n \log n \leqslant n^3$.
\qed\end{proof}

\section{Prospects for future work}
\label{sec:future}

\subsection{On the triangle solitaire}

While we know the diameter of the line orbit (for the triangle shape), and some bounds on its cardinality, we do not have a good understanding of its shape.

\begin{problem}
Describe the shape of the line orbit as a graph, with solitaire moves as edges.
\end{problem}

\begin{question}
What does a random element of the line orbit look like? What is its expected distance from the line? Can the uniform distribution on the line orbit be sampled efficiently?
\end{question}

Although we have no particularly strong evidence one way or another, we conjecture a nice formula for the size of the line orbit:

\begin{conjecture}
    There are constants $\frac2e \leqslant c \leqslant e$ and $d$ such that $|\gamma(L_n)| = \Theta \left(\left(\frac{n}c\right)^{n+d}\right)$.
\end{conjecture}

\subsection{On general groups and shapes}

\begin{question}
For which groups $G$, and subsets $C \Subset S \Subset G$, do we have a nice solitaire theory?
\end{question}

\begin{question}
Does $G = \Z^2$ (or more generally $\Z^d$) have nice solitaire theory for all $C = S$?
\end{question}

An interesting class of shapes are the simplices $\{\vec v \geq 0 \;|\; \sum \vec v_i \leq 1\}$ in $\Z^d$. We do not even know whether $G = \Z^3$ has a nice solitaire theory for the pyramid $\{(0,0,0), (1,0,0), (0,1,0), (0,0,1)\}$.

\begin{question}
Which groups $G$ have the finite filling (resp.\ solitaire) property?
\end{question}

\subsection{Other questions}

We introduced in Section~\ref{sec:Excess} the notion of an excess set. What can be said about the family of excess sets as a set system? How can we determine the maximum cardinality of an excess set? We do not know the answers even for the triangle case.

We also mention the following problem:

\begin{problem}
Describe the bases of triangles for the Ledrappier subshift, in the sense of Section~\ref{sec:Bases}. Can they be described by a variant of the solitaire process?
\end{problem}

\section*{Acknowledgements}

The authors thank their reviewers for their comments that allowed to improved greatly the article.

 \bibliographystyle{splncs04}
 \bibliography{biblio}

\end{document}